\documentclass[reqno,11pt]{amsproc}
\usepackage[margin=1.5in]{geometry}
\usepackage{lipsum} 
\usepackage{amssymb,amsmath,amsthm,amsopn, leftidx} 
\usepackage{amscd} 
\usepackage{latexsym,graphicx}
\usepackage{xspace}
\usepackage{pinlabel}
\usepackage{mathrsfs}
\usepackage[usenames,dvipsnames]{color}
\usepackage[implicit=true,colorlinks=true,linkcolor=Black,citecolor=Black,urlcolor=Black]{hyperref}

{\begin{list}{}{%
   \vspace{-.8cm}
   \setlength{\rightmargin}{7mm}
   \setlength{\leftmargin}{7mm}}
    \item[]}
{\end{list}}



\refstepcounter{equation}
\newtheorem{lem}{\bf Lemma}
\newtheorem{theo}[lem]{\bf Theorem}

\newtheorem{coro}[lem]{\bf Corollary}
\newtheorem{prop}[lem]{\bf Proposition}

\theoremstyle{definition}
\newtheorem{definition}[lem]{\bf Definition}
\newtheorem{rem}[lem]{\bf Remark}

\newtheorem{ex}[lem]{\bf Example}
\refstepcounter{lem}

\renewcommand{\descriptionlabel}[1]%
     {\hspace{\labelsep}\textsf{#1}}

\newtheorem*{lemma*}{\bf Lemma}
\newtheorem*{theorem*}{\bf Theorem}

\theoremstyle{remark}

\newcommand\F{{\mathbb F}}
\newcommand\Fhat{{\widehat{\mathbb F}}}

\def\(({{(\!(}}
\def\)){{)\!)}}
\def\vep{{\varepsilon}}
\def\ha{{\,\sf Haus}}
\def\d{{\bf d}}
\def\rs{{\sf rs}}
\def\cs{{\sf cs}}
\def\L{{\mathcal L}}
\def\Mbar{{\overline{\mathcal M}}}

\newcommand\bF{{\bf F}} 
\newcommand\C{{\mathbb C}}
\newcommand\wC{\widehat{\mathbb C}}
\newcommand\wR{\widehat{\mathbb R}}
\def\U{{\mathfrak U}}

\newcommand\Z{{\mathbb Z}}
\newcommand\R{{\mathbb R}}

\newcommand\bP{{\mathbb P}}

\newcommand\PGL{{\rm PGL}}

\newcommand\cD{{\mathcal D}}

\newcommand\cC{{\mathscr C}}

\newcommand\ssm{{\smallsetminus}}

 at 11truept

\def\bK{{\bf K}}
\def\PGL{{\rm PGL}}
\def\GL{{\rm GL}}\def\x{{\bf x}}
\def\HG{{H_1/N}}
\def\fg{{\mathfrak g}}
\def\cB{{\mathcal B}}
\def\cC{{\mathcal C}}
\def\cS{{\mathcal S}}
\def\bT{{\bf T}}
\def\bP{{\mathbb P}}
\def\Q{{\mathbb Q}}
\def\cH{{\mathcal H}}
\def\g{{\bf g}}
\def\p{{\bf p}}
\def\q{{\bf q}}
\def\r{{\bf r}}

\def\msk{\medskip}
\def\bsk{\bigskip}
\def\ssk{\smallskip}

\def \M {{\mathbb M}}
\def\cM{{\mathcal M}}

\def \bX {{\bf X}}
 
\def\floor{{\bf floor}}

\def\xx{{\bf x}}
\def\yy{{\bf y}}
\def\zz{{\bf z}}
\def\g{{\bf g}}
\def\h{{\bf h}}
\def\v{{\bf v}}

\def\cA{{\mathcal A}}
\def\bpi{{\boldsymbol\pi}}
\def\bbet{{\boldsymbol\beta}}
\def\bgam{{\boldsymbol\gamma}}
\def\bdel{{\boldsymbol\delta}}
\def\bep{{\boldsymbol\varepsilon}}
\def\vep{{\varepsilon}}
\def\bzet{{\boldsymbol\zeta}}
\def\balph{{\boldsymbol\alpha}}

\def\brh{{\boldsymbol\rho}}
\def\bmu{{\boldsymbol\mu}}

\def\bph{{\boldsymbol\varphi}}

\def\sm{{\,\sf sm}}
\def\bY{{\bf Y}}

\def\fs{{\,\sf fs}}
 
\def\eqv{{\sim}}
\def\fC{{\mathfrak C}}
\def\wfC{{\widehat{\mathfrak C}}}
\def\dim{{\rm dim}}
\def\irr{{\,\sf irr}}
\def\fW{{\mathfrak W}} 

\def\fD{{\mathfrak D}}
\def\fDno{{{\mathfrak D}_n^{\sf ord}}}
\def\wfD{{\widehat{\mathfrak D}}}
\def\fM{{\mathfrak M}}
\def\bJ{{\bf J}}
\def\fS{{\mathfrak S}}
\def\bG{{\bf G}}
\def\bE{{\bf E}}

\def\bH{{\bf H}}
\def\fV{{\mathfrak V}}
\def\<{{\langle}}
\def\>{{\rangle}}
\def\qu{{\setminus}}

\def\gr{{\g_{\textstyle r}}}
\def\fN{{\mathfrak N}}
\def\vp{{\stackrel{\to}{\p}}}
\def\bS{{\bf S}}
 \def\Pdual{{\bP^{2\textstyle{*}}}}
 \def\cR{{\mathcal R}}
 
\usepackage[all,cmtip]{xy}

\begin{document}

\setcounter{footnote}{1}
\setcounter{equation}{0}

\title[Group Actions, Divisors and Plane Curves]
{Group Actions, Divisors, and Plane Curves}

\author[Araceli Bonifant and John Milnor]
{Araceli Bonifant$^1$ and John Milnor$^2$}

\footnotetext[1]{Mathematics Department, University of Rhode Island; e-mail:\hfill{\mbox{}}\\ 
\hbox{bonifant@uri.edu}}
\footnotetext[2]{Institute for Mathematical Sciences, Stony Brook University;
e-mail: \hfill{\mbox{}}\\\hbox{jack@math.stonybrook.edu }}

\date{}

\keywords{effective 1-cycles, moduli space of curves, smooth complex curves, 
stabilizer  of curves,  algebraic set, group actions,  proper action, 
improper action,  orbifolds,  weakly locally proper, effective action, 
rational homology  manifold, tree-of-spheres, W-curves, moduli space of 
divisors, Deligne-Mumford compactification.}

\begin{abstract}
  After a general discussion of group actions, 
orbifolds, and ``weak orbifolds''  
this note will provide elementary introductions to two basic moduli spaces
over the real or complex numbers: First the moduli space of effective 
divisors with finite stabilizer on the projective space $\bP^1$
modulo the group $\PGL_2$ of projective transformations of $\bP^1$;
and then the moduli space of effective 1-cycles with finite stabilizer 
on $\bP^2$ modulo the group $\PGL_3$ of projective transformations of $\bP^2$.
\end{abstract}

\maketitle

\tableofcontents

\setcounter{lem}{0}
\setcounter{footnote}{2}

\section{Introduction.}\label{s-intro}
Section \ref{s-orb} of this paper will be a general discussion of 
group actions, and the associated quotient spaces. If the action is proper,
then the quotient space is an  orbifold; but we also introduce a notion
of ``weakly proper'' action, yielding a ``weak orbifold''.
The subsequent sections consist of detailed studies of two particular families
of examples.

Section \ref{s-div} describes the moduli space $\fM_n(\F)$ for divisors;
where the symbol $\F$ stands for either the real numbers $\R$ or the complex
numbers $\C$. By definition, an \textbf{\textit{effective divisor}} of 
degree $n$ over $\F$ is a formal sum of the form 
$$ \cD~=~ m_1\<\p_1\>+\cdots+m_k\<\p_k\>~,$$
where the $m_j$ are strictly positive integer coefficients,  where the 
$\p_j$ are distinct points of the projective line $\bP^1(\F)$, and where
$\sum m_j=n$. Each element $\g$ of the group $\bG=\PGL_2(\F)$ of projective 
transformations of $\bP^1(\F)$ acts on the space
$\wfD_n(\F)$ of all such divisors. 
The \textbf{\textit{stabilizer}} of such a divisor
$\cD$ is the subgroup $\bG_\cD$ consisting of all $\g\in\bG$ 
with $\g(\cD)=\cD$. By definition, the moduli space $\fM_n(\F)$ is the 
quotient space $\wfD^\fs_n(\F)/\bG$, where $\wfD^\fs_n$ is the open 
subset consisting of all $\cD\in\wfD_n$ with finite stabilizer. 
The moduli space $\fM_3(\R)\cong\fM_3(\C)$ is a single point, while
$\fM_4(\C)$ is naturally isomorphic to the 
projective line $\bP^1(\C)$ with one ``improper point'', corresponding to the
case where two of the four points crash together. It also has two ramified
(or orbifold) points, corresponding to curves 
with extra symmetries. The space $\fM_4(\R)$ can be identified with a 
line segment in $\fM_4(\C)$ which joins
one of the two ramified points to the improper point (Figure \ref{F-M4}).
For $n>4$, the moduli space $\fM_n(\F)=\wfD_n^\fs/\PGL_2$
is not a Hausdorff space. However, it contains a unique maximal 
Hausdorff subspace, which is compact when $n$ is odd.
The section concludes with a comparison of $\fM_n$ with the 
 moduli spaces $\cM_{0,n}$ and $\cM^{\sf un}_{0,n}$ for ordered or 
 unordered $n$-tuples of distinct points in $\bP^1$, and with their
 compactifications.
\ssk

Section \ref{s-mod} begins the study of the moduli space $\M_n(\F)$
 for curves (or more
generally \hbox{1-cycles)} of degree $n$ in the projective plane $\bP^2(\F)$.
Here and in later sections $\bG$ will be the automorphism group
of $\bP^2$, so that $\bG=\PGL_3(\R)$ in the real case, and  $\bG=\PGL_3(\C)$
in the complex case. 

Section \ref{s-deg3} is a detailed study of the degree three case, it shows
that $\M_3(\F)$ has a natural analytic structure. In the complex case
$\M_3(\C)$ is isomorphic to the Riemann sphere $\bP^1(\F)$; but with one
``improper point'',
 corresponding to curves with a simple double point, and also with two ramified
(or orbifold) points corresponding to curves with extra symmetries. 
(In fact $\M_3(\C)$ is naturally isomorphic to the moduli space
$\fM_4(\C)$ for divisors.) In the real case, the space $\M_3(\R)$ forms a 
circle of curve-classes. There is still one improper point (corresponding to
a curve with a simple self-crossing), but there is no ramification.\ssk

Section \ref{s-cc} begins the study of $\M_n=\M_n(\C)$ for $n>3$. 
Although $\M_n$ is non-Hausdorff for $n\ge 7$ (and possibly also for some
smaller $n$),  it does have large open 
subsets which are Hausdorff. In fact, we provide two different ways of
proving that a suitable region is Hausdorff. This section describes
one method, which makes use of
``virtual flex points''.  This section also compares $\M_n$ with the 
classical moduli space $\cM_\fg$, consisting of conformal isomorphism 
classes of closed Riemann surfaces of genus $\fg$, where $\fg={n-1\choose 2}$.

Section \ref{s-G} describes another method of proving that suitable subsets of
 $\M_n$ are Hausdorff, making use of the genus invariant associated with a 
 singular point. Although our methods are actually
 applied only to curves,  
 there is a concluding attempt to adapt them to the more general case of
 1-cycles.
\ssk

Section \ref{s-aut} studies the stabilizers $\bG_\cC$  
for cycles $\cC\in\wfC_n(\C)$. In particular, it studies the algebraic set
$\fW_n\subset\wfC_n$  consisting of cycles with infinite stabilizer. It
also contains remarks about which finite stabilizers are possible, and some
comparison with automorphism groups of more general Riemann surfaces.

Section \ref{s-genR}  discusses the real moduli space $\M_n(\R)$,
concentrating on the Harnack-Hilbert classification problem for smooth
real curves. In conclusion, there is an appendix discussing the 
relevant literature. 
\smallskip

\setcounter{lem}{0}
\section{Proper and Improper Group Actions: 
The Quotient Space.}\label{s-orb}
This will be a general exposition of quotient spaces under a smooth 
group action. In the best case, with a proper action, the quotient space is 
Hausdorff, with an orbifold structure. 
Since the group actions that we consider are
 not always proper, we also introduce a modified requirement of
``weakly proper'' action, which suffices to prove that the quotient
is locally Hausdorff, with a ``weak orbifold structure'' which includes
only some of the usual orbifold properties.\msk

First consider the complex case. 
Let $\bX$ be a complex manifold and $\bG$ a
complex Lie group\footnote{Although our Lie groups are always positive
dimensional, the discussion would apply equally well to the case 
of a discrete group, which we might think of as a zero-dimensional Lie group.}
 which acts on the left by a holomorphic map $\bG\times \bX\to \bX$,
$$ (\g,~\xx)~\mapsto~ \g(\xx)~,$$
where $~\g_1\big(\g_2(\xx)\big)=(\g_1\g_2)(\xx)$. We will always assume that the
action is \textbf{\textit{effective}} in the sense that 
$$ \g(\xx)=\xx \quad{\rm for~all~}~~\xx\quad\Longleftrightarrow\quad
\g\quad{\rm is~ the~ identity~ element} ~~{\bf e}\in \bG\,.$$
The \textbf{\textit{quotient space}} (or orbit space)
 in which $\xx$ is identified with $\xx{\bf '}$ if and only if
$\xx{\bf '}=\g(\xx)$ for some $\g$ will be denoted\,\footnote{Since $\bG$ acts
 on the left, many authors would use the notation $\bG\qu \bX$.} by $\bX/\bG$.
\smallskip
  
In the real case, the definitions are completely analogous, although
we could equally well work 
either in the $C^\infty$ category or in the real analytic category.
To fix ideas, let us choose the real analytic category.
Thus in the real case, we will  assume that $\bG$ is a real Lie group, that
 $\bX$ is a real analytic manifold,
 and that $\bG\times \bX\to \bX$ is a real analytic map. It will often be 
convenient to use the word ``analytic'', by itself, to mean 
real analytic in the real case, or complex analytic in the complex case.
\medskip

\begin{definition}
For each $\xx\in \bX$ the set $\((\xx\))$ consisting of all images $\g(\xx)$
with $\g\in \bG$ is called the \textbf{\textit{$\bG$-orbit}} of $\xx$. 
We will also use the notation $$ \((\xx\))~=~\bF~=~\bF_\xx $$
if we are thinking of  $\((\xx\))$ as a  "\textbf{\textit{fiber}}"
of the  \textbf{\textit{projection map}} ${\boldsymbol\pi}:\bX\to \bY=\bX/\bG$.
\end{definition}
\medskip

\begin{rem}\label{R-qtop} In other words, each fiber is an equivalence
class, where two points of $\bX$ are equivalent if and only if they belong
to the same orbit under the action of $\bG$. More generally, given any
equivalence relation $\eqv$ on $\bX$, we can form the quotient space  
$\bY=\bX/\!\eqv$. Such a quotient space $\bY$ always has a well defined 
\textbf{\textit{quotient topology}}, defined by the condition that a set 
${\bf U}\subset \bY$ is open if and only if the preimage 
$\bpi^{-1}({\bf U})$ is open as a subset of $\bX$. 
\end{rem}
\medskip

\begin{rem}[{\bf Closed Orbits and the ${\rm T}_1$ Condition}]\label{R-T1} 
By definition, a topological space $\bY$ is a 
${\rm T}_1$-\textbf{\textit{space}} if every point of $\bY$ is closed as a
subset of $\bY$. Evidently a quotient space $\bX/\bG$ (or more generally
$\bX/\!\eqv$) is a ${\rm T}_1$-space if and only if each orbit (or each 
equivalence class) is closed as a subset of $\bX$.
\end{rem}

\subsection*{\bf Orbifolds and Weak Orbifolds}\bsk

Let $\F$ stand for either the real or the complex numbers.

\begin{definition} \label{D-orbi} By a  $d$-dimensional 
  $\F$-\textbf{\textit{orbifold chart}}  around a point $\yy$ of a topological
  space $\bY$ will be meant the following: 

  \begin{quote}
    \begin{enumerate}
    \item[{\bf(1)}] a finite group  $\cR\subset\GL_d(\F)$
acting linearly on   $\F^d$;

\item[{\bf(2)}] an $\cR$-invariant open neighborhood $W$ of the point
 ${\bf 0}\in\F^d$; and

\item[{\bf(3)}] a homeomorphism $h$ from the quotient space $W/\cR$ onto an open
neighborhood $U$ of $\yy$ in $\bY$ such that
the zero vector in $W$ maps to $\yy$.
\end{enumerate}
\end{quote}

\noindent The group $\cR$ will be called the
\textbf{\textit{ramification group}}
at $\yy$, and its order $r_\yy\ge 1$ will be called the \textbf{\textit
  {ramification index}}. A point $\yy'$ is \textbf{\textit{ramified}}
if $r_{\yy'}>1$ and \textbf{\textit{unramified}} if $r_{\yy'}=1$.\ssk

 The space $\bY$ together with an integer valued function
$\yy\mapsto r_\yy\ge 1$ will be called
a $d$-dimensional \textbf{\textit{weak orbifold}} over $\F$ if there exists
such an orbifold chart $W/\cR\cong U$  around every point $\yy\in\bY$,
such that the
associated ramification function from $U$ to the positive integers
coincides with the specified function $\yy'\mapsto r_{\yy'}$ throughout $U$.
 \end{definition}

{\bf Example.} On any Riemann surface, we can choose any
function $\yy\mapsto r_\yy\ge 1$ which takes the value $r_\yy=1$ except
at finitely many points.  A corresponding collection of orbifold charts
is easily constructed.\ssk
 
\begin{lem}\label{L-usc}  Every weak orbifold is 
  a locally Hausdorff space; and the function  $\yy\mapsto r_\yy\ge 1$
from $\bY$ to the set of positive integers is always upper semicontinuos, taking
 the value $r_\yy=1$ on a dense open set. More precisely,
for any orbifold chart
  $W/\cR\stackrel{\cong}{\longrightarrow} U=U_\yy$,
  we have $r_{\yy'}\le r_\yy$ for every
  $\yy'\in U$, with $r_{\yy'}=1$ on a dense open subset of $U$. 
\end{lem}
\ssk

Of course in good cases $\bY$ will be a Hausdorff space; but even a 
locally Hausdorff space can be quite useful.\footnote{Perhaps the
most startling application of locally Hausdorff spaces in science would be to
the ``Many Worlds''  interpretation of quantum mechanics, in which 
the space-time universe  continually splits into two or more alternate 
universes.  (See for example \cite{Bec}.)
The  resulting object is possibly best described as a space which is locally
Hausdorff, but wildly non-Hausdorff. It can be constructed mathematically
 out of infinitely  many copies of the Minkowski space 
$\R^{3,1}$ by gluing together corresponding open subsets. 
 (Of course it does 
not make any objective sense to ask whether these alternate universes
 ``really exist''. The only legitimate question is whether a mathematical model
including such alternate universes can provide a convenient and testable
model for the observable universe.)} \ssk

\begin{proof}[Proof of Lemma \ref{L-usc}] 
  For  any chart $W/\cR\stackrel{\cong}{\longrightarrow} U$ around $\yy$, and
any non-identity element
$\g\in\cR$, the set of elements of  $\F^d$ fixed by $\g$ must be a linear
subspace of $\F^d$ of dimension at most  $d-1$. If $W_0$ is
the complement of this finite union of linear subspaces intersected with $W,$ and
$U_0$ is its image, then the associated  mapping $W_0\to U_0$ is precisely
$r_\yy$-to-one. In a small neighborhood of any point of $W_0$, it follows
that this map is one-to-one.
Each $\cR$-orbit in $U$ is compact and non-empty, 
so we can use the Hausdorff metric for compact subsets of  $\F^d$ to make 
$U\cong W/\cR$ into a metric space. In particular, it follows that
$\bY$ is locally Hausdorff.
  \end{proof}\msk
  
  An orbifold chart $W/\cR\cong U$ around $\yy$ gives rise to a smaller
  orbifold chart
around any point $\zz$ of the neighborhood $U$. In fact, choosing a
representative point $w_0\in W$ over $\zz$, let $\cR_{w_0}$ be the stabilizer,
consisting of all $\g\in\cR$ for which $\g(w_0)=w_0$. Evidently $\cR_{w_0}$
acts linearly on $W$, fixing the point $w_0$.
Note that the $\cR$-orbit of $w_0$ contains $r_\yy/r_{\zz}$
distinct points. Choose a $\cR_{w_0}$-invariant neighborhood $W'$ of $w_0$
which is small enough so that its $r_\yy/r_{\zz}$ images under the action of
$\cR$ are all disjoint. Then the projection  from $W'$ to
$W'/\cR_{w_0}\subset U$ is the required \textbf{\textit{restriction}} to an
orbifold chart around $\zz$.
\msk

\begin{definition}
An \textbf{\textit{orbifold atlas}} on $Y$ is a collection of orbifold charts
$$W_j/\cR_j\stackrel{\cong}{\longrightarrow} U_j\subset Y ~, $$
where the $U_j$ are open sets covering $Y$,
which satisfy the following compatibility condition.
\begin{quote} For each point $\yy$ in an overlap $U_i\cap U_j$
and each sufficiently small neighborhood $U'$ of $\yy$, the restriction
of the $i$-th and $j$-th orbifold charts to $U'$ are \textbf{\textit{
isomorphic}} in the following sense. Let
$$ W'_i/\cR_{i,w_i}\stackrel{\cong}{\longrightarrow} U'\qquad{\rm and}
\qquad W'_j/\cR_{j,w_j}\stackrel{\cong}{\longrightarrow} U' $$
be the two restrictions. Then we require that there should be a group
isomorphism $~\phi:\cR_{i, w_i}\stackrel{\cong}{\longrightarrow}\cR_{j,w_j}~$
and an  analytic
isomorphism $~\psi:W'_i\stackrel{\cong}{\longrightarrow} W'_j~$ so
that the following diagram is commutative
for every $\g\in\cR_{i,w_i}$. 
$$\xymatrix{W'_i~ \ar[r]^{\psi} \ar[d]^\g & W'_j\ar[d]^{\phi(\g)}\\
  W'_i  \ar[r]^{\psi} & W'_j}$$
\end{quote}
Two such atlases are \textbf{\textit{equivalent}} if their union also
satisfies this compatibility condition. The space $\bY$ together with
an equivalence class of such atlases is called an
\textbf{\textit{orbifold}}.\footnote{Caution. Most authors require
  orbifolds to be Hausdorff spaces; but we will allow orbifolds which are only
  locally Hausdorff.}
(Compare \cite{Th}, \cite{BMP}.) 
\end{definition}
\ssk

The main object of this section will be to describe conditions on the
group action which guarantee that the quotient will be an orbifold or weak
orbifold.
\msk

\subsection*{Proper and Weakly Proper Actions.}

\begin{definition}\label{D-prop} A  continuous action of $\bG$ on a
 locally compact space $\bX$ is 
\textbf{\textit{proper}} if the following condition is satisfied: 
\begin{quote} For every pair of points $\xx$
and $\xx{\bf '}$ in $\bX$,  there exist neighborhoods $U$ of $\xx$
 and $U'$ of $\xx{\bf '}$ which are small enough
  so that the set of \textbf{\textit{all}}  $\g\in \bG$
        with $\g(U)\cap U'\ne \emptyset$ has compact closure.\footnote
{For further discussion, see Remark \ref{R-p}. In the special case of a 
discrete group, such an action is called 
\textbf{\textit{properly discontinuous}}.}
\end{quote}

\noindent
The action  is \textbf{\textit{locally proper}} at $\xx$ if this condition
is satisfied for the special case where $\xx=\xx{\bf '}$; or in other words if
 the action is proper throughout some $\bG$-invariant open neighborhood of 
$\xx$.
\smallskip

It will be called \textbf{\textit{weakly $($locally$)$ proper}}
 at $\xx$ if the following still weaker condition is satisfied. There should be
a neighborhood $U$ of $\xx$ and a compact set 
$K\subset \bG$ such that two points 
$\xx_1{\bf '}$ and $\xx_2{\bf '}$ of $U$ belong to the same $\bG$-orbit if and 
only if $\xx_2{\bf '}=\g(\xx_1{\bf '})$ for \textbf{\textit{at least one}}
 $\g\in K$. 
\end{definition}

\begin{definition}\label{D-stab}
Given an action of $\bG$ on $\bX$, the \textbf{\textit{stabilizer}}
 $\bG_\xx$ of a point $\xx\in \bX$ is the  closed 
subgroup of $\bG$ consisting of all $\g\in \bG$ for which $\g(\xx)=\xx$. 
Note that points on the same fiber have isomorphic stabilizers, since 
$$\bG_{\g(\xx)}~=~\g\, \bG_\xx \g^{-1}~.$$
If the stabilizer $\bG_\xx$ is finite, 
then it follows easily that the fiber $\bF$ 
through $\xx$ (consisting of all images $\g(\xx)$ with $\g\in \bG$)
is a smoothly embedded submanifold which  is locally diffeomorphic to $\bG$.  
\end{definition}

Under the hypothesis that all stabilizers are finite. we will prove
 the following (in Theorem \ref{T-wot} together with Lemma \ref{L-Haus} 
 and  Corollary \ref{C-wot}):

\begin{quote}\it  For a proper action the quotient space is
a Hausdorff orbifold.

\noindent For a locally proper action the quotient
is a locally Hausdorff\break orbifold.\msk

\noindent For a weakly proper action the quotient is a locally
Hausdorff weak orbifold.
\end{quote}

\noindent
(See Figure \ref{F-badpic} for an example of a smooth locally proper action
with trivial stabilizers where the quotient is not a Hausdorff space.)
 \medskip
 
It will be convenient to call a  point in $\bX/\bG$ either 
\textbf{\textit{proper}} or \textbf{\textit{improper}} according as the action
 of $\bG$ on corresponding points of $\bX$ is or is not locally proper.
Similarly an improper point in $\bX/\bG$ will be called 
\textbf{\textit{weakly proper}} if the action of $\bG$
 on corresponding points of $\bX$ is weakly proper. 
\smallskip

\begin{rem}\label{R-weak-but} Although there are examples which are 
  weakly proper but not locally proper, they seem to be hard to find.
Remark~\ref{R-improp} will show that divisors of degree four with only
three distinct points give rise to such examples;  and the proof of 
Lemma \ref{L-M3C} will show that curves of degree three with a simple
double point also provide such examples. However,
these are the only examples we know.
\end{rem}
\smallskip

The following is well known.
\smallskip

\begin{lem}\label{L-Haus}
If the action is proper, then the quotient $\bX/\bG$ is a Hausdorff space.
\end{lem}\smallskip

It follows as an immediate Corollary 
 that a locally proper action yields a quotient space
which is locally Hausdorff. However, such a quotient need not be Hausdorff.
(Compare Figure  \ref{F-badpic}.) If stabilizers are finite,  
then we will see in Theorem \ref{T-wot} that even a weakly proper  
action yields a quotient space which is locally Hausdorff.
\smallskip

\begin{rem}\label{R-T1a}
Note that every locally Hausdorff space is 
${\rm T}_1$. In fact if one point $\p$ belongs to the closure of a 
different point $\q$, then no neighborhood of $\p$ is Hausdorff.
\end{rem}
\smallskip

\begin{proof}[Proof of Lemma~$\ref{L-Haus}$] It will be convenient to
 choose a metric on $\bX$. Given $\xx$ and
$\xx{\bf '}$ there are two possibilities. If we can choose neighborhoods $U$
 and $U'$ so that no translate $\g(U)$ intersects $U'$, then the images
 $\bpi(U)$ and  $\bpi(U')$ in the quotient space are disjoint open sets.

On the other hand, taking $U_j$ and $U'_j$ to be  respectively a sequence
 of neighborhoods of $\xx$ and $\xx{\bf '}$  of radius
 $1/j$, if we can choose group elements $\g_j$ for all $j$ 
with $\g_j(U_j)\cap U'_j\ne\emptyset$, 
then by compactness we can pass to an infinite subsequence so that the $\g_j$ 
converge to a limit $\g$. It follows easily that
 $\g(\xx)=\xx{\bf '}$, so that $\xx$ and $\xx{\bf '}$ map
to the same point in the quotient space.\end{proof}\medskip

\begin{figure}[h!]
\centerline{\includegraphics[width=3 in]{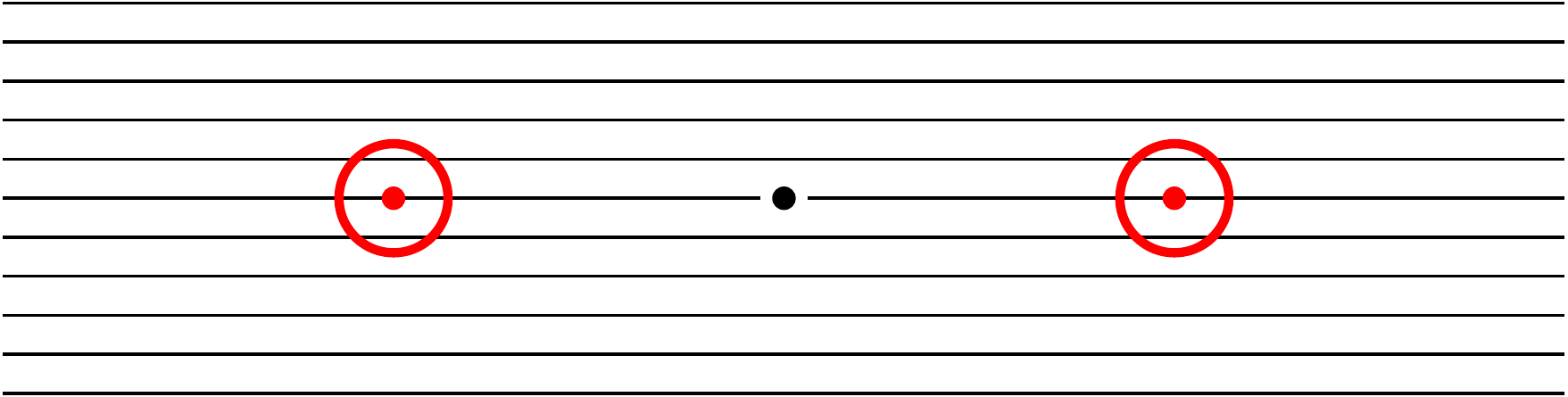}}
\caption{\sf Example:
The additive group of real numbers acts on the punctured 
$(x,y)$-plane $~\R^2\ssm\{(0,0)\}~$ by an action 
$(x,y)\mapsto \g_t(x,y)$ for $t\in\R$ which
satisfies the differential equation \hfill{\mbox{}}
\centerline{$ d\g_t(x,\,y)/dt~=~\big(\sqrt{x^2+y^2}.~0\big),$}
Since $\sqrt{x^2+y^2}$ is strictly positive throughout the punctured plane,
it follows that $\g_t$ moves every point 
to the right for $t>0$; although no orbit can reach the origin.
(Note that $g_t$ acts on the real axis by $\g_t(x,0)=(e^{\pm t}x,\,0)$ where
$\pm t$ stands for $+t$ when $x>0$, but $-t$ when $x<0$.) 
The action is locally proper but not proper; and the 
quotient space is locally Hausdorff but not Hausdorff. In fact, within every 
neighborhood of a point on the negative real axis and every neighborhood
of a point on the positive real axis (as illustrated by circles 
 in the figure), we can choose points which belong to the
same orbit under the action.
\label{F-badpic}}
\end{figure}
\bigskip

\begin{rem}
The converse to Lemma \ref{L-Haus} is false: A quotient space may be Hausdorff 
even when the action is not proper. Compare the discussions of $\fM_4(\C)$
and $\M_4(\R)$ in Remark \ref{R-improp}, as well as
 the quotient spaces $\M_3(\C)$ and $\M_3(\R)$ of Section \ref{s-deg3}. 
Both of these quotients are Hausdorff, even though the associated group
action fails to be proper everywhere.
(See the proof of Lemma~\ref{L-M3C} below.)
 However, for large $n$ we will have to deal with moduli spaces $\M_n(\C)$
and $\M_n(\R)$ which are definitely not Hausdorff. 
 (Proposition \ref{P-nh}.)
\end{rem}
\medskip

The discussion of weakly proper actions will be based on the following.
Given any  fiber $\bF$,  and given any point $\xx\in \bF$, we will refer to
the quotient of tangent vector spaces 
$$ V_\xx~=~T_\xx \bX/T_\xx \bF $$
as the \textbf{\textit{transverse vector space}} to $\bF$
 at $\xx$. 
(If $\bX$ is provided with a Riemannian metric,  
then $V_\xx$ can be identified with the normal vector space at $\xx$.)
 Note that the finite group $\bG_\xx$ acts linearly 
on both $T_\xx \bX$ and $T_\xx \bF$, and hence acts
 linearly on the \hbox{$d$-dimensional} quotient space $V_\xx$, where $d$
is the codimension\footnote{In practice, we will always assume 
that the stabilizer $\bG_\xx$ is finite, so that $d$ is equal to the difference
$\dim(\bX)-\dim(\bG)$.}
of $\bF$ in $\bX$. However, this action is  not always effective. 
(The group $\bG_\xx$ may act non-trivially
on $T_\xx\bF$, while leaving the transverse vector space pointwise fixed.)
In order to describe a weak orbifold structure on the quotient, we must first
construct the associated ramification groups. 
\medskip

\begin{definition} 
Let $\bH_\xx$ be the normal subgroup of $\bG_\xx$ consisting of all
group elements which act as the identity map on $V_\xx$ (that is,
all $\h \in \bG_\xx$ such that $\h(\v)=\v$ for all $\v \in V_\xx$).
The quotient group $$~{\mathcal R}_\xx~=~\bG_\xx/\bH_\xx~$$ will be called the 
\textbf{\textit{\hbox{ramification} group}} at $\xx$. Note that by 
its very \hbox{definition,} ${\mathcal R}_\xx$ comes with a linear 
action on the vector space $V_\xx$, which is isomorphic to $\R^m$ or $\C^m$.
It is not hard to check that different
points on the same fiber $\bF$ have isomorphic ramification groups.
As in Definition \ref{D-orbi},
the order $|{\mathcal R}_\xx|\ge 1$ of this finite group will be called the  
\textbf{\textit{ramification index}} $r=r(\bF)$. 
The fiber $\bF$ (or its image in the quotient space $\bY=\bX/\bG$) will be
 called \textbf{\textit{unramified}} if $r=1$.
\end{definition}
\smallskip

Now let $\bF$ be any fiber with finite 
stabilizers,  and let $\xx_0\in \bF$ be an arbitrary base 
point. Since the stabilizer $\bG_{\xx_0}$ is finite, we can choose a 
$\bG_{\xx_0}\!$-invariant metric.
Using this metric, the transverse vector space
\hbox{$V_{\xx_0}~=~T_{\xx_0} \bX/T_{\xx_0}\bF$} can be identified 
with the normal vector space consisting of all tangent vectors to $\bX$ 
at $\xx_0$ which are orthogonal to the fiber at ${\xx_0}$.
Given $\varepsilon>0$, we can consider geodesics of length $\varepsilon$
starting at $\xx_0$ which are orthogonal to $\bF$ at $\xx_0$. 
If $\varepsilon$
 is small enough, these geodesics will sweep out a smooth $d$-dimensional disk 
$D_\varepsilon$ which meets $\bF$ transversally, 
where  $d$ is the codimension
 of $\bF$ in $\bX$. 
Since the transverse disk $D_\varepsilon$ is canonically diffeomorphic to the 
$\varepsilon$-disk in $V_{\xx_0}$ by this construction, it follows that the 
group ${\mathcal R}_{\xx_0}$ acts effectively on $D_\varepsilon$.
\bigskip

\begin{figure}[h!]
\centerline{\includegraphics[width=2.in]{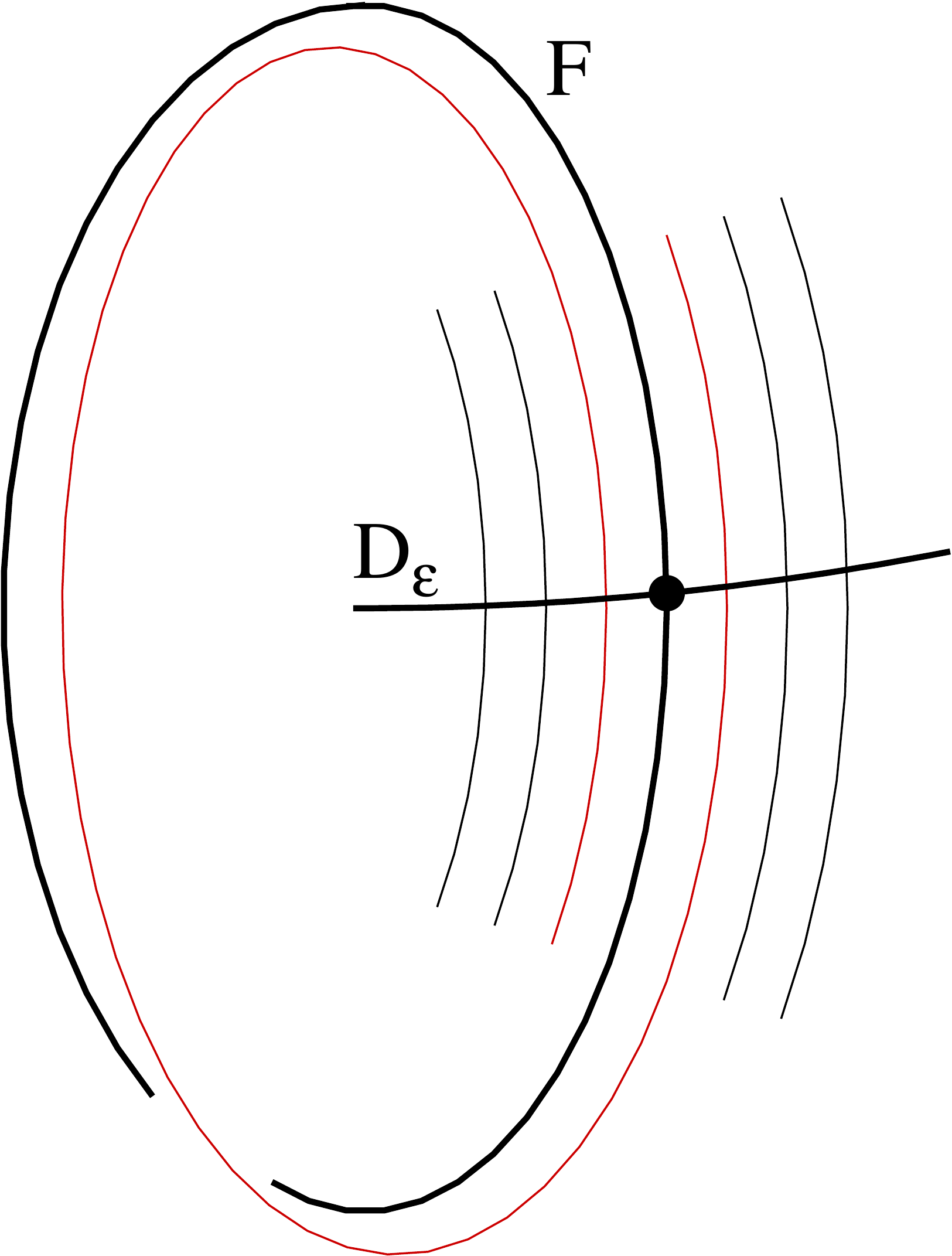}}
\caption{\label{F-moeb} \sf A transversal $D_\varepsilon$ to the fiber $\bF$ and
several nearby fibers. In this example, a neighborhood of $\bF$ within $\bX$
 is a M\"obius band.}
\end{figure}
\bigskip

We will prove the following.
\smallskip

\begin{theo}[{\bf Weak Orbifold Theorem}]\label{T-wot} 
Let $\xx$ be a point with finite stabilizer, and let $\yy=\bpi(\xx)$ be its
image in the quotient space $\bY=\bX/\bG$. If the action is weakly 
 proper, then $\bY$  is a weak orbifold. More explicitly,
$\bY$ is locally homeomorphic at $\yy$ to
the quotient of the \hbox{$d\!$-dimensional} transverse vector space $V_\xx$ 
by the action of the  
finite group ${\mathcal R}_\xx$, which acts linearly on it, at the origin.
In particular, $\bY$ is locally Hausdorff and metrizable near $\yy$,
and also locally compact. Furthermore, the projection map from a small
transverse disk $D_\varepsilon$ to $\bY$ is $r$-to-one 
outside of a subset of measure zero.
\end{theo}
\smallskip

\begin{coro}\label{C-?}
Given an action of a Lie group $\bG$ on a manifold $\bX$ with finite 
stabilizers, there are three well
defined open subsets 
$$ U_{\sf LP}~\subset~U_{\sf WP}~\subset~U_{\sf LHaus}~\subset~\bX/\bG~.$$
Here $U_{\sf LP}$ is the set of locally proper points,
$U_{\sf WP}$ is the set of weakly proper points, and $U_{\sf LHaus}$
is the set of all locally Hausdorff points.
\end{coro}\ssk

This corollary follows easily from the theorem and the discussion above.
The proof of Theorem \ref{T-wot} will depend on three lemmas. 
\medskip

\begin{lem}[\bf Invariant Metrics]\label{L-met} In the real case,
given any finite subgroup  
$\boldsymbol\Gamma\subset \bG$ 
there exists a smooth \hbox{$\boldsymbol\Gamma\!$-invariant} 
Riemannian metric on the
 space $\bX\,.$ Similarly, in  the complex case $\bX$ has
 a smooth $\boldsymbol\Gamma\!$-invariant Hermitian metric.
\end{lem}
\smallskip

\begin{proof} Starting with an arbitrary smooth Riemannian or Hermitian metric,
average over its transforms\footnote{A Riemannian metric
can be described as a smooth function $\mu$ which
assigns to each $\xx\in \bX$ a symmetric positive definite inner product
$\mu_\xx(v,w)$ on the vector space $T_\xx \bX$ of tangent vectors at $\xx$. 
Given any diffeomorphism $f:\bX\to \bX{\bf '}$, and given a Riemannian metric
 $\mu$ on $\bX{\bf '}$, we can use the first derivative map $f_*:T_\xx \bX
\stackrel{\cong}{\longrightarrow}T_{f(\xx)}\bX{\bf '}$ 
to pull back the metric, setting
$$ (f^*\mu)_\xx(v,w)~=~\mu_{f(\xx)}\big(f_*(v),\,f_*(w)\big)~\in~\R~.$$
In particular, given any finite group $\boldsymbol\Gamma$ consisting of $r$ 
diffeomorphisms from $\bX$ to itself, we can form the average 
$\widehat\mu ~=~\frac{1}{r}\sum_{\g\in \boldsymbol\Gamma} \g^*\mu~.$
The construction in the complex case is similar, using Hermitian inner
 products.} by elements of $\boldsymbol\Gamma$. Then each element
of $\boldsymbol\Gamma$ will represent an isometry for the averaged metric.
\end{proof}
\smallskip

\begin{lem}[{\bf Local Product Structure}]\label{L-lps}
Using this metric, let $D_\varepsilon$  be the
open disk swept out by normal geodesics of length less than $\varepsilon$ at
 $\xx$. Then any translated disk $\g(D_\varepsilon)$ is determined uniquely by
 its center point $\xx' = \g(\xx)$.  For $\xx'$ near $\xx$ in $\bF$,
 two such translated disks are disjoint, unless 
they have the same center point. It follows that some neighborhood of $\xx$ 
in $\bX$  is diffeomorphic to the product of $D_\vep$ with a neighborhood of
 the identity element in $\bG$.
\end{lem}
\smallskip

\begin{proof} (Compare Figure \ref{F-moeb}.) 
If $\g_1(\xx)=\g_2(\xx)$, then evidently $\g_1^{-1}\g_2\in \bG_\xx$. Since 
elements of $\bG_\xx$ map $D_\varepsilon$ to itself, it follows that
$\g_1(D_\varepsilon)=\g_2(D_\varepsilon)$.  Finally, if $W$ is a small
 neighborhood of  the identity in $\bG$, and if $\varepsilon$ is small enough,
 then since the  tangent space to $\bX$ at $\xx$ is the direct sum of the 
tangent space to $\bF$ and the space of normal vectors at $\xx$, it is not hard
 to check that  the map $(\g,\,\delta)\mapsto \g(\delta)$  sends  
 $W\times D_\varepsilon$  diffeomorphically onto an open subset of $\bX$. 
\end{proof}
\smallskip

\begin{figure}[h!]
\centerline{\includegraphics[width=2.5in]{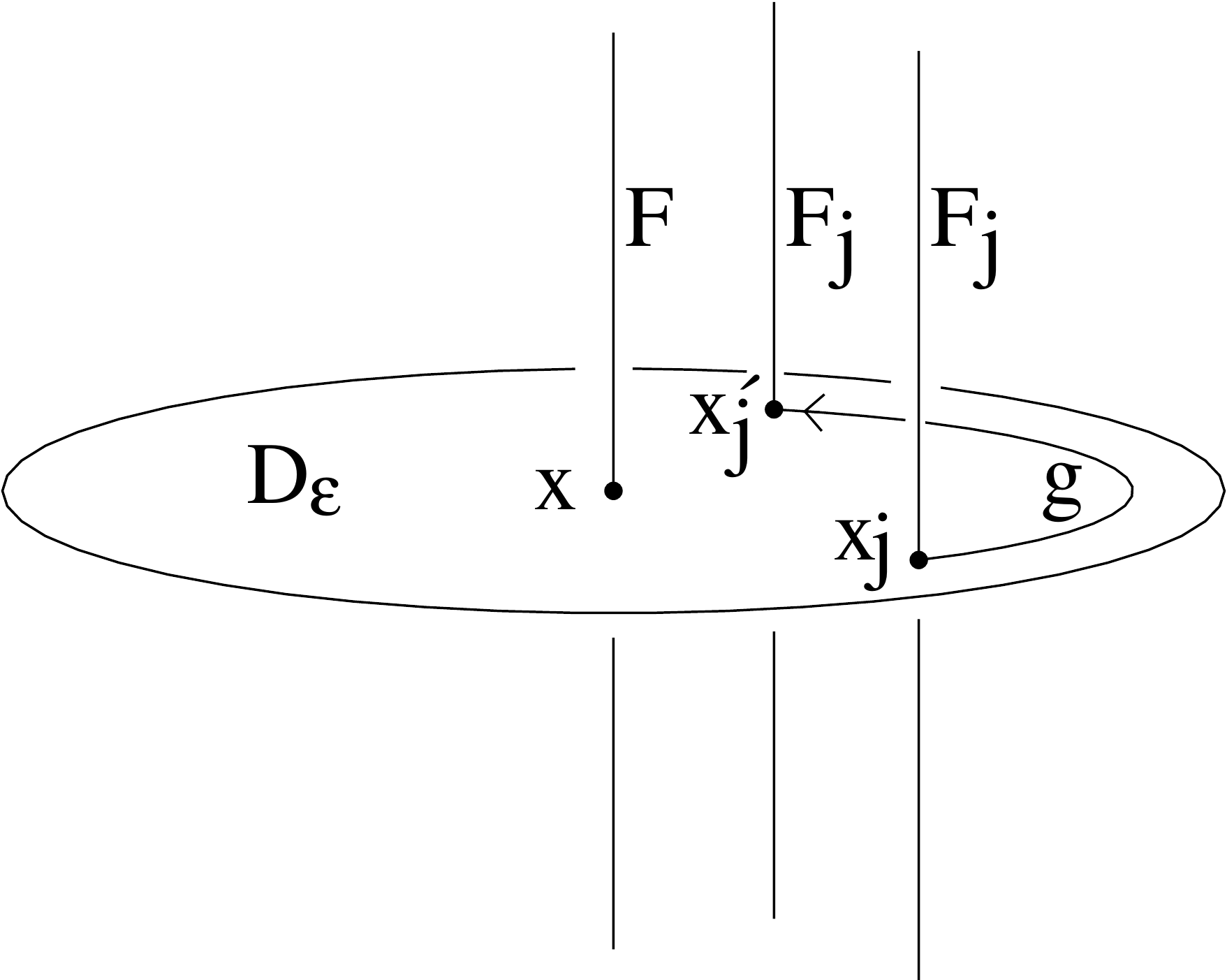}}
\caption{\label{F-thmpic}\sf In the weakly proper case, 
two points of
$D_\varepsilon$ belong to the same fiber only if there is an element
$\g\in \bG_\xx$ carrying one to the other.}
\end{figure}
\medskip

\begin{proof}[Proof of Theorem \ref{T-wot}]
Note first that elements of $\bG_\xx$ carry geodesics to geodesics, mapping a
small disk $D_\varepsilon$ onto itself, and mapping each fiber onto itself.
We must first show that two points of a sufficiently small disk
$D_\vep$ will belong to the same fiber only if some element of $\bG_\xx$
maps one to the other. (Compare Figure~\ref{F-thmpic}.) 
Suppose, for arbitrarily large $j>0$, that there exist points $\xx_j$ 
and $\xx_j{\bf '}$ in $D_{1/j}$ which belong to the same fiber, so that 
$\xx_j{\bf '}=\g_j(\xx_j)$ for some $\g_j\in \bG$; but so that $\xx_j$ and 
$\xx_j{\bf '}$ are not in the same orbit under $\bG_\xx$. Since the action is
 weakly proper, we  can choose
these group elements $\g_j$ within a compact subset $K\subset \bG$. After
passing to an infinite subsequence, we can assume that these elements $\g_j$
tend to a  limit \hbox{$\g\in K$}. Since $\g(\xx_j)=\xx_j{\bf '}$ with both 
$\xx_j$ and $\xx_j{\bf '}$ tending to $\xx$, it follows by continuity that 
$\g(\xx)=\xx$. Therefore $\g_j(\xx)$ tends to $\xx$ as $j\to\infty$ within the 
subsequence. Thus, by Lemma \ref{L-lps}, each image  $\g_j(D_\varepsilon)$ is 
either equal to or disjoint from $D_\varepsilon$. 
Since $\g_j(\xx_j)=\xx_j{\bf '}\in D_\varepsilon$, 
it follows that $\g_j\in \bG_\xx$ whenever $j$ is sufficiently 
large,  as required. 

This shows that the quotient $D_\vep/\bG$ maps bijectively to its image
in $\bX/\bG$. If $\vep$ is small enough, the same will be true for the compact
disk $\overline D_\vep$.
It is easy to see that the quotient of any compact metric space by a
finite group action is compact, with an induced metric. 
In fact, we can use the Hausdorff metric
on the space of all non-empty compact subsets; and each $\bG$-orbit is such a
compact subset.  Therefore, under the hypothesis of Theorem~\ref{T-wot}
it follows that
the quotient space is locally compact, metric, 
and hence Hausdorff, near $\yy$. Since the action of ${\mathcal R}_\xx$ 
on the transverse vector space $V_\xx$ is linear and effective,
it follows, for each non-trivial cyclic subgroup of ${\mathcal R}_\xx$, that the
action is free except on some proper linear subspace of $V_\xx$. Therefore the
 projection map from $D_\varepsilon$ to $\bY$ is  $r$-to-one except on a 
finite union  of  linear subspaces.
\end{proof}

\begin{rem}\label{R-anal}
In the unramified case, of course $\bY$ inherits the structure of a real
or complex analytic manifold locally. However, in general the quotient
need not  be even a topological manifold.
 Perhaps the simplest non-manifold example is the quotient of
the Euclidean space $\R^3$ by the two element group $\{\pm 1\}$ acting by
$\g(\xx)=\pm\xx$. In this case, the quotient is not locally orientable
near the origin. 
\medskip

\begin{rem}[{\bf Rational Homology Manifolds}]\label{R-rathom} 
We will show that:

\begin{quote} \it
Any complex weak orbifold or any locally orientable real weak 
 \hbox{orbifold}
is a  \hbox{\textbf{\textit{rational homology manifold}}}, 
 in the sense that any point  of such 
an orbifold has a neighborhood homeomorphic to the cone over a space
with the rational homology of a sphere.
\end{quote}
\noindent
Here the local orientability condition in the real case means that the
 space $\bY$ must be locally of the form
$D_\varepsilon/{\Gamma}$, where the action of the finite group $\Gamma$
preserves orientation. In other words $\Gamma$ must be contained in the
rotation group ${\rm SO}(d)$, rather than the full orthogonal group 
${\rm O}(d)$.\ssk

The statement then follows from the following more general principle:

\begin{lem} \label{L-rathom}
If a finite group $\Gamma$ acts on a finite cell complex $K$, then the
rational homology $H_*(K/\Gamma;\,\Q)$ is isomorphic to the subgroup 
$$H_*(K;\,\Q)_\Gamma\subset H_*(K;\,\Q)$$ consisting of all elements which are
 fixed under the induced action of $\Gamma$. 
There is a similar statement for cohomology.
\end{lem}

\begin{proof} After passing to a suitable subdivision of the cell complex $K$
we may assume that each group element which maps a cell onto itself acts 
as the identity map on this cell. Choosing some orientation for each cell,
the associated chain complex $C_*(K)=C_*(K;\,\Q)$ is the graded rational vector
 space with one  basis element for each cell.  The projection map 
$\bpi:K\to K/\Gamma$ induces a chain mapping\break 
$\bpi_*:C_*(K)\to C_*(K/\Gamma)$ between these rational chain complexes. 
That is, $\bpi_*$ maps $C_n(K)$ to $C_n(K/\Gamma)$, and commutes with the 
boundary operator \break $\partial:C_n(K)\to C_{n-1}(K)$  i.e., 
we obtain the following commutative diagram
$$\xymatrix{C_n(K) \ar[d]_\partial\ar[r]^{\bpi_n} & C_{n}(K/\Gamma)
\ar[d]^\partial\\
C_{n-1}(K) \ar[r]^{\bpi_{n-1}} & C_{n-1}(K/\Gamma)}~.$$
  But there is also  a less familiar chain map 
$$ \bpi^*:C_*(K/\Gamma)\to C_*(K)$$
in the other direction, which sends each cell of $K/\Gamma$ 
to the weighted sum of the cells of $K$ which lie over it. 
Here each such cell $\sigma$ of $K$ is to be weighted
by the number of elements in the stabilizer $\Gamma_\sigma\subset\Gamma$.
Then the composition
$$C_*(K/\Gamma)\stackrel{\pi^*}{\longrightarrow} C_*(K)
\stackrel{\pi_*}{\longrightarrow} C_*(K/\Gamma) $$
is just multiplication by the order of $\Gamma$; so the same is true of the
 induced composition
$$H_*(K/\Gamma)\stackrel{\pi^*}{\longrightarrow}  H_*(K)
\stackrel{\pi_*}{\longrightarrow} H_*(K/\Gamma) $$
of rational  homology groups. Since this composition is bijective,
it follows easily that $H_*(K/\Gamma)$ maps isomorphically onto its 
image in $H_*(K)$; and also that 
$H_*(K)$ splits as the direct sum of the image of 
$\bpi^*$ and the kernel of $\bpi_*$.
On the other hand the other composition
$$H_*(K)\stackrel{\pi_*}{\longrightarrow} H_*(K/\Gamma)
\stackrel{\pi^*}{\longrightarrow}  H_*(K)$$
maps each element of $H_*(K)$ to the sum of its images under the various 
elements of $\Gamma$. It follows that the kernel of $\bpi_*$ is the subspace
consisting of all elements $\eta\in H_*(K)$ such that 
$~\sum_{\bgam\in\Gamma}~\bgam_*(\eta)\,=\,0$. Since every element of the image
of $\bpi^*$ is $\Gamma$-invariant, and no non-zero element of the kernel of 
$\bpi_*$ is $\Gamma$-invariant, the conclusion follows.
(We thank Dennis Sullivan for supplying this argument.)\end{proof}
\medskip

In particular, if $K$ is a rational homology sphere and the action of $\Gamma$
preserves orientation, then it follows that $K/\Gamma$ is also a rational 
homology sphere. 

Now suppose as in Theorem \ref{T-wot} that the quotient space
 $\bY$ is locally homeomorphic to $\R^d/\Gamma$, where $\Gamma$ is now the
ramification group.
Then we can choose a $\Gamma$-invariant simplicial structure on $\R^d$. 
Taking $K$ to be the star boundary of the origin; that is the boundary
of the union of all closed simplexes which contain the origin, it follows that
$K/\Gamma$ is a homology $(d-1)$-sphere, and hence that
$\bY/\Gamma$  is a rational homology $d$-manifold. 
The corresponding statement in the complex case follows easily.
\end{rem}
\medskip

\begin{rem}[{\bf Quotient Analytic Structures}]\label{R-QAS}
Whether or not the quotient $\bY$ is a (possibly non-Hausdorff)
 topological manifold, we can put
some kind of ``analytic structure'' on it as follows. (Recall  that we
 use the word
analytic as an abbreviation for real analytic in the real case, and complex
analytic in the complex case.)
To every open subset $\bY{\bf '}\subset \bY$ assign the algebra 
$\cA(\bY{\bf '})$ consisting of 
all functions $f$, from $\bY{\bf '}$ to $\R$ in the real case or to $\C$ in the
complex case, such that the composition $~  f\circ\bpi~$ mapping 
$\bpi^{-1}(\bY{\bf '})$ to $\R$ or $\C$ is analytic,  where 
$\bpi: \bX \rightarrow \bY$ is the  projection  map. \medskip

{\bf Definition.\/} We say that $\bY$ is a $d$-dimensional
\textbf{\textit{analytic manifold}} if for every point of $\bY$ there is a 
neighborhood $\bY{\bf '}$ and functions 
\hbox{$f_1,\ldots,\,f_d\in\cA(\bY{\bf '})$} 
such that:

\begin{quote}
\begin{itemize}
\item[{\bf(1)}] The correspondence 
$\quad\yy\mapsto \big(f_1(\yy),\,\ldots,\,f_d(\yy)\big)\quad$
maps $\bY{\bf '}$ homeomorphically onto an open subset of\,\footnote
{In the real case, it would be natural to also include
manifolds-with-boundary by allowing a closed half-space as model space in 
Item {\bf(1)} above. In fact, one could also include 
``manifolds with corners'' by allowing a convex polyhedron as model space.}
 $\R^d$ or $\C^d$; and
\item[{\bf(2)}] Every element $f\in\cA(\bY{\bf '})$ can be expressed as an
 analytic function of $f_1,\ldots,\, f_d$. 
\end{itemize}
\end{quote}
\noindent In other words, for every such $f$
there must be an analytic
function $F$, defined on some open subset of $\R^d$ or $\C^d$, such that
$$f(\yy)~=~F\big(f_1(\yy),\,\ldots,\,f_d(\yy)\big)\quad{\rm for~ all}
\quad \yy\in \bY{\bf '}~.$$
\end{rem} 
\medskip

{\bf Another Smooth Example.} 
Let ${\mathfrak S}_n$ be the symmetric group on $n$ elements acting on $\R^n$ 
by permuting the $n$ coordinates. Then the quotient
$\R^n/{\mathfrak S}_n$ is a real analytic manifold which is isomorphic to $\R^n$
itself. We can simply choose $f_1,\, \cdots,\, f_n$ to be the elementary
symmetric functions of the $n$ coordinates. Similarly $\C^n/{\mathfrak S}_n$ is
biholomorphic to $\C^n$.\medskip

Such smooth examples seem to be rather rare when $n\ge 2$.
Here is more typical example of a ramified action  with a topological
manifold as quotient.\bigskip

{\bf A Simple Non-Smooth Manifold Example.} Let the two element group 
$\{\pm 1\}$ act on $\R^2$ by 
$$ (x,\,y)~\mapsto ~\pm(x,\,y)~.$$ 
Then the quotient space $\bY$ is clearly homeomorphic to $\R^2$. In fact 
if we introduce the complex variable $z=x+iy$, then
\hbox{$z^2=x^2-y^2+2ixy$} provides a good complex parametrization. 
However, the set $\cA(\bY)$ consists of all maps \hbox{$\bY\to \R$} 
which can be expressed as real analytic functions of 
$$x^2,~~~y^2,\quad{\rm and}\quad x\,y~.$$ The
functions $f_1=x^2-y^2$ and $f_2=2\,x\,y$ would satisfy Condition~$\bf(1)$
of the Definition in Remark~\ref{R-QAS}. However there is no way of
 expressing
$$ x^2+y^2~=~\sqrt{f_1^{\,2}+f_2^{\,2}}$$
as a smooth function of $f_1$ and $f_2$. In fact,
 no choice of $f_1$ and $f_2$ will satisfy Condition~{\bf(2)}. One way to see 
this is to note that the correspondence
$$ (x,\,y)~\mapsto~(\xi,\,\eta,\,\zeta)~=~(x^2,\,y^2,\,xy)$$ 
sends the real plane $\R^2$ to a topological submanifold of $\R^3$ which is
clearly not smooth, since it projects onto the positive quadrant in the
  $(\xi,\,\eta)$-plane.

In the complex analog, with the group $\{\pm 1\}$ acting on $\C^2$,
the quotient $\C^2/\{\pm 1\}$ is not even a topological manifold, since the
 quotient space with the origin removed is not simply-connected. 
\bigskip

{\bf A Wild Example.} Let $\mathbb H$ be the space of quaternions.
We will give an example of a finite group acting smoothly on 
$\R\times{\mathbb H}$ with the following rather startling property.
The quotient space $(\R\times{\mathbb H})/G_{120}$ is homeomorphic to $\R^5$;
but the set of ramified points $\R\times{\bf 0}$ corresponds to a line
in this quotient space which is so wildly embedded that its complement
is not simply connected.

To begin the construction, note that the unit sphere $S^3\subset{\mathbb H}$
can be described as the universal covering group of the rotation group
 ${\rm SO}(3)$. The 60 element icosahedral subgroup of ${\rm SO}(3)$ is
covered by the 120 element \textbf{\textit{double icosahedral group}} 
\hbox{$G_{120}\subset S^3$.} The quotient 
space $S^3/G_{120}$ is the ``Poincar\'e fake sphere", with
 the\break homology of the standard \hbox{3-sphere.} 
If we let the group $G_{120}$ act on $\mathbb H$ by left multiplication,
then the quotient 
${\mathbb H}/G_{120}$ is not a manifold, since a punctured neighborhood
of the origin is not simply-connected. However, the 
``double suspension theorem'' of Cannon and Edwards implies that the product
 
\centerline{$\R\times({\mathbb H}/G_{120}~)\cong~(\R\times{\mathbb H})/G_{120}$}

\noindent is a simply-connected manifold, homeomorphic to $\R^5$.
(Compare  \cite{Ca} or \cite{Ed}.) This product cannot 
be given  any differentiable structure such that the subset
\hbox{$\R\times {\bf 0}$} of ramified points is a differentiable submanifold.
This follows since the complement of this
one-dimensional topological submanifold has fundamental group $G_{120}$.
\end{rem}
\smallskip

\begin{rem}[{\bf Locally Proper Actions and Orbifolds}]\label{R-p}
 Recall from Definition~\ref{D-prop}
that the action of $\bG$ on a locally compact space $\bX$ is called 
\textbf{\textit{proper}} if every pair $(\xx,\xx{\bf '}\,)\in \bX\times \bX$ has
 a neighborhood $U\times U'$ such that the set of $\g\in G$ with 
$\g(U)\cap U'\ne\emptyset$ has compact closure. 

(Here are two alternative versions of the definition. The action is proper
if and only if: 

\begin{quote} for any compact sets $K_1\,,~K_2\subset \bX$, the set of all
$\g\in \bG$ with $\g(K_1)\cap K_2\ne\emptyset$ is compact; 
\end{quote}

\noindent or equivalently, if and only if
 the map  $(\g,~\xx)\mapsto\big(\g(\xx),~\xx\big)$
from $\bG\times \bX$ to $\bX\times \bX$ is a proper map. The proofs
are straightforward.) \smallskip

The action is \textbf{\textit{locally proper}} at $\xx$ if it is proper 
throughout some $\bG$-invariant neighborhood of $\xx$.
\end{rem}
\medskip

One important property of locally proper actions is the following.
\smallskip

\begin{lem}\label{L-prop} If the action is locally proper, 
with finite stabilizers, then  for all $\xx{\bf '}$ sufficiently close to
 $\xx$ the stabilizer $\bG_{\xx{\bf '}}$  is isomorphic to a subgroup of 
$\bG_\xx$. In  particular, the order $|\bG_\xx|$ of the stabilizer is upper
 semi-continuous as a function of $\xx$, so that $|\bG_{\xx{\bf '}}|\le|\bG_\xx|$
 for all $\xx{\bf '}$  sufficiently close to $\xx$. 

\end{lem}
\smallskip

\begin{proof}
Suppose that there were points $\xx_j{\bf '}$ arbitrarily close to $\xx$ with
$\bG_{\xx_j{\bf '}}$\, {\it not}\, isomorphic to a subgroup of $\bG_\xx$.
Since the action is locally proper, there is a compact set $K\subset \bG$
such that the stabilizer $\bG_{\xx{\bf '}}$ is contained in $K$
for all $\xx{\bf '}$ in some neighborhood of $\xx$. The collection 
of all compact subsets of $K$ forms a 
compact metric space, using the Hausdorff metric. 
Therefore, given any sequence
of such points $\xx_j{\bf '}$ converging to $\xx$, after passing to an infinite 
subsequence we can assume that
the sequence of finite groups $\bG_{\xx_j{\bf '}}\subset K$ converges to a 
Hausdorff limit set $\bG{\bf '}\subset \bG_\xx$ as $j$ tends to infinity. It is
 not hard to see that this limit $\bG'$ must be a subgroup of $\bG_\xx$.
We claim that the group $\bG_{\xx_j{\bf '}}$ is actually isomorphic to $\bG'$ for
large $j$. In fact, the correspondence which maps each $\g\in \bG_{\xx_j{\bf '}}$
to the closest point of $\bG'$ is certainly a surjective homomorphism  
for large $j$. Since we are assuming that $\bG_{\xx_j{\bf '}}$ is not isomorphic
to any subgroup of $\bG_\xx$, it follows that the kernel of this surjection
$\bG_{\xx_j{\bf '}}\to \bG'$ must contain some non-identity element $\g_j$ of
 $\bG$.
 But the sequence $\{\g_j\}$ must converge to the identity element.
Now consider the exponential map $\exp:{\L}\to \bG$,  which maps a 
neighborhood of the zero element in the Lie algebra to
a neighborhood of the identity. (Recall that the Lie algebra $\L$
 can be  identified with the tangent space to $\bG$ at the identity element.)
 We can set $\g_j=\exp(v_j)$ where $v_j$ tends 
to zero. Thus the group generated by $\g_j$ corresponds to the set
of all images $\exp(k\,v_j)$ with $k\in \Z$. Clearly these images fill out the
corresponding one-parameter subgroup more and more densely as 
$v_j\to 0$, so 
that the Hausdorff limit could not be a finite group. This contradicts our
hypothesis that $\bG_\xx$ is finite; and hence completes the proof.
\end{proof}

Note: This statement also follows from the proof of  Corollary 2.26.
\smallskip

\subsection*{ $\bG$-Invariant Tubular Neighborhoods}
Given any smooth $\bG$-action with finite stabilizers, and given any fiber 
$\bF\subset \bX$, it is not difficult to construct arbitrarily small
 $\bG$-invariant neighborhoods 
 $\bE$ of $\bF$ in $\bX$. Simply choose  a transverse disk 
$D_\varepsilon=D_\varepsilon(\xx_0)$ to the fiber $\bF$ at 
$\xx_0$ as in Lemma \ref{L-lps},  and let $\bE$ be the union of its images 
$\g(D_\varepsilon)$ as $\g$ varies over $\bG$. Recall that the image disk
$\g(D_\varepsilon)$ depends only on its center point $\xx=\g(\xx_0)$.
 We will use the alternate notation
$$D_\varepsilon(\xx)~=~ \g(D_\varepsilon) \qquad{\rm whenever}
\qquad \xx~=~\g(\xx_0)~.$$
\smallskip

If the action is locally proper, we can give a much more precise description.
\smallskip

\begin{theo}\label{T-tub} If the action is locally proper with finite
stabilizers, and if $\varepsilon$ is small enough, then the various 
disks $D_\varepsilon(\xx)$ with $\xx\in \bF$ are pairwise disjoint. It follows
 that $\bE$ is the total space of a locally trivial fiber bundle, with
 projection map  $\bE\to \bF$ which carries each fiber 
$D_\varepsilon(\xx)\subset \bE$ to its center point $\xx\in \bF$.
Furthermore, the $\bG$-orbits provide a foliation of $\bE$ which is 
transverse to the fibers, providing a local product structure as in Lemma 
$\ref{L-lps}$.
\end{theo} 
\smallskip

\begin{proof} (Compare \cite{Mei}, \cite{DK}.) 
 {\bf Step 1.} Since the stabilizer is finite, it follows 
that $\bF$ is locally diffeomorphic
to $\bG$. Therefore it follows as in Lemma~\ref{L-lps} that the various 
image disks  $\g\big( D_\varepsilon(\xx)\big)$, with 
$\xx$ close to $\xx_0$ in $\bF$,
are all disjoint, provided that $\varepsilon$ is small enough.\smallskip

{\bf Step 2.} Taking $\varepsilon>0$ as in Step 1. Suppose that there is a 
sequence of numbers $\varepsilon>\varepsilon_1>\varepsilon_2>\cdots$ tending
 to zero such that for each $j$  there are points 
$\xx_j\ne \xx'_j$ on $\bf F$ and group elements $\g_j$ and $\g'_j$ 
 with $\g_j(\xx_0)=\xx_j$ and $\g'_j(\xx_0)=\xx_j'$ such that
the two disks
$$  D_{\varepsilon_j}(\xx_j)~=~\g_j\big(D_{\varepsilon_j}(\xx_0)\big)
\quad{\rm and}\quad  D_{\varepsilon_j}(\xx'_j)~=~
\g'_j\big(D_{\varepsilon_j}(\xx_0)\big) $$
intersect each other at some point $\xx^*$. (See Figure \ref{F-prop}.) Then
we can write 
$$ \xx^*~=~\g_j(\delta_j)~=~\g'_j(\delta'_j)$$
for appropriate points $\delta_j\,,~\delta'_j\in D_{\varepsilon_j}$.
Now setting $\g_j^*=\g_j^{-1}\g'_j$, 
it follows that the disk $\g^*_j(D_\varepsilon)$ 
intersects $D_\varepsilon$ at the point  $\delta_j=\g^*_j(\delta')$,
although $\g^*_j(\xx_0)\ne \xx_0$.

Since the action is locally proper, it follows that all group elements $\g$
which satisfy
$\g(\overline D_\varepsilon)\cap \overline D_\varepsilon\ne\emptyset$ 
must be contained in some compact set $K\subset \bG$.
After passing to an infinite subsequence, we may assume that the
 group elements $\g^*_j$ tend to a limit in $\g^*\in K$.  Taking the limit
of the equation $\g^*_j(\delta_j)=\delta'_j$ as $j\to\infty$, we see that
$\g^*(\xx_0)=\xx_0$. Therefore the sequence
 $\g^*_j(\xx_0)$  must tend to $\xx_0$. Thus we have constructed disks
$\g^*_j(D_\varepsilon)$ with center point arbitrarily close to $\xx_0$ which
intersect $D_\varepsilon$ but are not equal to $D_\varepsilon$. 
This contradicts Step 1, and proves that {\it all}
of the disks $D_\varepsilon(\xx)$ must be pairwise disjoint.

\begin{figure}[t]
\centerline{\includegraphics[width=2in]{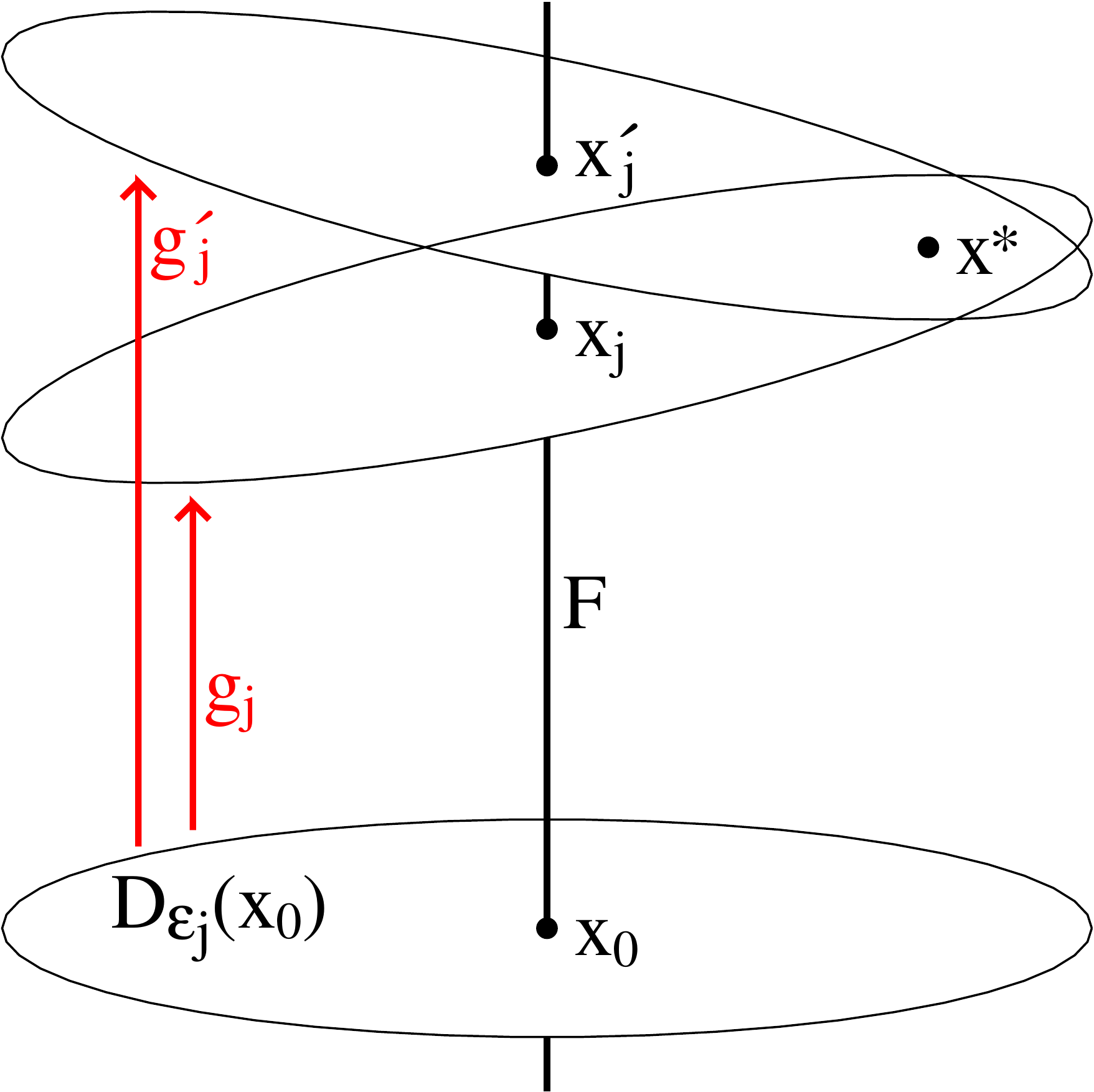}}
\caption{Illustrating the proof of Theorem \ref{T-tub}.\label{F-prop}}
\end{figure}

It follows from Lemma \ref{L-lps} that the mapping $\bE\to {\bf F}$ has a 
local product structure near the disk $D_\varepsilon$. Since we can use
translation by any group element $\g$ to translate this product structure
to a neighborhood of any disk
$\g(D_\varepsilon)$, this  completes the proof of Theorem \ref{T-tub}.
\end{proof}
\begin{figure}
\centerline{\includegraphics[width=2in]{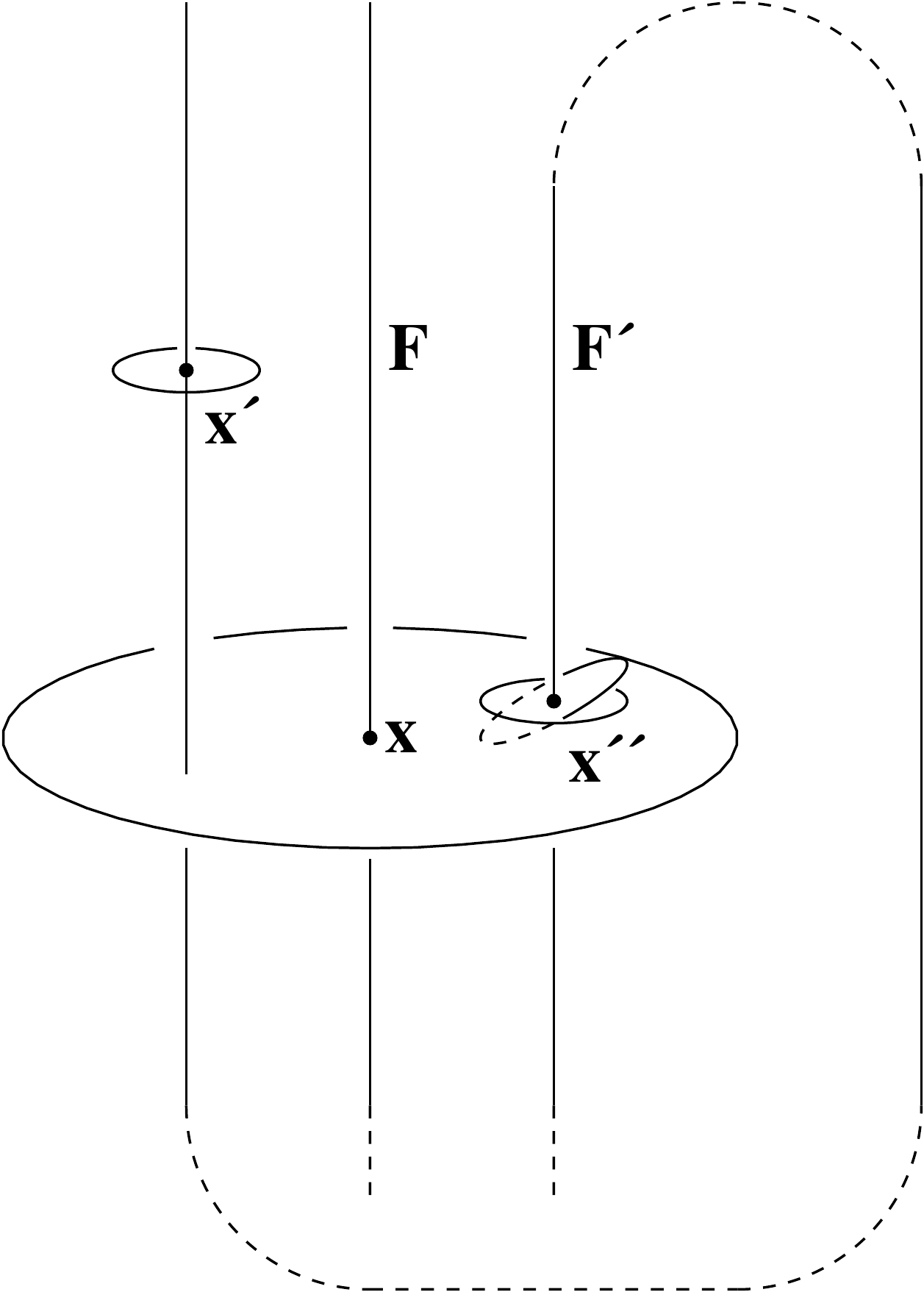}}
\caption{\sf Illustrating the proof of Corollary \ref{C-wot}.\label{F-orbi}}
\end{figure}

\msk
\begin{coro}\label{C-wot} If the action is locally proper with finite
 stabilizers, then the quotient is an orbifold $($Definition~$\ref{D-orbi})$.
\end{coro}

\begin{proof} We  know from Theorem \ref{T-wot}
that the quotient is a weak orbifold.
 Choose a disk $D_\vep$ as in Theorem \ref{T-tub}, and let $\bE$
be the associated tubular neighborhood. For any
fiber $\bF'$ which intersects $D_\vep$, we must study how a sufficiently small
tubular neighborhood of $\bF'$ is related to $\bE$. Let $\xx'$ be an 
arbitrary base point on $\bF'$, and let  $D'_{\vep'}$ be a small
transverse disk to $\bF'$ at  $\xx'$, using a $\bG_{\xx'}$-invariant metric.
Using any group element $\g$ which moves $\xx'$ to a point 
 $\xx''\in D_\vep\cap\bF'$,
we can move $D'_{\vep'}$ to a disk which is transverse to $\bF'$ at $\xx''$.
Using the local product structure, we can project $\g(D'_{\vep'})$ to a
subdisk of $D_\vep$. If $\vep'$ is small enough, this projection will
be a diffeomorphism. Now for any  $\g\in \cR_{\xx'}$ the action of $\g$
on $D'_{\vep'}$ will correspond to the action of some uniquely defined 
$\phi(\g)\in\cR_\xx$ on the image disk in $D_\vep$.
\end{proof}
\bsk

\begin{rem}[{\bf The Projective Linear Group ${\rm PGL}_m$}]\label{R-PGL}
In our applications, the group ${\bf G}$  
will be the real or complex projective linear group, 
either $\PGL_2$ in \S\ref{s-div} or $\PGL_3$ in later sections.

More generally, the group $\PGL_m$ over any field
can be defined as the quotient $\GL_m/N$, where $\GL_m$ is the group of linear
automorphisms of an $m$-dimensional vector space, and $N$ is the normal
 subgroup consisting of scalar transformations 
$$ \x\mapsto t\,\x~.$$
Here $t$ can be any fixed non-zero field element.
Writing the linear transformation as $(x_1,\,x_2,\,\ldots,\,x_m)\mapsto
(x'_1,\,x'_2,\,\ldots,\,x'_m)$, there is an associated automorphism 
$$(x_1:x_2:\ldots:x_m)\mapsto(x'_1:x'_2:\ldots:x'_m)$$
of the $(m-1)$-dimensional projective space over the field.
Automorphisms obtained in this way are called
\textbf{\textit{projective automorphisms}}.  Thus:
\smallskip

\begin{quote}\it Over any field, $\PGL_m$ can be identified with the group of 
all projective automorphisms of the projective space ${\mathbb P}^{m-1}$. 
\end{quote}\smallskip

\noindent Equivalently, $\PGL_m$ can be described as the group of all
equivalence classes of non-singular
$m\times m$ matrices over the field, where two matrices are equivalent if
one can be obtained from the other by multiplication by
a non-zero constant. Let ${\mathbb P}^{m^2-1}$ be the projective space 
consisting of  all lines through the origin in the $m^2$-dimensional
vector space consisting of
$~m\times m~$ matrices. Then it follows easily that: 

\smallskip

\begin{center}\it Over any field, the group $\PGL_m$ can be considered as\\ 
a Zariski open subset of the projective space ${\mathbb P}^{m^2-1}$.
\end{center}\medskip

\noindent 
Specializing to the real or complex case, it follows that $\PGL_m$ 
is a smooth real or complex manifold of dimension $m^2-1$,
with a smooth product operation. Hence it is a Lie group. 
\end{rem}

We will also need the following statement:\ssk

\begin{lem}\label{L-linalg}  Every element of $\PGL_m(\R)$ or $\PGL_m(\C)$
can be written as a composition
$${\bf g}~=~{\bf r}\circ{\bf d}\circ{\bf r}'~,$$
where ${\bf r}$ and ${\bf r}'$ are isometries, that is elements of the
projective orthogonal group ${\rm PO}_m$ in the real case
or the projective unitary group ${\rm PU}_m$ in the complex case,
 and where $\bf d$ is a
diagonal transformation of the form
$${\bf d}(x_1:\cdots:x_m)~=~(a_1x_1:\cdots:a_mx_m)$$
where the $a_j$ are real numbers with
$a_1\ge a_2\ge\cdots\ge a_m>0$. Furthermore, these real numbers $a_j$ are
 uniquely determined by $\bf g$ $($although   $\bf r$ and $\bf r'$ may not be
 uniquely determined$)$.
\end{lem}\medskip
 
(The numbers $a_j$ provide an
invariant description of how far $\bf g$ is from being an isometry with 
respect to the standard metric for ${\mathbb P}^{m-1}$.)
\ssk

\begin{proof}[ Proof of Lemma~\ref{L-linalg}.] To fix ideas we will discuss
 only  the complex case; but the real case is completely analogous. 
 This is proved by applying the Gram-Schmidt process to a
 corresponding linear transformation $\ell:V\to V'$, where 
 $V$ and $V'$ are $m$-dimensional complex vector spaces
with Hermitian inner product and with
 associated norm $\|v\|=\sqrt{v\cdot v}$. Given a linear bijection 
$\ell:V\to V'$, choose a unit vector $u_1\in V$
 which maximizes the norm $\|\ell(u_1)\|$. Then $\ell(u_1)$ can be
written as a product $a_1u'_1$ where $a_1>0$ is this maximal norm, and
where $u_1'$ is a unit vector in $V'$. Note that $\ell$
maps any unit vector $v$ orthogonal to $u_1$ in $V$ to a vector $v'$
 orthogonal  to $u'_1$ in $V'$. In fact, each linear combination
$u_1\cos(\theta)+v\sin(\theta)$ is another unit vector in $V$, which maps
to $a_1u'_1\cos(\theta)+v'\sin(\theta)$ in $V'$. A brief
computation shows that the derivative of the squared norm of this image
vector with respect to $\theta$ at $\theta=0$ is $2a_1 u'_1\cdot v'$.
Since the derivative at a maximum point must be zero, this proves that
$u'_1\cdot v'=0$, as asserted.

Thus $\ell$ maps  the orthogonal complement of $u_1$ to the orthogonal
complement of $u'_1$. Repeating the same argument for this map of 
orthogonal complements
we find unit vectors $u_2$ orthogonal to $u_1$
and $u'_2$ orthogonal to $u'_1$ so that  $\ell(u_2)=a_2u'_2$
with $a_1\ge a_2>0$. Continuing inductively,
we find an orthonormal basis $\{u_j\}$ for $V$ and an orthonormal
basis $\{u'_j\}$ for $V'$ so that
$$\ell(u_j)~=~ a_ju'_j\qquad{\rm with}
\qquad a_1\ge a_2\ge\cdots\ge a_m>0~.$$
Now taking $V$ and $V'$ to be copies of the standard $\C^m$,
it follows that $\ell$ is the composition of:
\begin{quote}
\begin{itemize}
\item[(1)] a unitary transformation which takes the standard basis for 
$~\C^m$ to the basis $\{u_j\}$,

\item[(2)] a diagonal transformation of the required form, and

\item[(3)] a unitary transformation taking $\{u'_j\}$ to the standard basis.
\end{itemize}
\end{quote}
\noindent This statement about the general linear group ${\rm GL}_m$
clearly implies the required statement about the projective
linear group $\PGL_m$.
This proves Lemma \ref{L-linalg}.
\end{proof}
\bigskip 

\setcounter{lem}{0}
\section{Moduli Space for Effective Divisors  of Degree $n$.}\label{s-div}
We first look at a basic family of moduli spaces which are relatively 
easy to understand. Since the discussions in the real and complex cases
are sometimes very similar, it will be convenient to use the symbol $\F$
to denote either $\R$ or $\C$. Let $\bP^n=\bP^n(\F)$ be the $n$-dimensional
projective space over $\F$. In this section, we will be interested in the
 projective line $\bP^1$, which is a circle in the real case, or a Riemann
sphere in the complex case. It will often be convenient to identify
$\bP^1$ with the union ${\widehat\F}=\F\cup\{\infty\}$. More precisely, 
 each point $(x:y)\in\bP^1$ can be identified
with the quotient $~~x/y\,\in\, {\widehat\F}=\F\cup\{\infty\}$. Note that 
the group $\bG=\PGL_2(\F)$ acting on $\bP^1(\F)$ corresponds to the group
of fractional linear transformations,
$$ z~\mapsto \frac{az+b}{cz+d}\quad {\rm with}\quad \
 a,\,b,\,c,\,d\in\F\,,~~~ ad-bc\ne 0~,$$
acting on $\Fhat$.

\ssk

By definition, an \textbf{\textit{effective divisor}} of degree $n$ on $\bP^1$
is a formal sum of the form
$$\cD~=~m_1\<\p_1\>+\cdots+m_k\<\p_k\>~,$$
where the $\p_j$ are distinct points of $\bP^1$,
and where the \textbf{\textit{multiplicities}} $m_k\ge 1$ are integers,
with $\sum m_j=n$. The set $|\cD|=\{\p_1,\ldots,\p_k\}\subset \bP^1$ 
will be called the \textbf{\textit{support}} of $\cD$.\ssk

Let $\wfD_n=\wfD_n(\F)$ be the space 
of all effective divisors of degree $n$ on  $\bP^1=\bP^1(\F)$. 
 In the complex case, if we think of a divisor
as the set of roots of a homogeneous polynomial, then it follows easily
that the space $\wfD_n(\C)$ can be
given the structure of a complex projective space $\bP^n(\C)$.
In the real case,
$\wfD_n(\R)$ can be  identified with the closed subset of $\bP^n(\R)$ 
corresponding to those real homogeneous polynomials which have only real roots.

The group $~\bG=\PGL_2(\F)~$ 
acts on $\bP^1$, and hence on the  space $\wfD_n$ of formal sums. Note that
the action on $\bP^1$ is \textbf{\textit{three point simply transitive}}. 
That is, there is one and only only one group element which take any ordered 
set of three distinct points of $\bP^1$ to any other ordered set of
 three distinct points.
It follows easily that the stabilizer $\bG_\cD$ for the action 
at a point $\cD\in\wfD_n$
is finite if and only if the number $k$ of points in $|\cD|$
satisfies $k\ge 3$.\ssk

\begin{definition} Let $\wfD_n^\fs$  be the open subset
of $\wfD_n$ consisting of effective divisors with finite stabilizer,
or in other words with at least three distinct 
points, and define the \textbf{\textit{moduli space}} for divisors
to be the quotient $\fM_n=\wfD_n^\fs/\bG$.
\end{definition}

One basic invariant for the $\bG$-orbit of a divisor is the 
\textbf{\textit{maximum multiplicity}} $1\le\max_j\{m_j\}\le n$ 
 of the points in $|\cD|$.
Note that the collection of all unordered $n$-tuples of distinct points
 of $\bP^1$ can be identified with the open subset $\fD_n\subset\wfD_n$ 
consisting of divisors with  \hbox{$\max_j\{m_j\}=1$}. 

In the complex case, this subset can be compared with the classical 
moduli space $\cM_{0,n}$ consisting of closed\footnote{By definition, 
a Riemann surface is closed if it is compact without boundary.}
Riemann surfaces of genus zero which are provided with an
ordered list of $n\ge 3$ distinct points; where two such marked Riemann
surfaces are identified
if there is a conformal isomorphism taking one to the other. (See 
Remark~\ref{R-oms} below.)
We can unorder these points by taking the quotient 
$\cM^{\sf un}_{0,n}=\cM_{0,n}/\fS_n$
under the action of the symmetric group of permutations of the ordered list.
It is not hard to check that the resulting space, consisting of isomorphism 
classes of genus zero curves with an unordered
list of distinct marked points, can be identified with our
open subset $\fD_n\subset\wfD_n$.
\ssk

\begin{theo}\label{T-D1} $\fM_n$ is a ${\rm T}_1$-space for every $n$; 
but it is a
Hausdorff space only for $n\le 4$. For any $n$, the
open subset of $\fM_n$ consisting of \hbox{$\bG\!$-equivalence} 
classes of divisors with maximum multiplicity satisfying
$$\max_j\{m_j\}<n/2$$ 
is a Hausdorff space and an orbifold. However if $n>4$, then any point for which
\hbox{$\max_j\{m_j\}\ge n/2~$} is not even locally Hausdorff.
\end{theo}
\smallskip

For $n>4$, we will use the notation $\fM^\ha_n$ for this maximal open 
Hausdorff subset of $\fM_n$.
\smallskip

\begin{theo}\label{T-D2} For $n\ge 5$,  
the space $\fM^\ha_n$ is compact for $n$ odd; but not for $n$ even.
\end{theo}
\smallskip

The case $n=5$ is particularly striking, since the non-Hausdorff
space $\fM_5$ consists of a compact Hausdorff space $\fM^\ha_5$ together
with just one ``bad'' point   of the form 
$~\((\,3\<\p\>+\<\q\>+\<\r\>\,\))$.\bsk

To begin the proof of Theorem \ref{T-D1}, we will study the cases $n\le 4$.
It is easy to check that $\fM_n$ is empty for $n<3$.
 For $n=3$, since the action of $\bG$ on 
$\bP^1$ is three-point simply transitive,
 it follows easily  that $\fM_3(\R)=\fM_3(\C)$ consists of a single point; 
with stabilizer the symmetric group $\fS_3$. For $n=4$ we have the following.
(Recall that each point of $\fM_n$ corresponds to an entire $\bG$-orbit of
divisors $\cD\in\fD^\fs_n$.)
\smallskip

\begin{prop}\label{P-D4}
The moduli space $\fM_4(\C)$ is isomorphic to the Riemann sphere
$\widehat\C$, while $\fM_4(\R)$ corresponds to a closed interval
contained in the circle\break
$\widehat\R\subset\widehat\C$. In both cases, $\fM_4$
contains one and only one improper point, corresponding to the $\bG$-orbit 
consisting of divisors $\cD$ of degree four with only three distinct
points. $(${\rm See Definition} $\ref{D-prop}.)$
However, this improper point  is weakly proper. In both the real 
and complex cases, the stabilizer $\bG_\cD$ is isomorphic to $\Z/2\oplus\Z/2$ 
at a generic point.\footnote{We say that a property holds
  for a \textbf{\textit{generic point}}
   if it is true for all points in some set which
   is dense and open in the Zariski topology.
   (Some authors prefer the term ``general point''.)}
In the complex case, there are three exceptional points,
namely the improper point with stabilizer $\Z/2$, and two
ramified points with ramification 
indices $r=2$ and $r=3$ respectively, and with stabilizers 
the dihedral group of order $8$ and the tetrahedral group of order
$12$. In the real case, only the improper point and the
dihedral group with $r=2$ can occur. 
\end{prop}

\begin{figure} [t!]
\begin{minipage}{1in}\vspace{.5in}$\fM_4(\C)$
\end{minipage}\begin{minipage}{3.5in}\includegraphics[width=3.in]{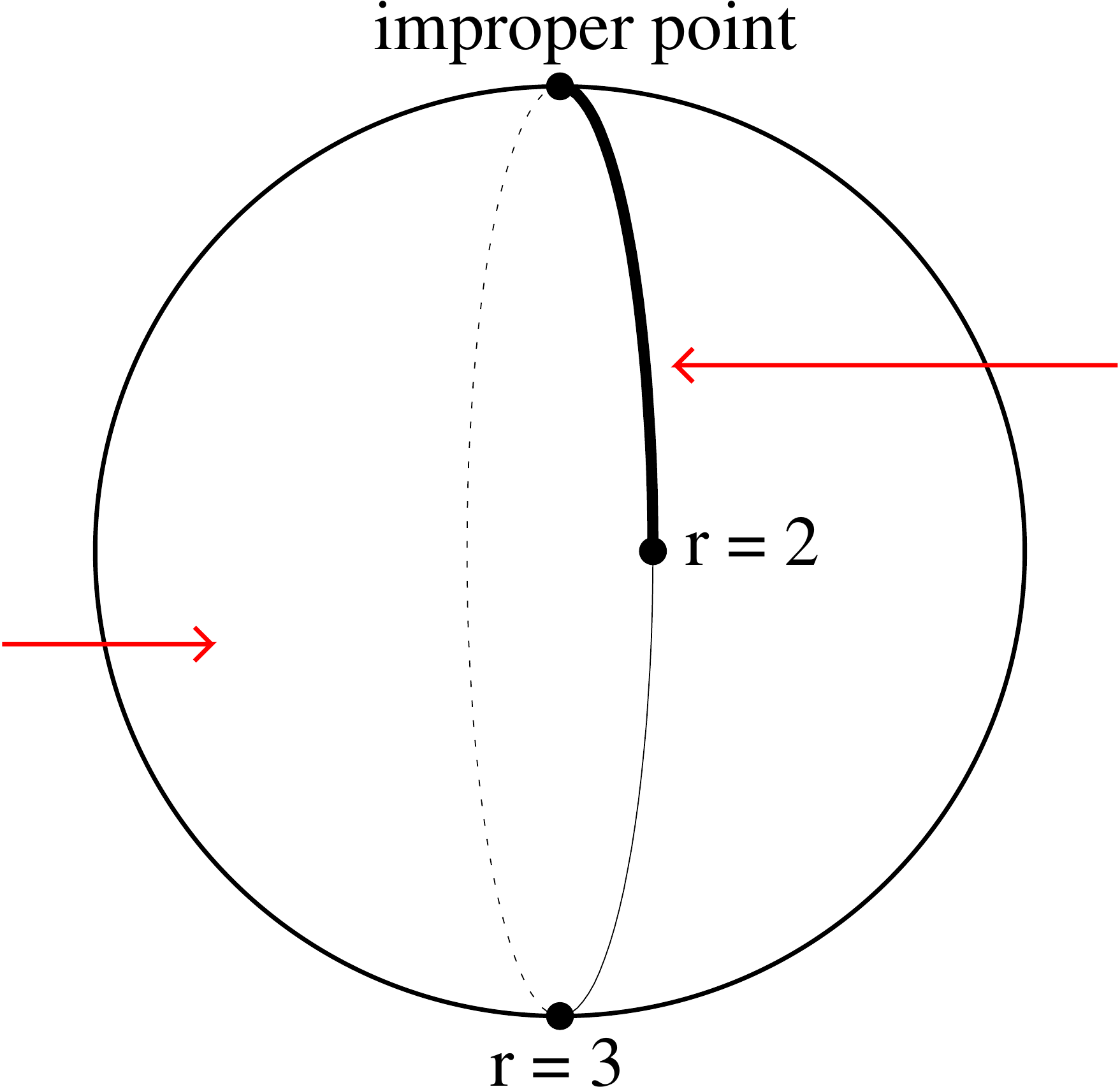}
\end{minipage}\begin{minipage}{1in}$\fM_4(\R)$\vspace{.89in}\end{minipage}
\caption{The moduli spaces $~~\fM_4(\R)\subset\fM_4(\C)$.\label{F-M4}}
\end{figure}\bigskip

In particular, it follows that $\fM_4$ is a compact Hausdorff space in both
the real and complex cases, and an orbifold except at one point. 
\smallskip

The proof of Proposition \ref{P-D4} will make use of two different projective
invariants associated with a 4-tuple of points in $\bP^1$. The first is the
cross-ratio, which depends on the ordering of the four arguments, and the
second is the ``shape invariant'' $\bJ$ which is independent of order.

It will be convenient to use cross-ratios of the form
\begin{equation}\label{E-cr1}
\brh(x,\,y,\,z,\,w)~=~ 
 \brh\left[\begin{matrix} x&y\\z&w\end{matrix}\right]~=~\frac{(x-y)(z-w)}
{(x-z)(y-w)}~, \end{equation}
where $x,\,y,\,z,\,w$ are distinct real or complex numbers.
This expression is well defined and continuous
on the space of ordered \hbox{4-tuples} of distinct points of $\R$ or $\C$,
taking values in $\widehat\R$ or $\widehat\C$. In either case it extends
uniquely to the case where any one of the four points is allowed to take 
the value $\infty$. For example as $w\to\infty$ Equation~(\ref{E-cr1}) tends
 to the limit 
\begin{equation}\label{E-cr2}
 \brh\left[\begin{matrix} x&y\\z&\infty\end{matrix}\right]
~=~\frac{x-y}{x-z}~.
\end{equation}
\smallskip

\begin{lem}\label{L-cr1}
 There is a necessarily unique projective
 automorphism carrying one ordered set of four distinct points of $\bP^1$ to
another if and only if they have the same cross-ratio, which can take any 
value other than $0,\,1$, or $\infty$.
\end{lem}

\begin{proof} It suffices to consider the special case
where the second 4-tuple has the form $(0,\,y,\,1,\,\infty)$, so that
the cross-ratio is $y$ by equation (\ref{E-cr2}). Using three point
transitivity, there is a unique projective automorphism taking the 
appropriate points to $0,\,1$ and $\infty$; and it follows that the remaining
point must map to the cross-ratio~$y$.\end{proof}
\medskip

However, as two of the four points  come together (so that only three
are distinct), the cross-ratio will tend to a limit belonging to
the complementary set $\{0,\,1,\,\infty\}$. (If only two of the four
points are
distinct, then the cross-ratio cannot be defined in any useful way.)

 Note that the cross-ratio is always unchanged as we interchange
the two rows, or the two columns, of the matrix 
$\left[\begin{matrix}x&y\\z&w\end{matrix}\right]$.
Thus we obtain the following:
\smallskip

\begin{lem}\label{L-cr2} For any $4\!$-tuple $(x,y,z,w)$ of four 
distinct points, there is a transitive $4$ element group of permutations 
of the  four points, isomorphic to\break $\Z/2\oplus\Z/2$, which preserves their
cross-ratio. Hence the stabilizer $\bG_\cD$ for the associated divisor
$\cD=\<x\>+\<y\>+\<z\>+\<w\>$ always contains $\Z/2\oplus\Z/2$ as a subgroup.
\end{lem}
\smallskip

\begin{proof} In both the real and complex cases, this follows immediately 
from the discussion above.\end{proof}
\smallskip

\subsection*{\bf The Shape Invariant $\bJ$.}
We next describe a number $\bJ=\bJ(x,y,z,w)$ which is invariant
under permutations of the four variables, and also under  projective
automorphisms of $\bP^1$. 

First consider the generic case where all four
points are distinct. After a projective transformation, we may assume that
$w=\infty$ and that $x,\,y,\,z$ are finite. Then $x,\,y,\,z$ are uniquely 
determined up to a simultaneous affine transformation. Therefore the
differences
\begin{equation}\label{E-difs}
 \alpha=x-y~,\quad \beta=y-z~,\quad \gamma=z-x
\end{equation}
are uniquely determined up to multiplication by a common non-zero constant.
Next consider the elementary symmetric functions
$$\sigma_1=\alpha+\beta+\gamma=0\,,\qquad  \sigma_2
=\alpha\beta+\alpha\gamma+\beta\gamma\,,\qquad \sigma_3=\alpha\beta\gamma~.$$
If we multiply $\alpha,\,\beta,\,\gamma$ by a common constant $t\ne 0$,
then each $\sigma_j$ will be multiplied by $t^j$. Therefore the ratio\footnote
{Here the factor of $-4/27$ has been inserted
so that $\bJ$ will take the value $+1$ in the case of dihedral symmetry,
where two of the three numbers $\alpha,\,\beta,\,\gamma$ are equal.}
\begin{equation}\label{E-J1}
\bJ~:=~ -\frac{4}{27}\frac{\sigma_2^{\;3}}{\sigma_3^{\;2}}~. 
\end{equation}
will remain unchanged. 
It might seem that $w$ plays a special role in this construction;
but remember from Lemmas \ref{L-cr1} and \ref{L-cr2} that there is a
 transitive group of projective automorphisms permuting the four variables.
Therefore it doesn't matter which of the four variables we put at infinity.

If only three of the four variables $x,\,y,\,z,\,w$ are distinct,
then we set $\bJ=\infty$. For example if $x=y$ so that $\alpha=0$, then
$\sigma_3=0$, hence $\bJ=\infty$. A brief computation shows that
$\bJ$ also tends to infinity if one of the variables $x,\,y,\,z$  tends
to infinity while the other two remain bounded.\medskip

In general, four points of $\bP^1$ determine six different cross-ratios,
according to the order in which they are listed 
(compare Remark \ref{R-cr}), but only one shape invariant.
\smallskip

\begin{lem}\label{L-shape}
Any one of these six cross-ratios determines the shape invariant according 
to the formula
\begin{equation}\label{E-J2}
\bJ~=~ \frac{4}{27}\frac{(\brh^2-\brh+1)^3}{\brh^2(1-\brh)^2}~.
\end{equation}
\end{lem} 
\smallskip

\begin{proof} 
Since both sides of  equation (\ref{E-J2}) are invariant under 
affine transformations of the plane, it suffices to consider the special 
case where the 4-tuple $(x,\,y,\,z,\,w)$ is equal to 
$(0,\, t,\, 1,\, \infty)  $, with cross-ratio $t$.
 (In fact one can choose an affine transformation
which maps $x$ to zero and $z$ to one, while keeping $\infty$ fixed.
The point $y$ will then necessarily map to the cross-ratio.) We then
have 
$$ \alpha=-t\,,~~\beta=t-1\,,~~\gamma=1~,$$
hence $\sigma_2=-(t^2-t+1)$ and $\sigma_3=t(1-t)$; and
the required identity (\ref{E-J2}) follows immediately.\end{proof}
\smallskip

\begin{rem}\label{R-J}
The shape invariant $\bJ$ is just the classical $j$-invariant of an 
associated cubic curve, divided by a constant factor 
of $12^3=1728$. (Compare Section \ref{s-deg3}.) To see the relationship,
first subtract the average $(x+y+z)/3$ from $x,\,y$, and $z$, in order to 
obtain a triple with $x+y+z=0$. These corrected variables will then be the roots
of a uniquely defined cubic equation $~X^3+AX+B=0$. If we express the
$\sigma_2$ and $\sigma_3$ of equation (\ref{E-J1})  as functions of these three
variables,  then computation shows that
$$ \bJ~=~\frac{4A^3}{4A^3+27B^2}~.$$ 
Here the denominator is the classical expression for the discriminant of a 
cubic polynomial, up to sign. Details of the computation will be omitted.
\end{rem}
\medskip

\begin{proof}[Proof of Proposition \ref{P-D4}]
First consider the complex case. The discussion above shows that every
divisor $~\cD=\<x\>+\<y\>+\<z\>+\<w\>~$ with at least three distinct elements
determines a point $\bJ(x,\,y,\,z,\,w)$ in the Riemann sphere $\wC$,
and that this image point is 
invariant under the action of the group $\bG=\PGL_2(\C)$ on the divisor.
It is easy to check that the resulting correspondence 
$$\fM_4(\C)~~\longrightarrow~~ \wC$$
is continuous and bijective, and hence is a homeomorphism.  

To describe the precise stabilizers for the various points of $\fM_4(\C)$
we will need the following. \smallskip

\begin{rem}[\bf More About Cross-Ratios]\label{R-cr}
For an arbitrary permutation of a set $\{x,\,y,\,z,\,w\}$ of 
four distinct points the cross-ratio $\brh(x,\,y,\,z,\,w)$
 will be  transformed by some corresponding rational map. 
By Lemma \ref{L-cr2}, we can always construct a permutation which preserves
cross-ratios so that the composition will map $w$ to itself.
Therefore it suffices to consider the symmetric group $\fS_3$
consisting of permutations of $\{x,\,y,\,z\}$ with $w$ fixed.
This group $\fS_3$ consists of a cyclic
subgroup of order three, together with three elements of order two.
It is not hard to check that the  elements of order two correspond
to the involutions which takes $\brh$ to either
\begin{equation}\label{E-odd}
 1/\brh\qquad{\rm or}\qquad 1-\brh\qquad{\rm or}\qquad \brh/(\brh-1)~;
\end{equation}
while the two elements of order three correspond to the rational maps 
\begin{equation}\label{E-even}
\brh~~\mapsto~~ 1/(1-\brh)\qquad{\rm and}\qquad\brh~~\mapsto~~
1-1/\brh~.\end{equation}
Thus a generic element  $\cD\in\wfD_4$ has six different associated 
cross-ratios, and any one of the six determines the other five. 
\end{rem}

If $x,\,y,\,z,\,w$ are distinct, then
evidently the stabilizer $\bG_\cD$
of the associated divisor $$\cD=\<x\>+\<y\>+\<z\>+\<w\>$$
can be identified with the group of all 
permutations of $\{x,\,y,\,z,\,w\}$ which preserve the cross-ratio. 
In particular, it always contains a subgroup isomorphic to 
\hbox{$\Z/2\oplus\Z/2$}.
If the six cross-ratios are all distinct, then the
stabilizer is equal to this commutative subgroup of order 4; but there are
three exceptional cases (including the degenerate case), 
corresponding to equalities between various
of the numbers (\ref{E-odd}) and (\ref{E-even}) and $\brh$.\medskip

{\bf Dihedral Symmetry.} If the shape invariant is $\bJ=1$, then
 there are only three associated cross-ratios, namely $-1,\,1/2$ and $2$.
(Each of these is fixed under one of the involutions of equation (\ref{E-odd}).)
As an example, a corresponding divisor can be chosen as
$$\cD~=~\<-1\>+\<0\>+\<1\>+\<\infty\>.$$
The associated stabilizer is the dihedral group of order eight,
generated by the rotation
$$ x\mapsto\frac{1+x}{1-x}
\quad{\rm with}\quad 0\mapsto 1\mapsto\infty\mapsto -1\mapsto 0~,$$
 together with the reflection $x\mapsto -x$. The ramification index is $r=2$.
\medskip

{\bf Tetrahedral Symmetry.} If $\bJ=0$ then there are only two associated
cross-ratios, namely $\brh=\frac{1\pm\sqrt{-3}}{2}$. A corresponding divisor 
can be obtained by placing the four points at the vertices
of a tetrahedron on the Riemann sphere (identified with the unit sphere
 in Euclidean 3-space). Then evidently the corresponding stabilizer
is the tetrahedral group of order 12 (the group of orientation preserving
isometries of the tetrahedron).
Since the cross-ratios are not real, this possibility can occur
only in the complex case. \medskip

It is not hard to check from the equations (\ref{E-odd}) and (\ref{E-even})
that these are the only non-degenerate examples for which there are not
six distinct cross-ratios.\medskip

{\bf The Degenerate Case.} If two of the four points come  together, then 
the possible cross ratios are $0,\,1,\,\infty$. (Compare Lemma \ref{L-cr1}.)
 Much of the discussion above breaks down
in this case. In particular, it is easy to check that the stabilizer has
only two elements. According to Lemma~\ref{L-prop}, this implies that the 
action of $\bG$ is not locally proper at such points. 
\medskip

\begin{rem}\label{R-improp}
Although the action of $\bG$ is not locally proper at such degenerate 
points, it is still weakly proper (Definition \ref{D-prop}).
In fact, any two divisors  in a neighborhood of an improper
divisor will have the form 
$$\cD_j~=~\<x_j\>+\<y_j\>+\<z_j\>+\<w_j\>\quad{\rm for}\quad j=1,\,2\,;$$ where
 the $x_j$ and $y_j$ are
very close (or equal) to each other for each $j$ (we will write 
$x_j\approx y_j$), but where $z_1\approx z_2$ and 
$w_1\approx w_2$ are well separated. If there is a group element
 taking $\cD_1$ to $\cD_2$,  then since $\brh(x_j,\,y_j,
\,z_j,\,w_j)\approx 0$, it must take $\{x_1,\,y_1\}$ to either  $\{x_2,\,y_2\}$
or   $\{z_2,\,w_2\}$. After composing with an
element of the central subgroup $\Z/2\oplus\Z/2$, we may assume that
$$ \{x_1,\,y_1\}\mapsto\{x_2,\,y_2\}\quad{\rm and~that}\quad
z_1\mapsto z_2,\,~~~w_1\mapsto w_2~.$$ This shows that we can choose the group
element to belong to a compact subgroup of $\PGL_2$, which proves
that the action is weakly proper. 
\end{rem}
\medskip

This completes proof of Proposition \ref{P-D4} in the complex case.\medskip

{\bf The Real Case.} A completely 
 analogous argument shows that the shape invariant induces an injective
map from $\fM_4(\R)$  into $\wR$. However the image $\bJ(\fM_4)$ is
no longer the entire $\wR$. It follows by inspection of equation 
(\ref{E-J2}) that the image is contained in the half-open interval 
$0<\bJ\le \infty$. In fact, we will show that the image 
is equal to the closed interval $1\le\bJ\le\infty$.

Here is a more precise statement in terms of cross-ratios. Recall that 
the  cross-ratio takes the value $-1,\,1/2,\,$ or $2$ in the case of dihedral
symmetry, and the value $~0,\,1,\,$ or $\infty$ in the degenerate case.
\medskip

\begin{lem} 
These six special values 
$$ -1,~~~ 0,~~~ 1/2,~~~ 1,~~~ 2,~~~ \infty$$
divide the circle $\bP^1(\R)$ into six closed subintervals.
 Each of these six intervals maps homeomorphically
 onto the interval $1\le\bJ\le\infty$.
\end{lem}
\smallskip

\begin{proof} Computation shows that the derivative of the function
$\brh\mapsto \bJ(\brh)$ of equation (\ref{E-J2}) is given by
$$ \frac{d\,\bJ}{d\,\brh}
~=~\frac{4(\brh-2)(2\,\brh-1)(\brh+1)\big((\brh-1)\brh+1\big)^2}{27\brh^3(\brh-1)^3}~.$$
It follows easily that this function $\brh\mapsto \bJ$ is 
alternately  increasing and decreasing  on these six intervals. Since we 
know that it takes the value $\infty$ at improper points, and the 
value $1$ at points with dihedral symmetry, this completes the proof.
\end{proof}
\medskip

The rest of the 
proof of Proposition \ref{P-D4} in the real case can easily be completed,  
since the arguments are almost the same as those in the complex case.
\end{proof}
\bigskip 

Next we must study the case $n>4$.
\smallskip

\begin{lem}\label{L-D5}
If $\cD=\sum_j m_j\<\p_j\>$ is a divisor of degree $n>4$ with 
$$~\max_j\{m_j\}~\ge ~n/2\,,$$ 
then the quotient space $\fM_n$ is not locally Hausdorff at the image
point $\bpi(\cD)$.
\end{lem}
\medskip

For the rest of this section, the real and complex cases are completely
analogous, so it will suffice to concentrate on the complex case. 
\medskip

\begin{proof}[Proof of Lemma \ref{L-D5}]
First consider the special case of divisors of even degree
  $n=2h\ge 6$, with \hbox{$\max_j\{m_j\}=h\ge 3$}.
Identifying $\bP^1$
with $\C\cup\{\infty\}$, let $\cD_1$ and $\cD_2$ be of the form 
$$\cD_j=\cD_j'+h\langle\infty\rangle~,$$
 where both $\cD_1'$ and $\cD_2'$ are divisors of degree $h=n/2$ 
 with support consisting of $h$ distinct points in the finite plane, and where
$h\langle\infty\rangle$ is the divisor
consisting only of the point $\infty$ with multiplicity $h$. Thus
the support $|\cD_j|$ has \hbox{$h+1\ge 4$} elements. Since $4>\dim(\bG)=3$,
we can always choose two such divisors $\cD_1$ and $\cD_2$ which do not 
belong to the same $\bG$-orbit,

Now consider the projective involution $\gr(z)=r^2/z$, where $r$ is a large
 real number.  Note that $\gr$ maps the neighborhood $|z|<r$ of zero
onto the neighborhood $|z|>r$ of infinity. Then the two divisors  
$$\cD_1'+\gr(\cD'_2)\qquad{\rm and}\qquad\gr(\cD'_1) +\cD_2'$$
belong to the same $\bG$-orbit. Yet by choosing $r$ sufficiently 
large we can place the first arbitrarily close to $\cD_1$ and the 
second arbitrarily close to $\cD_2$. This proves that the
quotient $\fM_{2h}=\wfD^\fs_{2h}/\bG$ is not a Hausdorff space. In fact, 
since $\cD_2$ can be arbitrarily close to $\cD_1$, it follows that $\fM_{2h}$
is not even locally Hausdorff at $\((\cD_1\))$. Furthermore,
since any divisor with $\max_j\{m_j\}>h$ can be approximated by one
with $\max_j\{m_j\}=h$, it follows that $\fM_{2h}$ is not locally
Hausdorff at any point with $\max_j\{m_j\}\ge h$.

The proof for $n=2h+1\ge 5$ is similar. For this case we take 
$$\cD_1=\cD_1'+(h+1)\langle\infty\rangle\quad{\rm and}
\quad\cD_2=\cD_2'+h\langle\infty\rangle~,$$
where $\cD_1'$ has degree $h\ge 2$, but $\cD_2'$ has degree $h+1\ge 3$.
It then follows as above that $\fM_{2h+1}$ is not Hausdorff. Again 
$\cD_1$ and $\cD_2$ can be arbitrarily close to each other:
Starting with any $\cD_1$, it is only
necessary to replace the point of multiplicity $h+1$ for $\cD_1$ 
by a point of multiplicity $h$, together with a nearby point of 
multiplicity one, in order to obtain an appropriate $\cD_2$. 
It follows easily that $\fM_{2h+1}$ is not locally
Hausdorff at any point with $\max_j\{m_j\}\ge h+1$.
This completes the proof of Lemma~\ref{L-D5}.
\end{proof}
\bigskip

The proof of Theorem \ref{T-D1} will also require a study of group elements
which are ``close to infinity'' in $\bG$ (or in other words, outside of a large 
compact subset of $\bG$). Choose some metric on $\bP^1$, for example the
 standard spherical metric, and let $N_\vep(\p)$ be the open $\vep$-neighborhood
of $\p$. \medskip

\begin{figure}[h!]
\centerline{\includegraphics[width=2.6in]{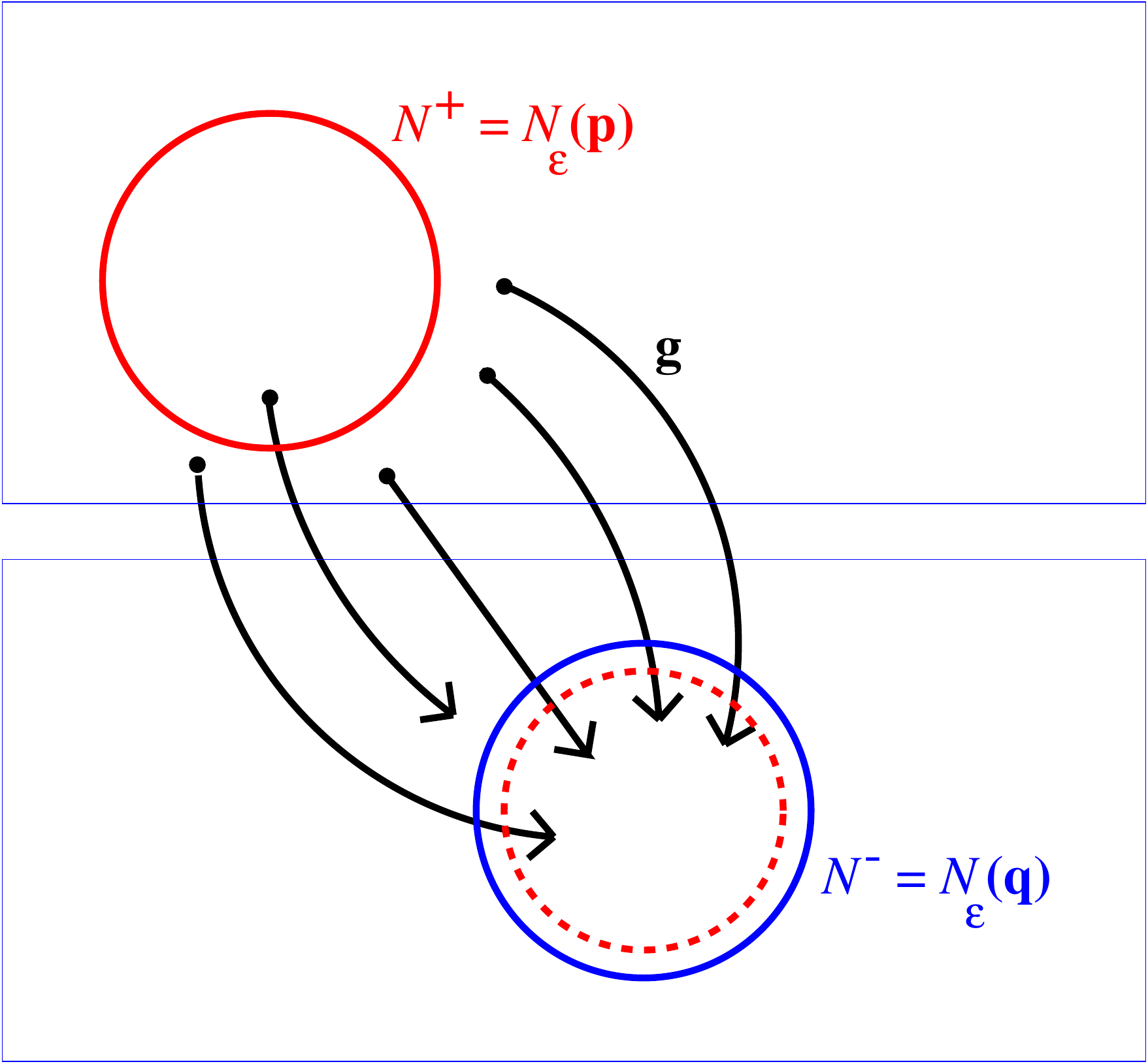}}
\caption{\sf Illustrating the Distortion Lemma. The two rectangles represent
copies of $\bP^1$. The image $\g\big(N^+\big)$
covers everything outside of the dotted circle. Hence everything outside
of $N^+$ must map into $N^-$.\label{F-dispic}}
\end{figure}
\medskip

\begin{lem}[{\bf Distortion Lemma for Automorphisms of $\bP^1$}]\label{L-dis1}
Given any $\vep>0$ there exists a compact set $K=K_\vep\subset \bG$
with the following property. For any $\g\in \bG\ssm K$ there exist two
$($not necessarily distinct$)$ open $\vep$-disks $~N^+=N_\vep(\p)~$
and $~N^-=N_\vep(\q)~$  such that
$$ \g\big(N^+\big)\,\cup\, N^-~=~\bP^1~.$$
\end{lem}
\medskip

It follows that $\g$ maps every point outside of $N^+$ into $N^-$. (Roughly
speaking, we can think of $N^+$ as a repelling disk and $N^-$ as an
attracting disk.) 

\begin{proof}[Proof of Lemma \ref{L-dis1}.]
First consider the corresponding statement for the group of diagonal 
automorphisms 
$$ \d_\kappa(x:y)~=~(\kappa x: y)~, \qquad {\rm with} \quad \kappa\, \in
\, {\mathbb C}\ssm \{0\}~,$$
or in affine coordinates with $z=x/y$,  $~~z\mapsto \kappa z$.
Interchanging the coordinates $x$ and $y$ if necessary, we may assume
that $|\kappa|\ge 1$. The condition that $\d_\kappa$ lies outside of a
large compact set then means that $|\kappa|$ is large. 
The proof can then easily be completed, choosing $\p=0$ and $\q=\infty$.
(For example, if $\kappa= 1/\vep^2$ then $\d_\kappa$ maps the small disk
$|z|<\vep$ onto the large disk $|z|<1/\vep$.)

The proof for the group of projective transformations $\bG$ then follows
immediately, using Lemma \ref{L-linalg}.
\end{proof}
\bigskip

\begin{proof}[Proof of Theorem~\ref{T-D1}] We will first show that 
$\fM_n$ is a ${\rm T}_1$-space. This means that every point of $\fM_n$ is a 
closed set; or equivalently that every $\bG$-orbit in $\wfD^\fs_n$ is 
a closed set. In other words, we must show that every limit point
of such an orbit within the larger space $\wfD_n$  either belongs to
the orbit or else has infinite stabilizer, so that it is outside of 
$\wfD^\fs_n$.

Given any $\cD\in\wfD^\fs_n$, 
let  $\cD'$ be a divisor which can be expressed as the limit
 $$~~\cD'=\lim_{j\to\infty} \g_j(\cD)~~$$ of points of the 
$\bG$-orbit $\((\cD\))$. If $\cD'$ itself
 does not belong to this $\bG$-orbit, then we will show that the
support $|\cD'|$ can have only one 
or two elements, so that  $\cD'\not\in\wfD^\fs_n$.

Choose $\vep>0$ small enough so that any two points of $|\cD|$ have distance
$\ge\,2\,\vep$,\; or in other words so that any $N_\vep(\p)$ can contain at 
most one point of $|\cD|$. The group elements $\g_j$ must tend to infinity in
$\bG$, since otherwise the limit point would be in the $\bG$-orbit 
$\((\cD\))$. Hence we can choose corresponding  $\vep_j$ tending to zero. 
For each $j$ with $\vep_j<\vep$, we can choose $\vep_j$-disks $N^+_j$ 
and $N^-_j$ as in the Distortion Lemma~\ref{L-dis1}, it follows that 
all but at most one point of $|\g_j(\cD)|$ must lie in the disk $N_j^-$.
Passing to the limit,  
it follows that $|\cD'|$ can have at most two points, as asserted. 

Thus it follows that $\((\cD\))$ is closed as a subset of $\wfD^\fs_n$,
and hence that $\bpi(\cD)$ is closed as a subset of $\fM_n$.\qed\bigskip

Finishing the proof of Theorem \ref{T-D1} will require one further
 preliminary step.
\smallskip

\begin{definition} Let $\fV_n$ be 
 the open set consisting of all divisors in $\wfD_n$ such that
$$ \max_j\{m_j\}~~<~n/2~.$$
\end{definition}
\smallskip

\begin{prop}\label{P-div} For any $n$,
the action of $\bG=\PGL_2$ on this open set \hbox{$\fV_n\subset\wfD_n$}
is proper.
\end{prop}
\smallskip

\begin{proof}
Consider two divisors $\cD_1\,~\cD_2\in\fV_n$.
We must construct 
neighborhoods $U_1$ of $\cD_1$ and $U_2$ of $\cD_2$ and a compact
set $K\subset \bG$ so that any $\g\in \bG$ which maps a point $\cD'_1\in U_1$
to a point $\cD'_2\in U_2$ must belong to~$K$.

Let $\vep$ be small enough so that any two distinct  
points of $|\cD_\ell|$ have distance
\begin{equation}\label{E-4eps}
{\rm dist}(\p,\,\p')~>~4\,\vep
\end{equation}
from each other, both for $\ell=1$ and 
for $\ell=2$. Thus no ball of radius $\vep$ can intersect more than 
one of the $\vep$ balls around the points of $|\cD_\ell|$.

Let $K_\vep$ be a corresponding compact subset of $\bG$, as described in  
the Distortion Lemma \ref{L-dis1}. 
Let $\fN_\vep(\cD_\ell)\subset\wfD_n$ 
be the neighborhood of $\cD_\ell$ consisting of all $\cD'_\ell\in\wfD_n$
such that, for each $\p\in|\cD_\ell|$, the number of points of $\cD'_\ell$ in 
$N_\vep(\p)$ counted with multiplicity, is precisely equal to the multiplicity
 of $\p$ as a point  of $\cD_\ell$. (In other words,
a point $\p$  of multiplicity $m_j$ for $\cD_\ell$ is allowed to split into as
many as $m_j$ distinct points in $\cD'_\ell$; but they are not allowed to 
move out of the $\vep$-neighborhood of $\p$.) \medskip

\begin{figure}[h!]
\centerline{\includegraphics[width=2.9in]{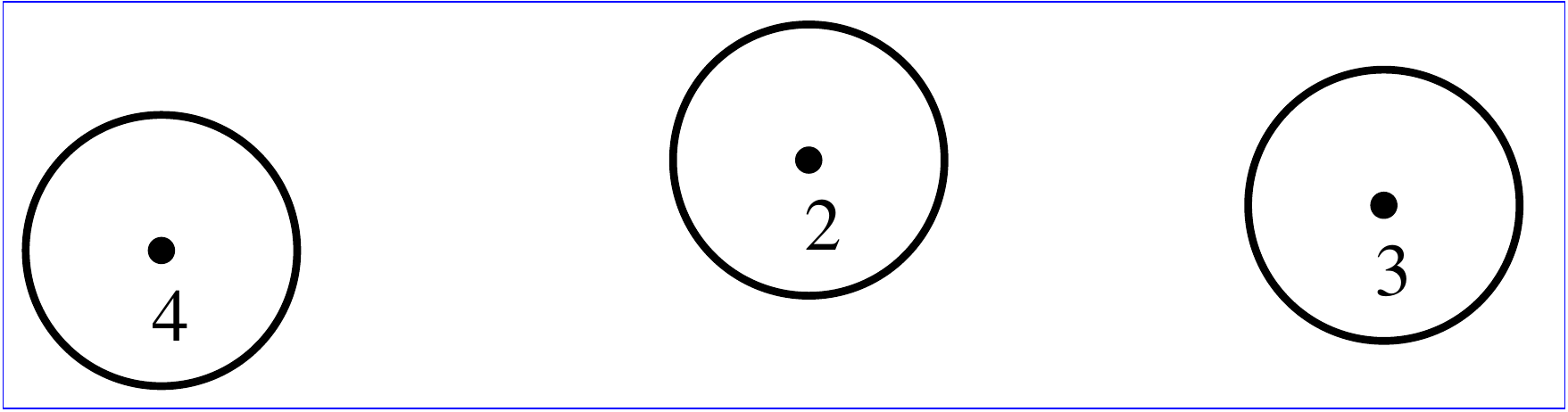}}\ssk
\centerline{\includegraphics[width=2.9in]{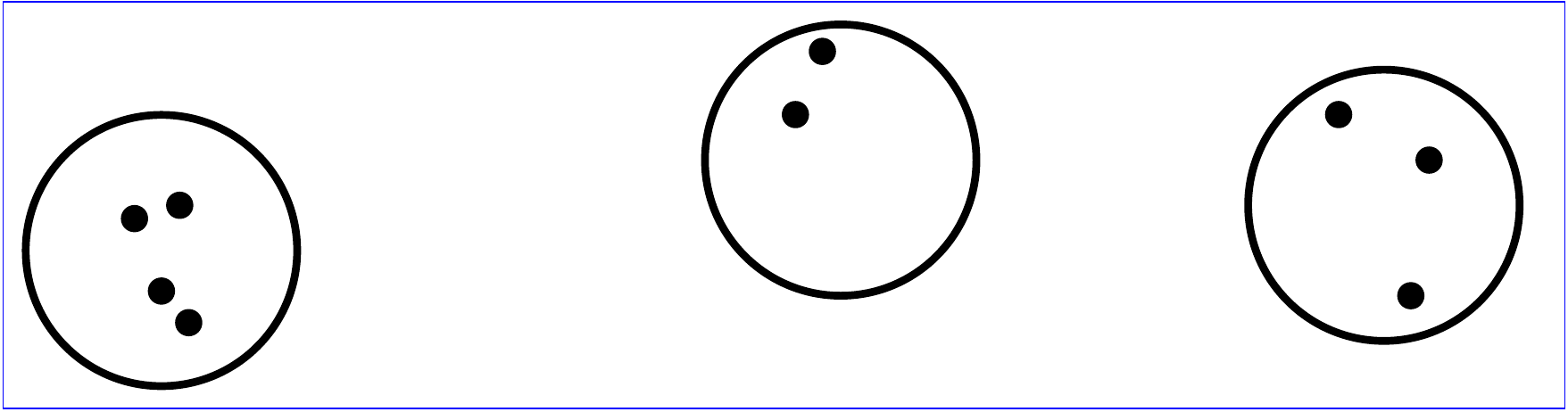}}\ssk
\centerline{\includegraphics[width=2.9in]{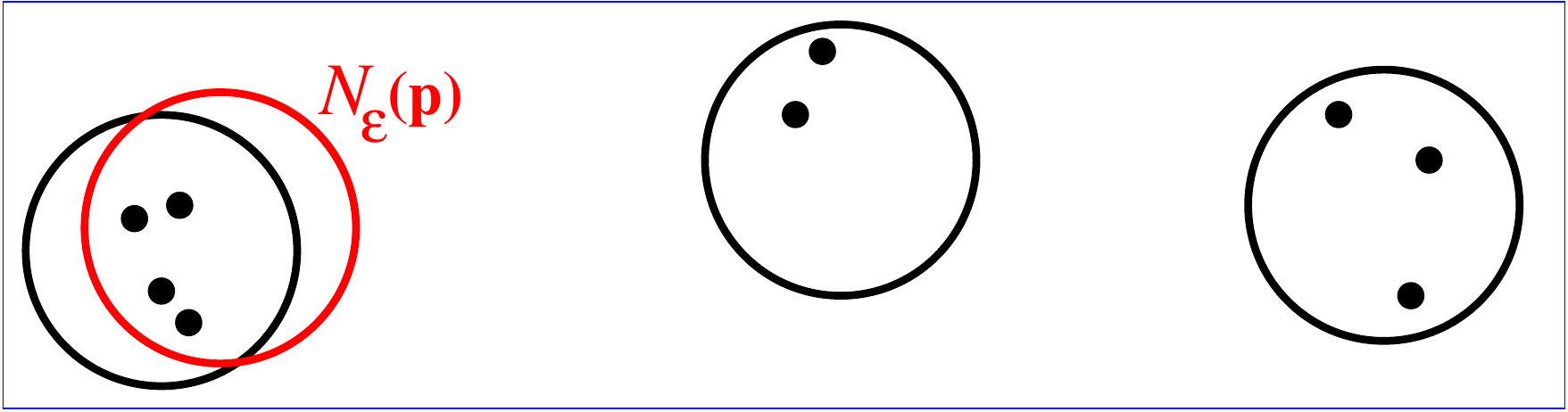}}\ssk
\centerline{\includegraphics[width=2.9in]{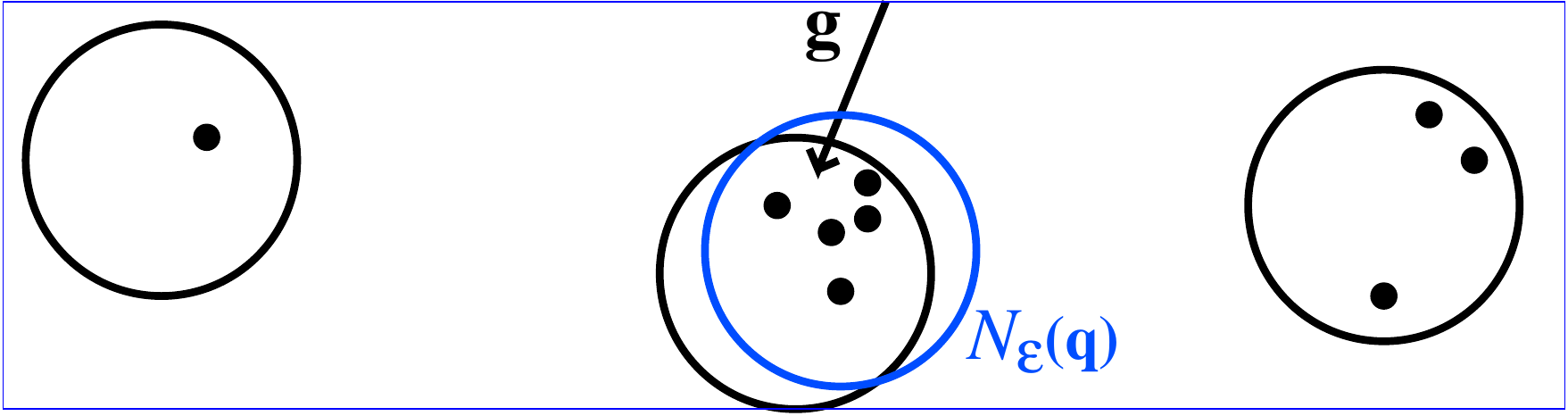}}\ssk
\caption{\sf The top frame illustrates a typical divisor $\cD_1\in\fV_9$, 
showing the  $\vep$-balls around points of multiplicity 4, 2, and 3. 
The next two frames illustrate a divisor $\cD'_1\in\fN_\vep(\cD_1)$, 
and the last frame illustrates a divisor $\cD'_2$ whose points are contained
in $\vep$-balls around the points of $\cD_2$.
Assuming that $\g\not\in K_\vep$ maps $\cD'_1$ to $\cD'_2$,
the last two frames show associated  balls  $N^+=N_\vep(\p)$ 
 and $N^-=N_\vep(\q)$
such that the five points outside of  $N^+$ all map into 
$N^-\!$, yielding a contradiction.\label{F-div2}}
\end{figure}
\medskip

We must prove that any $\g\in \bG$ which maps some
$\cD'_1\in \fN_\vep(\cD_1)$ to
a \hbox{$\cD'_2\in \fN_\vep(\cD_2)$} must belong to the compact set $K_\vep$.
Suppose to the contrary that   $\g\not\in K_\vep$. Then we could construct
corresponding $\vep$-balls  $N^+=N_\vep(\p_0$) and
  $N^-=N_\vep(\q_0)$, so that  $\g(N^+)\cup N^-=\bP^1$. 
Since any $\vep$-ball intersects at most one of the
$N_\vep(\p)$ with $\p\in|\cD|$, the ball  $N^+$
 must contain fewer than $n/2$ points of
$\cD'_1$, counted with multiplicity. Hence its complement must contain more
than $n/2$ such points. Since $\g$ maps the complement of  $N^+$ into
$N^-$, this means that $N^-$ contains more than $n/2$ points of $\cD'_\ell$,
counted with multiplicity. But this is impossible since  $N^-$ can 
intersect at most one of the $\vep$-balls around points of $|\cD_2|$.
This contradiction completes the proof of Proposition \ref{P-div}.
\end{proof} Since
proper action implies Hausdorff quotient, it also completes the proof of
Theorem \ref{T-D1}.
\end{proof}
\msk

\begin{proof}[Proof of Theorem \ref{T-D2} $($Compactness$)$]
 First consider the  case of even\break degree $n=2k\ge 6$. Consider the
 sequence of degree $n$ divisors
$$ \cD_h~=~\sum_{j=1}^k\big(\,\<j/h\> +\<jh\>\,\big)~,$$
which converges to $k\<0\>+k\<\infty\>$ as $h\to\infty$. We can spread out
either the summands $\<j/h\>$ or the summands $\<jh\>$ by suitable projective
transformations; but in either case there will always be $k$ summands tending
to a single point, so that the limit will not represent any point of 
$\fM_n^\ha$. Thus this moduli space is not compact.\ssk

To see that this cannot happen in the odd degree case $n=2k+1\ge 5$
we proceed as follows.
Using the standard spherical metric $$\frac{2\,|dz|}{1+|z|^2}~,$$ let
$0\le{\bf diam}(S)\le\pi$ denote the diameter of a set $S\subset\bP^1=
\C\cup\{\infty\}$. Define the function 
$\Theta:\wfD_n~\to~[0,\,\pi]$ as follows:

\begin{quote} \it Let $\Theta(\cD)$ be the smallest diameter among the 
finitely many sets\break $S\subset|\cD|$ which contain at least 
$k+1$ points, counted with multiplicity.
\end{quote}

\noindent Thus $\Theta(\cD)$ is zero
if and only if some point of $|\cD|$ has multiplicity $\ge k+1$.
In other words, \hbox{$\Theta(\cD)>0$} 
if and only if $\cD$ belongs to the set $\fV_n$
 consisting of divisors with \hbox{$\max_j\{m_j\}\le k$}, or in other words 
if and only $\bpi(\cD)\in\fM_n^\ha$.\ssk

We will need the  following.
\smallskip

\begin{lem}\label{L-comp}
Let $K_n\subset \wfD_n$ be the compact set consisting of divisors with
$$\Theta(\cD)\ge \pi/4~.$$ 
If $n=2k+1$, then every $\bG$-equivalence class $\((\cD\))$ with 
$\max_j\{m_j\}\le k$ has a representative $\g(\cD)$ which belongs 
to this set $K_n$. 
\end{lem}
\smallskip

\begin{proof}We will make use of the projective 
automorphism $z\mapsto \kappa\,z$ with $\kappa>1$,\break
 which is strictly distance increasing
when considered as a map from the disk\break
 $|z|<\kappa^{-1}$ to the larger disk $|z|<1$.
To prove this, it suffices to prove the equivalent statement 
that the inverse map $w\mapsto\lambda w$ with $\lambda=\kappa^{-1}<1$
is strictly distance decreasing as a map from $|w|<1$ to the smaller disk
 $|w|<\lambda$. It is not hard to show that this inverse map multiplies 
infinitesimal distances near $w$ by a factor of 
\begin{equation}\label{E-exp-fact}
\frac{\lambda\big(1+|w|^2\big)}{1+|\lambda w|^2}~~.
\end{equation}
This expression is strictly less than one, as one can check by
 multiplying the inequality
$~\lambda|w^2|<1~$ by $~1-\lambda~$, then rearranging terms to get
$~\lambda(1+|w^2|)<1+\lambda^2|w^2|~$, and dividing.

It follows that every curve in the region $|w|<1$ maps to  a shorter
curve, and hence that all distances are decreased. Therefore, all
distances are increased by the inverse map from $|z|<\lambda$ to $|z|<1$.
\ssk

Given any $\bG$-equivalence class $\((\cD\))\subset\fV_n$, we can
 choose a  representative $\cD'=\g(\cD)$ which maximizes the value of 
$\Theta(\cD')$. We will prove that this maximum value must satisfy
$\Theta(\cD')\ge\pi/4$. Suppose, to the contrary, that this
maximum value $\Theta(\cD')$ were
less than $\pi/4$. Then consider those  subsets of $|\cD'|$ which

\begin{itemize}
\item[(1)] have at least $k+1$ points, counted with multiplicity, and 

\item[(2)] have \hbox{diameter~$<\pi/4$}. 
\end{itemize}

\noindent After a unitary change of
coordinates, we may assume that one of these sets contains the point $z=0$.
Since any two of these sets must intersect, they will all lie in the ball
of radius $\pi/2$ centered as the origin, or in other words within the 
open set $|z|<1$. Therefore they will 
all lie within the region  $|z|\le 1-\vep$ for some $\vep>0$.
Hence we can increase all of the
distances between their points by choosing the expansion
 $z\mapsto 
\kappa\,z$ with $\kappa=1+\vep$. This contradicts the construction of $\cD'$, 
and completes the proof of Lemma \ref{L-comp}. 
\end{proof}\ssk

The image of the compact set $K_n$ under the continuous map 
\hbox{$\bpi:\fV_n\to \fM^\ha_n$} 
must itself be compact. Since it follows from Lemma \ref{L-comp}  
that this image  is the entire space $\fM^\ha_n$, this
 completes the proof of Theorem \ref{T-D2}.
\end{proof}\msk 

\begin{rem}[{\bf Trivial Stabilizers}] \label{R-triv} 
For $n\ge 5$, a generic divisor \hbox{$\cD\in\fD_n$} has trivial stabilizer. 
To see this, let us temporarily work with the space $\fDno$ consisting of 
ordered $n$-tuples $\vp=(\p_1,\ldots,\p_n)$ of distinct points in $\bP^1$. 
Every such $\vp$ determines a  corresponding divisor  
$$\cD(\vp)~=~ \<\p_1\>+\cdots+\<\p_n\>~;$$
and every \hbox{4-element}  subset 
\hbox{$\Sigma=\{ j_1,\,\ldots,\,j_4\}\subset\{1,\,\ldots,\,n\}$} determines
 a corresponding degree four divisor 
$$\cD(\vp,\,\Sigma)=\<\p_{j_1}\>\,+\cdots+\<\p_{j_4}\>~.$$
 Let $\bJ(\vp,\Sigma)\in\C$ be the shape invariant associated with this divisor.
Then for each $\Sigma\ne\Sigma'$ the equation
$$ \bJ(\vp,\Sigma)~=~ \bJ(\vp,\Sigma') $$
determines a proper algebraic subvariety of $\fDno$. The complement
of the union of these finitely many subvarieties is a dense open subset of
 $\fDno$; and any element in this dense
open set corresponds to a divisor \hbox{$\<\p_1\>+\cdots+\,\<\p_n\>$} which has
 trivial stabilizer. This follows since the stabilizer of $\cD(\vp)$ consists 
of all  permutations of  $\{1,\ldots,\,n\}$ which map $\cD(\vp)$ to itself. But
 any non-trivial permutation must clearly map some four point subset to a 
different four point subset.
\end{rem}\msk

\begin{rem}[{\bf Compactification of $\cM^{\,\sf un}_{0,n}$}]\label{R-DM}
Recall from the beginning of this section that 
$\cM^{\,\sf un}_{0,n}=\cM_{0,n}/\fS_n$ is defined to be  the moduli space for 
unordered\break
\hbox{$n$-point} subsets of the Riemann sphere. (Evidently
$\cM^{\,\sf un}_{0,n}$
can be identified with the quotient $\fD_n/{\bf G}$.)
It is interesting to compare
the Deligne-Mumford\footnote{By this we mean the compactification using
methods developed by Deligne, Mumford, and also Knudsen (based on ideas of
Grothendieck). As far as we know, Deligne and Mumford never studied this
particular family of examples. }
compactification  $\overline\cM^{\,\sf un}_{0,n}$ of this space with 
our moduli space  $\fM_n=\widehat\cD_n/{\bf G}$. In order to describe
 $\overline\cM^{\,\sf un}_{0,n}$ we will need the following.
Let $\F$ stand for either $\R$ or $\C$. 
\smallskip

\begin{definition}\label{D-TS} By an (unordered)
\textbf{\textit{tree-of-marked-spheres}} 
(or \textbf{\textit{circles}} in the real case), we will mean a space ${\bT}$ 
 which is the union $\bS_1\cup\cdots\cup\bS_k$ 
of one or more copies ${\bS}_j$ of the  projective line $\bP^1(\F)$,
together with  a finite subset of $\bT$ which we will call the set of 
\textbf{\textit{marked points}}. We require:

\begin{enumerate}
\item[{\bf(1)}] that each marked point belongs to only one of the $\bS_j$;

\item[{\bf (2)}] that each non-empty intersection $\bS_i\cap\bS_j$ with $i\ne j$
must consist of a single point, which will be called a \textbf{\textit{nodal
point}};
\item[{\bf(3)}] that each $\bS_j$ must contain at least three points 
which are either marked or nodal,\footnote{In the complex case, this 
means that $\bS_j$ with
these points removed must be a hyperbolic Riemann surface.} and that

\item[{\bf(4)}] the abstract graph with one vertex for each $\bS_j$ and one
 edge for each nodal point should be a tree; that is, it must be 
 connected and acyclic.
\end{enumerate}

\noindent Two such trees-of-marked-spheres
are \textbf{\textit{isomorphic}} if there is a homeomorphism between the
underlying spaces $\bT$ and ${\bT}'$ which preserves the marked 
points and which is fractional linear on each $\bS_j$.
By definition, the \textbf{\textit{degree}} $n\ge 3$ of a 
tree-of-marked-spheres is the total number of marked points.
\end{definition}\msk

\begin{figure}[h!]
\centerline{\includegraphics[width=3.1in]{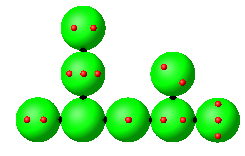}}

\caption{\sf Showing a tree-of-marked-spheres of degree 15. 
Each red dot represents a marked point.\label{F-ST} 
(For the count of 15, it
 is assumed that there are no such dots on the back sides of the spheres.)}
\end{figure}

\medskip

 By definition, each point of the \textbf{\textit{Deligne-Mumford 
compactification}} $\overline\cM^{\,\sf un}_{0,n}$ corresponds
to a unique isomorphism class of trees-of-marked-spheres of degree $n$.
(Compare \cite[\bf Ar2]{Ar1}.)   
In the complex case, we can think of this construction intuitively as follows. 
Starting with a tree-of-marked-spheres $\bT$, if
we remove a small round neighborhood of each intersection point and glue
the resulting boundary circles together, then we obtain a Riemann surface
of genus zero with $n$ marked points, corresponding to a nearby
element of  $\cM^{\sf un}_{0,n}$. Conversely, suppose that
 we start from a surface of 
genus zero with $n\ge 3$ punctures, provided with its natural hyperbolic
metric. If this surface has one or more very short closed geodesics, then we
can obtain a ``nearby'' tree-of-marked-surfaces of genus zero, by 
replacing each such geodesic by a single point. 
\msk

\begin{figure}[h!]
\begin{minipage}{0.33\textwidth}
$$\xymatrix{&\hbox to 0pt{\hspace{-.65in}\includegraphics[width=1.3in]{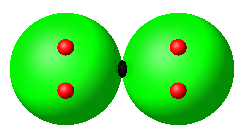}}\ar[ld]\ar[rd]\\&\relax&\\
&\hbox to 0pt{\hspace{-.75in}\includegraphics[width=1.5in]{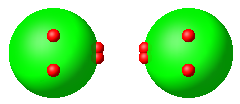}}&}$$
\end{minipage}\begin{minipage}{0.33\textwidth}
$$\xymatrix{&\hbox to 0pt{\hspace{-.65in}\includegraphics[width=1.3in]{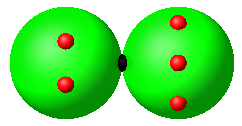}}\ar[ld]\ar[rd]\\ &\relax&\\
&\hbox to 0pt{\hspace{-.75in}\includegraphics[width=1.5in]{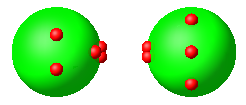}}&}$$
\end{minipage}
\begin{minipage}{0.33\textwidth}
$$\xymatrix{&\hbox to 0pt{\hspace{-.65in}\includegraphics[width=1.3in]{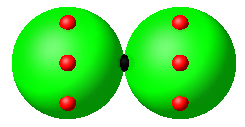}}\ar[ld]\ar[rd]\\ &\relax&\\
&\hbox to 0pt{\hspace{-.75in}\includegraphics[width=1.5in]{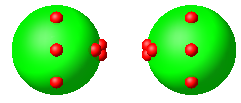}}&}$$
\end{minipage}
\caption{\sf \label{f-DMpic} The top row represents three examples of
 trees-of-marked-spheres,
corresponding to points in the compactification 
$\overline\cM^{\,\sf un}_{0,n}$ for $n$ equal to 
4, 5 and 6 respectively. Corresponding to each sphere ${\bS}_j$ in $\bT$,
there is a canonical
retraction map from $\bT$ to ${\bS}_j$, 
indicated by an arrow in the figure, which maps the $n$
marked points of $\bT$  to a divisor $\cD_j\in \wfD_n$.
Here each former intersection point is to be weighted by the number of
free marked points which map to it. 
These image divisors are shown in the bottom row. (Here a point of 
multiplicity two or three  is indicated schematically by a cluster
 of two or three overlapping dots.)}
\end{figure}
\medskip

There is a ``many-valued map'' from $\overline\cM^{\,\sf un}_{0,n}$ to 
$\fM_n$ defined as follows. (Compare Figure~\ref{f-DMpic}.) 
For each sphere ${\bS}_j\subset{\bT}$, 
there is a unique continuous 
retraction \hbox{${\bf r}_j:{\bT}\to {\bS}_j$} which maps every point of 
${\bS}_j$ to itself, and which maps each ${\bS}_i$ with $i\ne j$ to
a single intersection point in ${\bS}_j$. The collection of all free 
marked points of ${\bT}$ corresponds under ${\bf r}_j$
 to a divisor $\cD_j$ of degree $n$ on the copy 
$\bS_j$ of $\bP^1$, and hence on the standard $\bP^1$. \medskip

Note that the non-Hausdorff nature of $\wfD_n$ is closely related to 
this construction. In fact, whenever the spheres ${\bS}_j$ and ${\bS}_k$
have an intersection point, it is not hard to see that the corresponding
divisors $\cD_j$ and $\cD_k$ in $\wfD_n$ represent points of $\fM_n$ which
 do not have disjoint neighborhoods.
For example, our proof in Lemma \ref{L-D5}
that $\fM_5$ and $\fM_6$ are not Hausdorff makes
use precisely of the  divisors illustrated in Figure~\ref{f-DMpic}. However, 
in the degree four case, the
corresponding divisors $\cD_1$ and $\cD_2$ actually represent
the same point of $\fM_4$, which is a Hausdorff space.

It is not too difficult  to prove that $\overline\cM^{\,\sf un}_{0,n}$ can be 
identified with
$\fM_n$ for the cases $n=3$ and $n=4$, and with $\fM^\ha_n$ for $n=5$.
Similarly $\fM^\ha_6$ can be identified with an open subset of 
$\overline\cM^{\,\sf un}_{0,6}$. However there seems to be no 
such direct relationship when $n\ge 7$.
\end{rem}\ssk

\begin{rem}[{\bf The Ordered Moduli Space $\cM_{0,n}$}]\label{R-oms}
We can learn more about $\Mbar^{\sf un}_{0,n}$ by noting that it is 
equal to the
quotient of the compactification $\Mbar_{0, n}$ of the moduli space
for ordered $n$-tuples by the action 
of the symmetric group $\fS_n$.  
 This compactification $\overline\cM_{0,n}$
is a beautiful object which was introduced by Knudsen \cite{Knu}, based on ideas
of Grothendieck, Deligne, and Mumford. It has been studied by many
authors, and is well understood. 

We can work over either the real numbers or the complex numbers. The
construction of this compactification in terms of trees of marked spheres
(or marked circles in the real case) 
is just like the description of $\Mbar^{\sf un}_{0,n}$ as given above,
except that the $n$ marked points must now be given $n$ distinct labels,
using for example the integers between $1$ and $n$.

First consider the complex case.
(In Mumford's terminology, such a tree of labeled spheres is called
a ``stable curve of genus zero''.) Knudsen showed that the compactification
$\Mbar_{0,n}(\C)$
can be constructed out of a smooth variety by iterated blow-ups, and hence
that it is a smooth complex variety. (For an alternative proof,
using cross-ratios to embed $\Mbar_{0,n}(\C)$ smoothly  
in a product of many spheres, see \cite{MS1}.) It follows immediately that 
$\Mbar_{0,n}(\R)$ is also a smooth manifold, since it is just the fixed 
point set for complex conjugation on the complex manifold.

The topology of $\Mbar_{0,n}(\C)$ has been studied by Keel \cite{Ke}. He
showed for example that
these manifolds are simply-connected, with homology only in even dimensions, 
and with no torsion.
For the simplest cases: $\Mbar_{0,3}$ is a point; $\Mbar_{0,4}$ 
is a copy of the sphere $\bP^1$; and $\Mbar_{0,5}$ is the connected sum of 
one copy of $\bP^2$ with its standard orientation, together with four copies
with reversed orientation.

The topology of $\Mbar_{0,n}(\R)$ has been studied by Etingof, Henriques, 
Kamnitzer and Rains \cite{EHKR}, who showed
for example that there is an isomorphism of mod two
cohomology rings of the form
$$H^k\big(\Mbar_{0,n}(\R);~\Z/2\big)~\cong
~H^{2k}\big(\Mbar_{0,n}(\C);~\Z/2\big)~.$$
The manifold $\Mbar_{0,n}(\R)$ is non-orientable with a non-abelian fundamental
group for all  $n>4$.\ssk

\begin{figure}[h!]
\centerline{\includegraphics[width=4.1in]{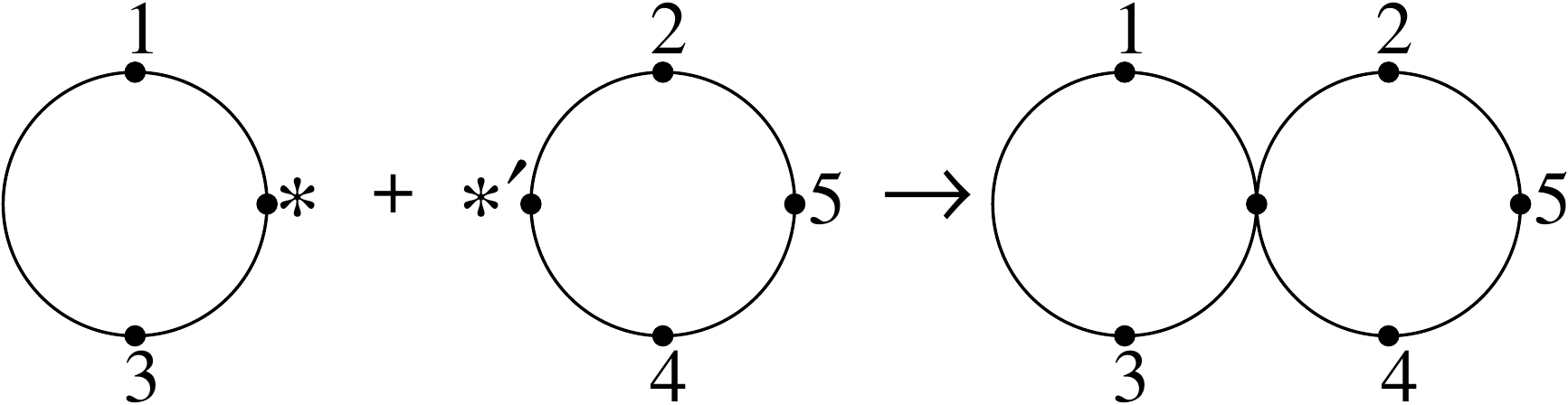}}

\caption{\sf  Illustrating the map (\ref{E-tree-emb}) on the subset
$\cM_{0,p+1}\times\cM_{0, q+1}$. \label{F-treesum}}
\end{figure}
\ssk

Consider a partition of the index set
$\{1,\,2,\,\ldots,\,n\}$ into two disjoint subsets $I$ and $J$, where $I$ has 
$p\ge 2$ elements and $J$ has $q\ge 2$ elements.
 In either the real or the complex case,
there is an associated embedding  
\begin{equation}\label{E-tree-emb}
\Mbar_{0,\,p+1}\times \Mbar_{0,\,q+1}~~~\hookrightarrow~~~\Mbar_{0,n}~.
\end{equation}
where $n=p+q$. If we restrict this map to the dense open subset
$\cM_{0,p+1}\times\cM_{0,q+1}$ then the associated trees of labeled spheres
can be described quite explicitly as illustrated in Figure \ref{F-treesum}.
Label the first sphere by the elements of $I$ together with one additional 
element $*$. Similarly, label the second sphere by the elements of $J$ together
with one additional element $*'$. Now construct the third tree by
gluing $*$ onto $*'$.
\ssk

Every element of the ideal boundary $\Mbar_{0,n}\ssm \cM_{0,n}$ is contained
in the image of at least one such embedding (\ref{E-tree-emb}). Furthermore
(using mod two coefficients in the real case), every homology class except 
in the top dimension is a sum of classes which come from these embedded
submanifolds.

\begin{figure} [t]
\centerline{\includegraphics[width=2.7in]{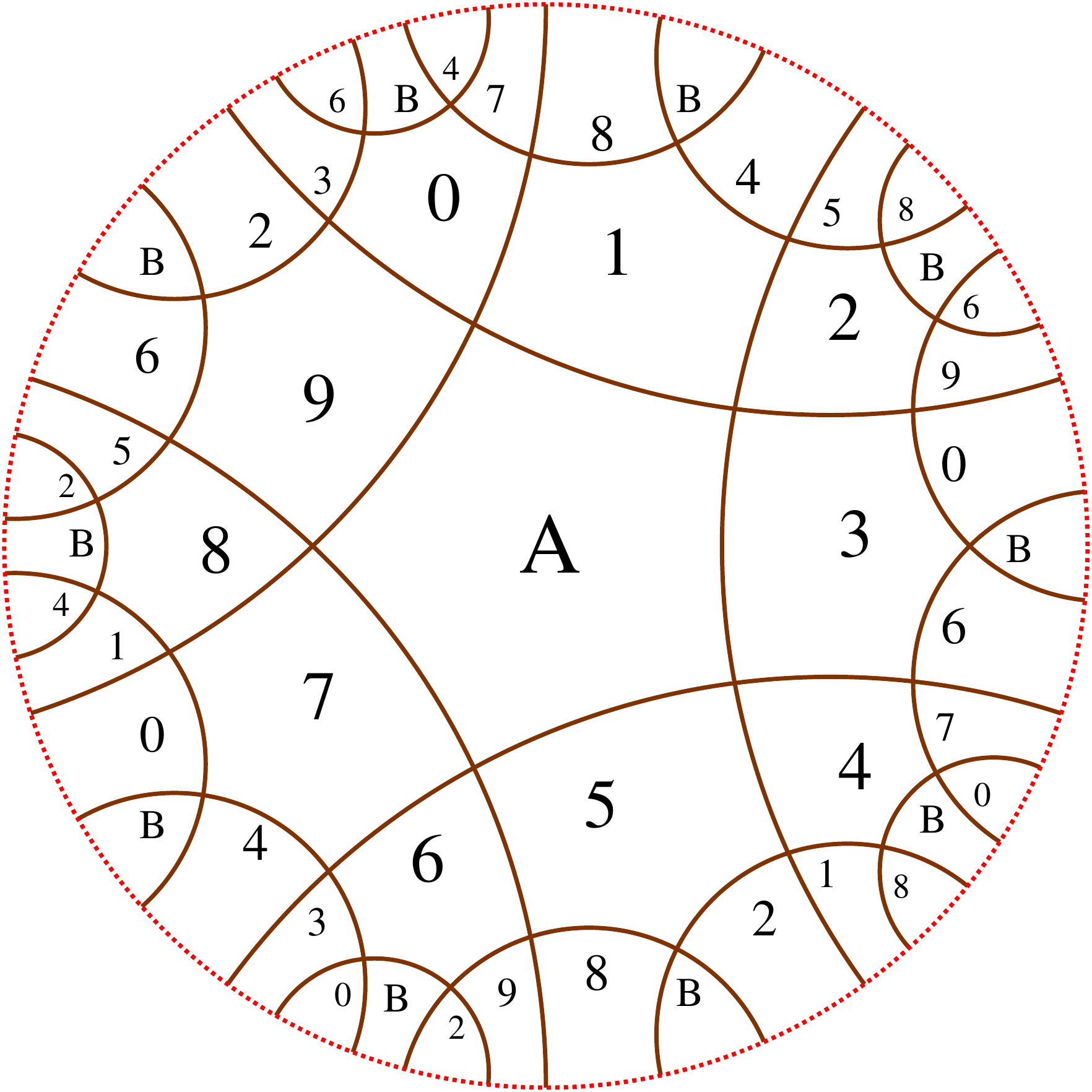}}

\caption{\sf Universal covering space of the ``hyperbolic dodecahedron''
$\Mbar_{0,5}(\R)$. Here regions with the same label correspond to a common
pentagon in $\Mbar_{0,5}(\R)$. \label{F-dodec}}
\end{figure}
\bigskip

{\bf Examples.\;} There are three distinct ways of partitioning $\{1,2,3,4\}$
into subsets with two elements. Correspondingly, there are three distinct
ways of embedding the point $\Mbar_{0,3}\times\Mbar_{0,3}$ into the circle or
2-sphere $\Mbar_{0,4}$. The complement of this set of three points is
the dense open subset $\cM_{0,4}$.

Similarly, there are ten ways of 
partitioning $\{1,2,3,4,5\}$ into subsets of order two and three, and 
correspondingly ten ways of embedding $\Mbar_{0,3}\times \Mbar_{0,4}
\cong \Mbar_{0,4}$ into $\Mbar_{0,5}$.
In the real case, the ten embedded circles divide this surface
into twelve pentagons, which
represent the twelve connected components of $\cM_{0,5}$. 
 (Compare Figure \ref{F-dodec}.) This surface 
can be given a hyperbolic metric, so that these circles are geodesics.
Like the standard dodecahedron in Euclidean 3-space, the surface admits a
group of 120 
isometries such that any isometry from one pentagon to another extends uniquely
to a global isometry. (However, the two isometry groups are not isomorphic.)
Like the standard dodecahedron, this surface has 12 faces and 30 edges; but it 
has only 15 vertices, so that the Euler characteristic is $12-30+15=-3$.
\ssk

\begin{figure}[h]
\centerline{\includegraphics[width=2.7in]{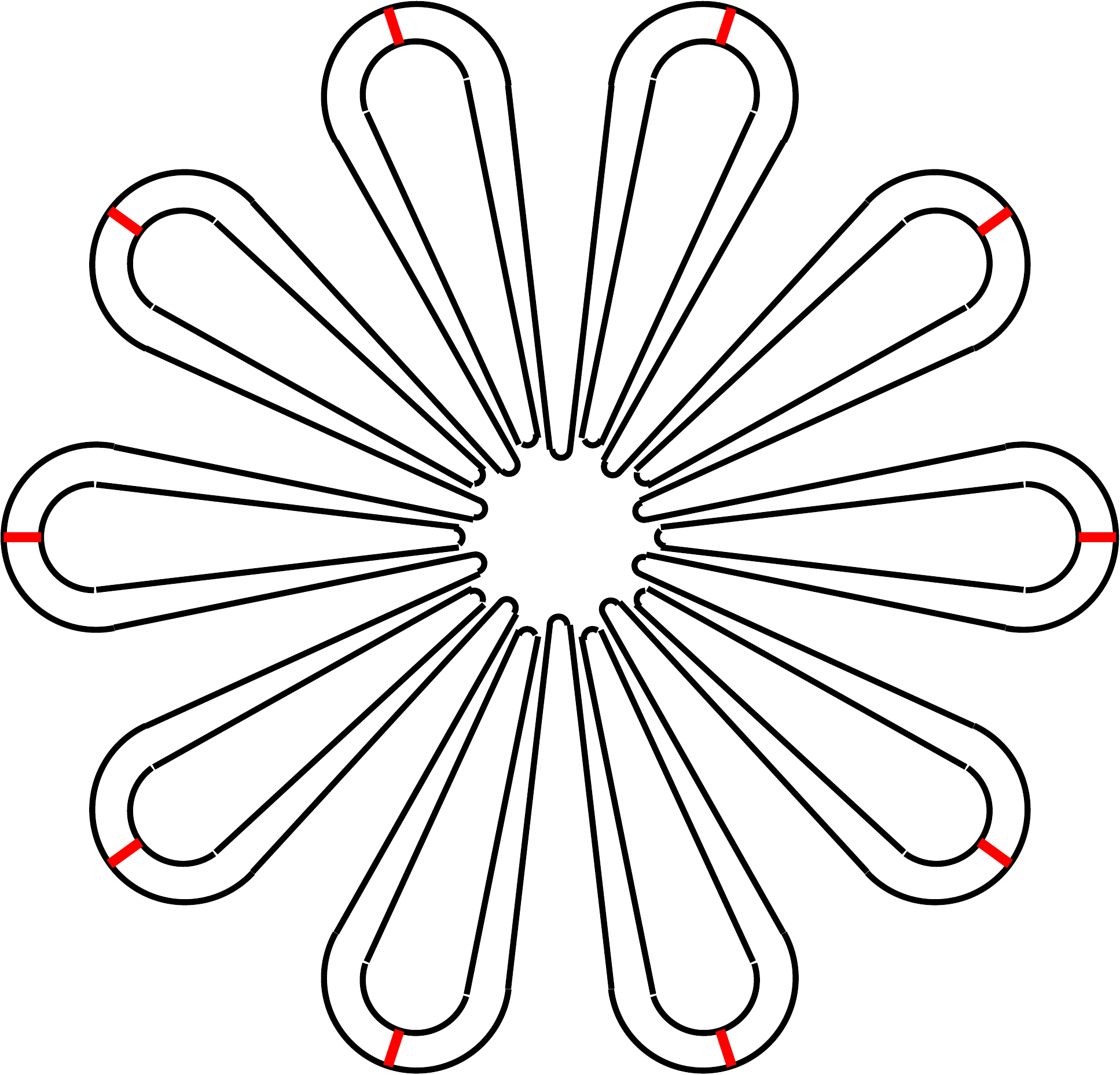}}

\caption{\sf A cartoon of the 3-manifold $\Mbar_{0,6}(\R)$. If we cut along the
ten tori (represented by short transverse line segments) then the remainder
can be given the structure of a complete hyperbolic manifold with 20 infinite
cusps. \label{F-M6}}
\end{figure}

The space $\Mbar_{0,6}(\R)$ is an interesting example for 3-manifold theory.
The ten partitions of $\{1,2,3,4,5,6\}$ into two subsets with three
elements each yield
ten embeddings of the torus $\Mbar_{0,4}\times\Mbar_{0,4}$ into $\Mbar_{0,6}$.
If we remove these ten tori, then the remainder can be given the structure of
a complete hyperbolic manifold of finite volume.\footnote{For further details,
see\; \url{http://math.stonybrook.edu/~jack/scgp18.pdf}.}
This is an example of a JSJ-decomposition, as first introduced by Jaco 
and Shalen \cite{JS} and by Johanssen \cite{J}. 
See Figure \ref{F-M6} for a cartoon of the resulting 3-manifold.
Using this decomposition,
it is not hard to check that the fundamental group 
$\pi_1\big(\Mbar_{0,6}(\R)\big)$ maps onto a free group on ten generators.
However, this fundamental group also contains copies of the
abelian group $\Z\oplus\Z$ corresponding to any one of the tori.

This manifold is highly symmetric, with a group of 720 automorphisms. 
In fact any $\Mbar_{0,n}$ with $n\ge 5$ has the full symmetric group on 
$n$ elements as a group of automorphisms.
\end{rem}\ssk

\bigskip

\setcounter{lem}{0}
\section{Moduli Space for Real or Complex
 Plane Curves.}\label{s-mod}

This section will be an outline of
basic definitions and notations for curves
or 1-cycles  of arbitrary degree $n$. 
(For more detailed 
 discussion of particular degrees, see \S\ref{s-deg3} and \S\ref{s-cc}.)
\ssk

Let $\F$ be either $\R$ or $\C$. 
It will be convenient to define an \textbf{\textit{irreducible curve 
of degree $n\ge 1$ over $\F$}}
as an equivalence class of irreducible homogeneous polynomials $\Phi(x,y,z)$
of degree $n$ with coefficients in $\F$, where two such polynomials are 
equivalent if one is a non-zero constant multiple of the other.  
Thus in the complex case two irreducible curves are equal if and only if 
they have the same zero locus
$$ \{(x:y:z)\,\in\,\bP^2(\C)~;~\Phi(x,y,z)=0\}~.$$
However, in the real case, the analogous zero locus in $\bP^2(\R)$ is no 
longer a complete invariant.
As an example, 
we will consider the equivalence class of $x^2+y^2+a\,z^2$ 
to be an irreducible real curve for each $a>0$, even though  the
corresponding zero locus
$$ x^2+y^2+a\,z^2~=~0$$ 
in $\bP^2(\R)$ is the empty set. In such cases, it is necessary to
look at the corresponding complex zero locus in order to distinguish
between two different irreducible real curves.
\ssk

By definition,
an \textbf{\textit{effective 1-cycle}} of degree $n$ over $\F$  is a formal sum
$$\cC~=~m_1\cdot\cC_1+\cdots+m_k\cdot\cC_k~, $$
where the $\cC_j$ are distinct irreducible curves defined over $\F$,
and where
the multiplicities $m_j$ are strictly positive integer coefficients with
$$n~=~\sum_{j=1}^k m_j\cdot{\sf degree}(\cC_j)~. $$
\ssk

The vector space consisting of all
polynomials $\Phi:\F^3\to \F$ which are homogeneous of degree $n$
has a basis consisting of the $n+2\choose 2$ monomials
$x^iy^jz^k$ with $i+j+k=n$. Each such $\Phi$
factors as a product of powers $\Phi_j^{m_j}$ of irreducible polynomials, 
which are uniquely defined up to a non-zero constant factor,
and hence corresponds to a unique effective 1-cycle over $\F$.
It follows that:

\begin{quote} \it The space $\widehat\fC_n(\F)$ consisting
of all effective 1-cycles 
of degree $n$ can be given the structure of a
 projective space of dimension 
$${n+2\choose 2}-1 ~=~ \frac{n(n+3)}{2}$$ over $\F$.
\end{quote}

\noindent
 This space $\widehat\fC_n$
 is known as the \textbf{\textit{Chow variety}} for 1-cycles
   of degree $n$. (Compare \cite[p.~272]{Harr}.)
The disjoint union
$$ \{0\}\,\sqcup\, \wfC_1\,\sqcup \,\wfC_2\,\sqcup\, \wfC_3\,\sqcup\,\cdots $$
of the $\wfC_n$ can be described as the free additive semigroup with one 
generator for each irreducible curve over $\F$.\ssk

For any 1-cycle $~\cC=\sum m_j\cdot\cC_j$ over $\F$, the zero set
$$ |\cC|~=~|\cC|_\F~=~\big\{(x:y:z)\in\bP^2(\F)~~;~ \Phi(x,y,z)=0\big\}~~=~
~|\cC_1|\cup\cdots\cup|\cC_k|~\subset~\bP^2(\F)$$
will be called the \textbf{\textit{support}} of $\cC$. Note that
 this definition  ignores multiplicities.
In the real case, it is also useful to consider the complex support 
 $|\cC|_\C\subset \bP^2(\C)$, which provides more information.
\ssk

The word \textbf{\textit{curve}} will be reserved for a 1-cycle
$$\cC~=~ \cC_1\,+\,\cdots\,+\,\cC_k~, $$
where the $\cC_j$ are distinct irreducible curves, and where of the
multiplicities are $+1$. In the complex case, a curve is uniquely
determined by its support. In both the real and complex cases, the space
$\widehat\fC_n$ consisting of all 1-cycles of degree $n$  
can be thought of as a compactification of the dense open subset 
$\fC_n$ consisting of degree $n$ curves. As one example, for each $a\ne 0$
the equation $x^2+a\,y^2=a\,z^2$ defines a smooth quadratic curve. But as 
$a$ tends to zero this curve converges to the 1-cycle $2\cdot{L}$,
where $L$ is the line $x=0$. \ssk

Let  $\bG=\PGL_3(\F)$ be the 8-dimensional Lie group consisting
 of all projective automorphisms 
of the projective plane $\bP^2=\bP^2(\F)$, as discussed in Remark
\ref{R-PGL}. 
Each such projective automorphism acts\footnote{Caution: If $\cC$ is 
defined by the equation $\Phi(x:y:z)=0$,
then $\g(\cC)$ is defined by  $$\Phi\circ\g^{-1}(x:y:z)=0~.$$}
on the space of curves in $\bP^2$,
and hence acts on the space $\wfC_n$ of effective 1-cycles of degree $n$. 
The notation $\((\cC\))\subset
\wfC_n$ will be used for the orbit of $\cC$ under the action of $\bG$.

\begin{definition}\label{D-stabi}
The subgroup $\bG_\cC\subset \bG$ consisting of all automorphisms
$\g\in \bG$ which map $\cC$ to itself,
is called the \textbf{\textit{stabilizer}} of $\cC$.
In the case of a smooth complex curve of degree $n>1$, it can be identified
with the group of all conformal automorphisms of $\cC$ which extend to 
projective automorphisms of $\bP^2$. (Here the condition $n>1$ is needed 
to guarantee that there is at most one such extension.)
\end{definition}
\smallskip

\begin{definition}\label{D-var}
An \textbf{\textit{algebraic set}} (or Zariski closed set)
 in a projective space (such as
in $\wfC_n$) over $\F$ will mean any subset
defined by finitely many homogeneous polynomial equations with 
coefficients in $\F$. Any algebraic set can be expressed uniquely as a 
union of maximal irreducible algebraic subsets.
\end{definition}

\begin{rem} \label{R-stab} If the stabilizer $\bG_\cC$ is  infinite, then it
cannot be described as an algebraic set in a projective space.
However it can be described
as the difference of two algebraic sets in the projective space $\bP^8$. 
First note that the group $\bG=\PGL_3$
can be identified with the complement $\bP^8\ssm V$, where $\bP^8$ is the 
projective space of lines through the origin in the space of $3\times 3$
matrices, and $V$ is the algebraic subset corresponding to singular $3\times 3$
matrices. Then,  
whether or not $\bG_\cC$ is finite,  it is easy to see that it can be 
described as an algebraic set intersected with this open variety $\bP^8\ssm V$.

It follows  that $\bG_c$
is either a finite group, or a Lie group with finitely many connected
components. The distinction between these two possibilities 
will be of fundamental importance in our discussion. 
Note that a 1-cycle $\cC$ has infinite stabilizer 
if and only if its $\bG$-orbit  \hbox{$\((\cC\))\subset \wfC_n$} 
has dimension strictly less than $\dim(\bG)=8$. In fact there is a natural
fibration with projection map $\g\mapsto \g(\cC)$ from $\bG$ to the subset
\hbox{$\((\cC\))\subset \wfC_n$} with fiber $\bG_\cC$. It follows that
\begin{equation}\label{E-orbitdim}
\dim(\bG_\cC)~+~\dim\((\cC\))~=~\dim(\bG)=8~.
\end{equation}\end{rem}
\msk

\begin{definition}\label{D-W}
Let $\fW_n$ be the subset of $\wfC_n$ consisting of all 1-cycles 
with infinite stabilizer.  
We will be particularly concerned with the complementary set
$$ \wfC^\fs_n~=~\wfC_n\ssm \fW_n~,$$
consisting of 1-cycles with  \textbf{\textit{finite stabilizer}}. 
The quotient space 
$$\M_n~=~\M_n(\F)~=~\wfC^\fs_n/\bG$$ 
consisting of all projective equivalence classes of
 such 1-cycles will be called the \textbf{\textit{moduli space}} for
 1-cycles of degree $n$ over $\C$. Thus each element $\((\cC\))\in\M_n$ is an
 equivalence class of effective 1-cycles in $\wfC^\fs_n$, where two 1-cycles
 $\cC$ and  $\cC'$
are equivalent if and only if $\g(\cC)=\cC'$ for some $\g\in \bG$,
or in other words if and only if they belong to the same
\textbf{\textit{$\bG$-orbit}} in the space of 1-cycles.
\end{definition} \smallskip

To begin the discussion, we will show
 that $\fW_n$ is a closed subset of $\wfC_n$, or equivalently that its 
complement $\wfC^\fs_n$ is an open set. Choose some metric on $\wfC_n$.
 If $\cC$ has finite stabilizer, then we can choose a small
sphere centered at the identity element of $\bG$ and an $\vep>0$ so that
 for every $\g$ in this sphere the image $\g(\cC)$ has distance greater than 
$\vep$ from $\cC$. Since this is an open condition, it will also be satisfied 
for any $\cC'$ which is sufficiently close to $\cC$.\qed 
\bsk

See \S\ref{s-aut} for a detailed study of $\fW_n$, including a proof that
it is an algebraic set.\medskip

\def\fR{{\mathfrak R}}

Another important algebraic set, at least in the complex case, is the 
locus ${\mathfrak R}_n$ consisting of all reducible  1-cycles in $\wfC_n$.
By definition, a 1-cycle is  \textbf{\textit{reducible}}
if and only if it is  in the image of the smooth map
$$(\cC,\,\cC')\mapsto \cC+\cC'\quad{\rm from}\quad\wfC_k\times\wfC_{n-k}
\quad{\rm to}\quad\wfC_n $$
for some $0<k<n$. 
In the complex case, each such image is an irreducible variety,
and it follows that $\fR_n$ is an algebraic set. 
However, in the real case
the best one can say is that $\fR_n$ is a closed
{\it semi}-algebraic set, defined
by polynomial equalities {\it and inequalities}.\footnote{See \cite{BCR}.}
(As an example the 1-cycle $x^2+2bxy+y^2=0$ in $\bP^2(\R)$
is reducible if and only if $|b|\ge 1$, in which case it is
 equal to a union of two lines through the point $(0:0:1)$. In the case $|b|<1$
this 1-cycle is irreducible (with vacuous real zero locus $|\cC|_\R$).

 We will be particularly interested in the
complementary open set 
$$\fC^\irr_n=\wfC_n\ssm{\mathfrak R}_n~\subset~ \fC_n$$
consisting of irreducible 1-cycles of degree $n$. 
Evidently every irreducible 1-cycle has all multiplicities equal to one,
and hence is actually an irreducible curve. \ssk 

Note that the dimension function $~n\mapsto \dim(\wfC_n)=n(n+3)/2~$ 
is convex, in the sense that $~\dim(\wfC_n)>\dim(\wfC_k)+\dim(\wfC_{n-k})~$
for $0<k<n$. Here are a few values:
\begin{equation}\label{E-dims}
\begin{matrix} n& 1&2&3 &4&5&\cdots\\ 
\dim(\wfC_n)&2&5&9&14&20&\cdots&,\end{matrix}\end{equation}
where it follows by an easy computation that $\dim(\wfC_n)=\dim(\wfC_{n-1})+n+1$.
It also follows easily that the dimension of the space of reducible cycles is
$$\dim({\mathfrak R}_n)~=~\dim(\wfC_{n-1})+2~=~\dim(\wfC_n)-n+1~.$$
\smallskip

\setcounter{lem}{0}

\section{Curves of Degree Three.}\label{s-deg3}
This section will study the moduli spaces $\M_3(\R)$ and $\M_3(\C)$ 
for curves (or \hbox{1-cycles}) of degree three. At first
we will not distinguish between the real and complex cases, 
since the arguments are exactly the same for both.

Note that $n=3$ is the first case where $\M_n\ne\emptyset$. In fact,
for $n<3$ it follows easily from Equation
 (\ref{E-orbitdim}) and Table (\ref{E-dims}) that
 there are no \hbox{1-cycles} with finite stabilizer; so that
 the corresponding moduli space $\M_n$ is empty.
(More precisely, it follows that every stabilizer 
must have dimension at least 6  when $n=1$, and at least 3 when $n=2$.)

Similarly, since $\dim({\mathfrak R}_3)=7<8$, it follows that: 
\begin{quote}{\it Every 
reducible curve or 1-cycle of degree 3 has infinite stabilizer.
Thus, when studying $\M_3$, it suffices to work with the open subset 
$\fC_3^\irr\subset\fC_3$ consisting of irreducible curves.}\end{quote}

\noindent Note that any irreducible curve of degree three 
(or more generally of odd 
degree) is uniquely determined by its real support $|\cC|_\R$. 
In fact any curve of odd degree has many real points, since every real line 
intersects it in at least one point. (However, a reducible curve such as
$~x(y^2+z^2)=0~$ may not be determined by its real support.)
\smallskip

The following statement is surely
 well known, although we don't know any explicit reference in this generality.

\begin{prop}\label{P-snf}
If the field $\F$ is either $\C$ or $\R$, then a cubic curve
 \hbox{$\cC\subset\bP^2(\F)$} is irreducible if and only if 
it can be transformed, by an $\F$-projective change of coordinates,
to the standard normal form, which can be written in affine coordinates
$(x:y:1)$ as
\begin{equation}\label{E-snf}
y^2~=~x^3+a\,x+b~.
\end{equation}

Furthermore:
\begin{itemize}
\item[{\bf(a)}] This associated normal form is unique up to the transformation
\begin{equation}\label{E-rescale}
 (a,\,b)~~\mapsto ~~(t^4 a,\, t^6 b) \end{equation}
where $t$ can be any non-zero element of $\F$.

\item[{\bf(b)}] The curve defined by 
 $(\ref{E-snf})$ has finite stabilizer 
if and only if $(a,\,b)\ne(0,0)$.

\item[{\bf(c)}] This reduction to normal form can be  carried out uniformly over
some neighborhood  ${\mathcal U}$ of any given $\cC_0$. That is, there
 is a smooth map $~\cC \mapsto {\bf g}_\cC~$ from 
${\mathcal U}$ to $\PGL_3(\F)$ so that for any $\cC\in {\mathcal U}$ the
 automorphism
${\bf g}_\cC :\bP^2\to\bP^2$ maps $\cC$ to a curve in standard normal form.
\end{itemize}
\end{prop}

 The proof will depend on the following.

\begin{lem}\label{L-1flex}
Every irreducible real or complex cubic curve contains at least one flex point.
\end{lem}

\begin{proof}[Proof of Lemma \ref{L-1flex}] In the complex case, we will
make use of the classical Pl\"ucker formulas,
which compare an irreducible curve $\cC\subset\bP^2(\C)$ with its dual 
curve  $\cC^*\subset\Pdual$. 
(See for example \cite{Nam} or \cite[p.278]{GH}.) 
Here $\bP^{2\textstyle{*}}$ is the dual complex projective  plane consisting 
of all lines  in $\bP^2$; and $\cC^*\subset \bP^{2\textstyle{*}}$ is the 
subset consisting of lines which are tangent to $\cC$. 

Consider a curve $\cC$ of degree $n$ in $\bP^2=\bP^2(\C)$ 
with no singularities other than  simple double points 
and cusps, and with only simple flex points and bitangent lines (that is lines
which are tangent at two different points). These conditions are certainly
satisfied  in the cubic case. (In particular, a cubic curve can have no 
bitangents.)  Let:

\begin{description}
\item{ $f$} be the number of flex points, 

\item{ $\delta$} the number of double points,

\item{ $\kappa$} the number of cusp points, and 

\item{$b$} the number of bitangents; 
\end{description}

\noindent and let $f^*,\,\delta^*,\,~\kappa^*$, and $b^*$
be the corresponding numbers for the dual curve. Then
$$ f=\kappa^*~\Longleftrightarrow~ f^*=\kappa\qquad{\rm and}\qquad \delta=b^*
\Longleftrightarrow\quad \delta^*=b~.$$
Furthermore the degrees $n^*$ of the dual curve and $n$ are given by 
the formulas: 
\begin{eqnarray*}
 n^*&=&n(n-1)-2\delta-3\kappa ~,\\
 n&=&n^*(n^*-1)-2b-3f~.
\end{eqnarray*}
Now let us specialize to the case $n=3$.
Note that $\delta+\kappa\le 1$, since otherwise it would follow 
from the last equation and its dual that $n^*\le 2$ hence $n\le 2$. Thus
there are only three possible cases to consider.
\medskip

\noindent In the smooth case $\delta=\kappa=0$ these equations yield: 
$$n^*=6\qquad{\rm and}\qquad n=3=30-3f~;
\quad{\rm hence~there~are}\quad f=9\quad{\rm flex~points}~.$$
If there is a simple double point so that $\delta=1$ and $\kappa=0$, they yield:
$$n^*=6-2=4\quad{\rm and}\quad 3=12-3f\quad{\rm with}\quad f=3
\quad{\rm flex~points}. $$
In the case of a cusp point, with $\delta=0$ and $\kappa=1$, they yield:
$$n^*=6-3=3\quad{\rm and}\quad 3=6-3 f\quad{\rm with}\quad
f=1\quad{\rm flex~point}~.$$
Thus the number of flex points is odd in all cases.

It follows that every real irreducible cubic must have at least one flex 
point. In fact, the associated complex curve must also be irreducible.
(Otherwise it would have a real factor.) 
 Since the non-real flex points occur in
complex conjugate pairs, there must be at least one real flex point;
which proves Lemma \ref{L-1flex}.
\end{proof}
\medskip

\begin{proof}[Outline Proof of Proposition \ref{P-snf}]  (Compare
\cite[\S3]{BM}.)
Given a single flex point, choose a projective transformation which
moves this  point to $(0:1:0)$, and moves the tangent line to $\cC$
at this point to the 
line $z=0$.  The tangent line has a triple intersection point
with the curve $\cC$ at $(0:1:0)$; hence it can have
 no other intersection point.
This means that the defining equation for $\cC$ must now consist of an $x^3$
term, plus other terms which are all
 divisible by $z$. Furthermore, it must include an $y^2z$ term,
since otherwise all of its partial derivatives at $(0:1:0)$ would be zero.
After switching to affine coordinates $(x:y:1)$, and after multiplying
 $x$ and $y$ by suitable constants, the equation will have the form
$$ y^2~=~x^3 +a\,x+b\quad{\rm plus~terms~in}
\quad x^2,~xy\,,\quad{\rm and}\quad y~.$$
However, the last two terms can  be eliminated by adding a
suitable $p\,x+q$ to the $y$ coordinate, and the  $x^2$ term can then be 
eliminated by adding a suitable constant to the $x$ coordinate.
\end{proof}

\medskip

\begin{figure}[t]
\centerline{\includegraphics[width=2.5in]{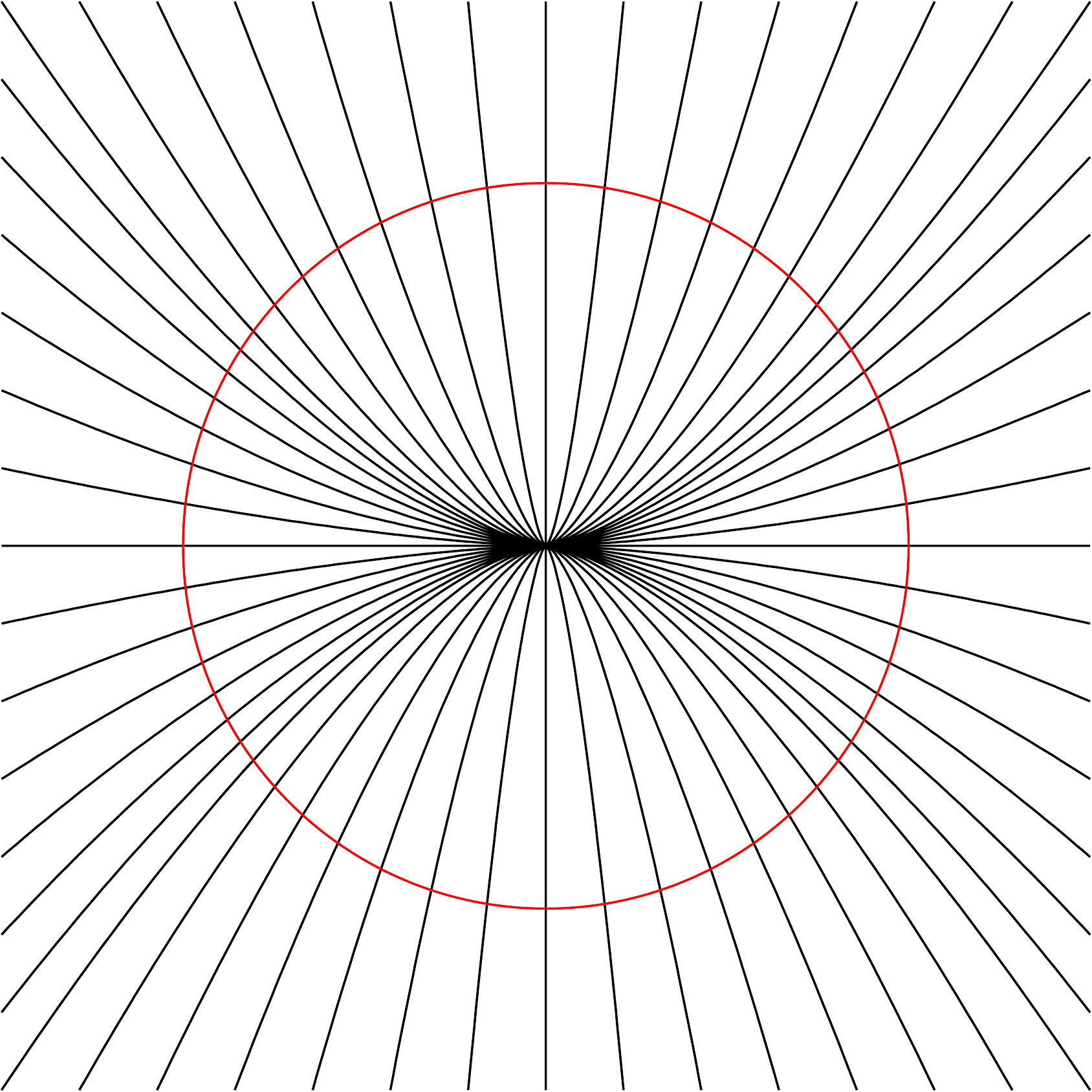}}
\caption{\label{F-abplane}\sf Foliation of the real $(a,b)$-plane,
with the origin removed, by the connected components of the
curves $(a^3:b^2) ={\rm constant}$. More explicitly, each curve can 
be parametrized as $a=a_0t^2$ and $b=b_0t^3$ with $t>0$, 
where $(a_0,\,b_0)$ can be any point on one of these curves.
The unit circle has a single transverse intersection
with each curve.}

\bigskip
\end{figure}
\bigskip 

Using Proposition~\ref{P-snf}, it follows easily that the quotient space 
$\fC_n^\irr/\bG$ of irreducible curves  modulo the action of $\bG$ 
can be identified with the quotient of the plane consisting of
 pairs $(a,\,b) \in\F^2$, by the equivalence relation
\begin{equation}\label{E-eq} 
(a,\,b)\sim(t^4a,\,t^6b)\qquad{\rm for~any}\quad t\ne 0~.\end{equation}
 (Compare  Figure~\ref{F-abplane}.) This entire quotient space has a rather
nasty topology, since the center point $a=b=0$, corresponding to the cusp curve
$$ y^2~=~x^3,$$
belongs to the closure of every other point. However, we eliminate this
problem by considering only curves with finite stabilizer,  and hence 
removing the center point.
\smallskip

\begin{theo}\label{T-M3}
If the field $\F$ is either $\R$ or $\C$, then the moduli space
$$\M_3(\F)~=~\fC^\fs_3(\F)/\bG(\F)   $$ 
can be identified with the quotient of the punctured $(a,\,b)$-plane
$\F^2\ssm\{(0,\,0)\}$ by the equivalence relation  $(\ref{E-eq})$.
In the real
case this moduli space is homeomorphic to the unit circle; while in the 
complex case it is homeomorphic to the topological 2-sphere consisting of all
 ratios $(a^3:b^2)\in\bP^1(\C)$.
\end{theo}

\begin{proof} It is easy to check
 that every orbit in $\R^2\ssm\{(0,0)\}$ intersects the unit circle
 $a^2+b^2=1$ exactly once.
Therefore, the moduli space $\M_3(\R)$ is homeomorphic to the unit circle.
In the complex case, each orbit corresponds to a fixed ratio
$(a^3:b^2)\in\bP^1(\C)$. (Equivalently, this ratio is captured by the shape
invariant
$$ \bJ~=~ \frac{4a^3}{4a^3+27 b^2}~\in\widehat\C $$
of Remark \ref{R-J}.)
 Thus $\M_3(\C)$ is homeomorphic to the
Riemann sphere \hbox{$\bP^1(\C)\cong\widehat\C$}.
\end{proof}

In fact $\M_3(\R)$ has a natural real analytic structure, and is real 
analytic\-ally diffeomorphic to the unit circle. Similarly $\M_3(\C)$ is
naturally biholomorphic to $\bP^1(\C)$. (See Remark~\ref{R-QAS}.)
 However, in both cases we will see  that there is just one improper
(but weakly proper) point, corresponding to the equivalence class of curves
with a simple self-crossing point. In
the complex case there is also a non-trivial orbifold structure 
with two ramified points; but in the real case, there are no ramified points.
\bigskip 

\begin{lem}\label{L-M3C}
The moduli space $\M_3(\C)$ for complex cubic curves is canonically isomorphic
to the moduli space $\fM_4(\C)$ for divisors of degree four. In particular
both spaces have one point with ramification index two, and one point
with ramification index three, as well as one improper point, which is
non-the-less weakly proper.
\end{lem}

\begin{proof} Given four
distinct points on the Riemann sphere, there is a unique 2-fold covering curve, 
branched at these four points. In fact, if we place one of the four points at 
infinity, and let $f(x)$ be the monic cubic polynomial with roots at the three 
finite points, then the locus
$$ y^2=f(x)$$
in the affine plane extends to the required 2-fold covering; and it follows from
Proposition \ref{P-snf} that every smooth cubic curve can be put in this form.
Furthermore, as two of the four points come together, the corresponding cubic 
curve will tend to a curve with a simple double point. Any symmetry
fixing the point at infinity will give rise to a corresponding symmetry
of the cubic curve, so that there is a precise correspondence between
ramified points for divisors of degree four and for curves of degree three.

[However, this does not mean that the corresponding stabilizers are isomorphic.
We have seen that a generic divisor of degree four has stabilizer 
$\Z/2\oplus\Z/2$. On the other hand, the stabilizer of a generic cubic curve
is the dihedral group of order eighteen. (See the remarks following
Theorem~\ref{T-noaut}.) For 
 the two ramified points, with ramification index two (or three), the order of
 the stabilizer is 8 (or 12) for divisors, but 36 (or 54) for cubic curves.]

 Now consider the degenerate
 case where two points of the divisor come together, or where the
cubic curve acquires a self-crossing point. Then
 the stabilizer has order two in
the divisor case,  and order 6 in the cubic curve case. (Compare
\cite[Figure~10]{BM}, where the six symmetries of one real form of
this curve are generated by 120$^\circ$ rotations, and reflections
on the vertical axis.)
By Lemma~\ref{L-prop}, it follows that the class of the complex cubic
 curve with a self-crossing  point is not  even locally proper.

However, the action is weakly  locally proper. We can write the singular
 curve in standard normal form for example as $y^2=x^3+2\,x-3$.
It follows from Proposition~\ref{P-snf}(c) that any curve which is close to this
 singular curve can be reduced
to standard normal form  by an automorphism
close to the identity. Furthermore,
if two such curves belong to the same $\bG$-orbit, then
it follows easily from Proposition~\ref{P-snf}(a) that we can transform one
to the other by an automorphism close to the identity.
\end{proof}
\medskip

\subsection*{\bf The Circle of Real Cubic Curves}
Now consider the real case. We have shown that the space $\M_3(\R)$
of projective equivalence classes
of cusp-free irreducible real cubic curves is diffeomorphic to the unit circle.
In fact each such curve-class has a unique representative with equation of
 the form
\begin{equation}\label{E-cnf}
y^2=x^3+ax+b\qquad{\rm with}\qquad a^2+b^2=1~,\end{equation}
However, this ``unit circle'' normal form seems somewhat arbitrary. For
example, if we used
a circle of different radius, then we would get a quite different 
parametrization.\smallskip

Note that the natural map from $\M_3(\R)$ onto the real part of $\M_3(\C)$ is
definitely not one-to-one. In fact it is precisely two-to-one.
Two real curves of the form $y^2=x^3+ax+b$ and $y^2=x^3+ax-b$ with $b\ne 0$
are not real projectively equivalent; yet they have the same ratio
$(a^3:b^2)$ and hence are complex projectively equivalent.\smallskip

We know of two quite natural ways of mapping $\M_3(\R)$ homeomorphically
(but not diffeomorphically) 
onto the real projective line $\bP^1(\R)$. One is given by representing 
each curve-class by its unique Hesse normal form
\begin{equation}\label{E-He}
 x^3+y^3+z^3~= 3\,k\,x\,y\,z ~.\end{equation}
(Compare \cite[Theorem 6.3 and Figure 10]{BM}.)
 This works beautifully for $k$ finite and different from $+1$, and yields 
precisely the open subset of $\M_3(\R)$ 
consisting of smooth curve-classes. However,
the limits as $k$ tends to $+1$ or $\pm\infty$ are badly behaved, and yield
reducible curves.\footnote{We could get around this, and obtain the correct 
curve-class by taking the limit of carefully rescaled curves.}

Another possible choice can be described as follows.
\smallskip

\begin{theo}[{\bf Flex-Slope Normal Form}]\label{T-FS} For any irreducible 
real cubic curve
 $\cC\subset\bP^2(\R)$, the following three properties are
equivalent:
\begin{enumerate}
\item[{\bf(a)}] $\cC$ is either smooth, or else smooth except at one isolated
 point.

\item[{\bf(b)}] $\cC$ contains three flex points.

\item[{\bf(c)}] $\cC$ is equivalent,
under a real projective change of coordinates, to one and only
one curve ${\mathcal F}(s)$
in the ``flex-slope'' normal form:
\end{enumerate}
\begin{equation}\label{E-FS}
y^2~=~x^3+(s\,x+1)^2~,
\end{equation}
using affine coordinates,\footnote{In homogeneous coordinates, the
equation is $y^2z=x^3+x(s\,x+z)^2$.}
where  $s$ can be any real number.
\end{theo}
\smallskip

Note that every curve in the form (\ref{E-FS}) has a flex point of  
slope $s$ at  $(0,\,1)$, as well as a flex point of slope $-s$ at $(0,\,-1)$.
In fact  it follows easily from (\ref{E-FS}) that 
$$y~=~\pm(1+s\,x)\,+\,O(x^3)\qquad{\rm as}\quad x\to 0~,$$
so that  the points $(0,\,\pm 1)$ satisfy $dy/dx=\pm s$ and $d^2y/dx^2=0$.\ssk

See Figure \ref{F1} for some typical examples. Note that the first four
curves in this figure are connected, while the last three have two components.
(The first and 
last curves would be much larger if drawn to scale: they have been shrunk
to fit in the picture. 
If we rescale
by setting $x=s^2X$ and $y=s^3Y$ so that $Y^2=X^3+(X+1/s^3)^2$,
 then the curve would converge to 
 $Y^2=X^3+X^2$ as $s\to \pm\infty$, with a self-crossing singular point at
 $X=Y=0$,  as shown in Figure \ref{F-alf}.)

\begin{figure}[t!]
\centerline{\includegraphics[height=1.5in]{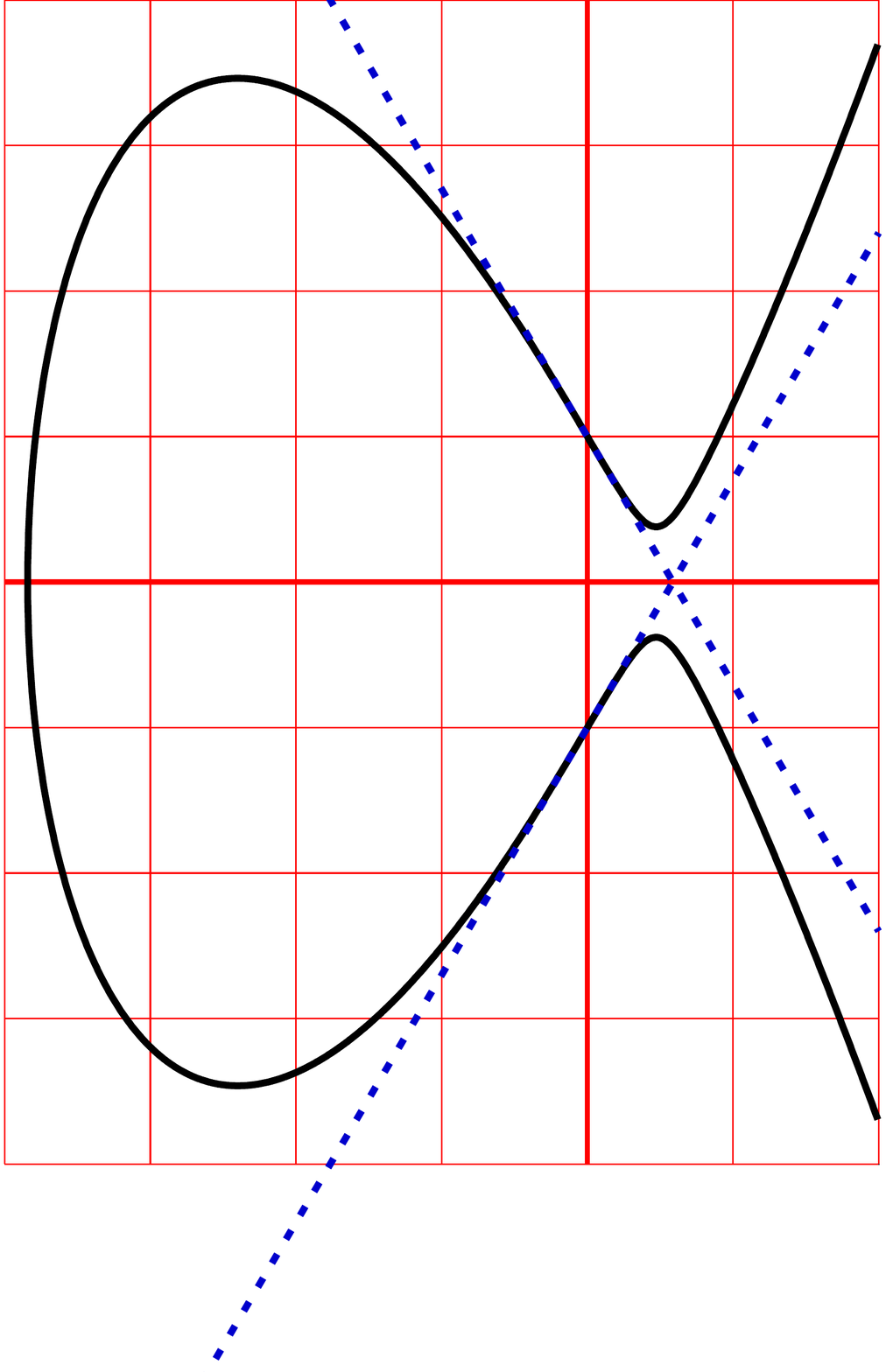} 
\includegraphics[height=1.5in]{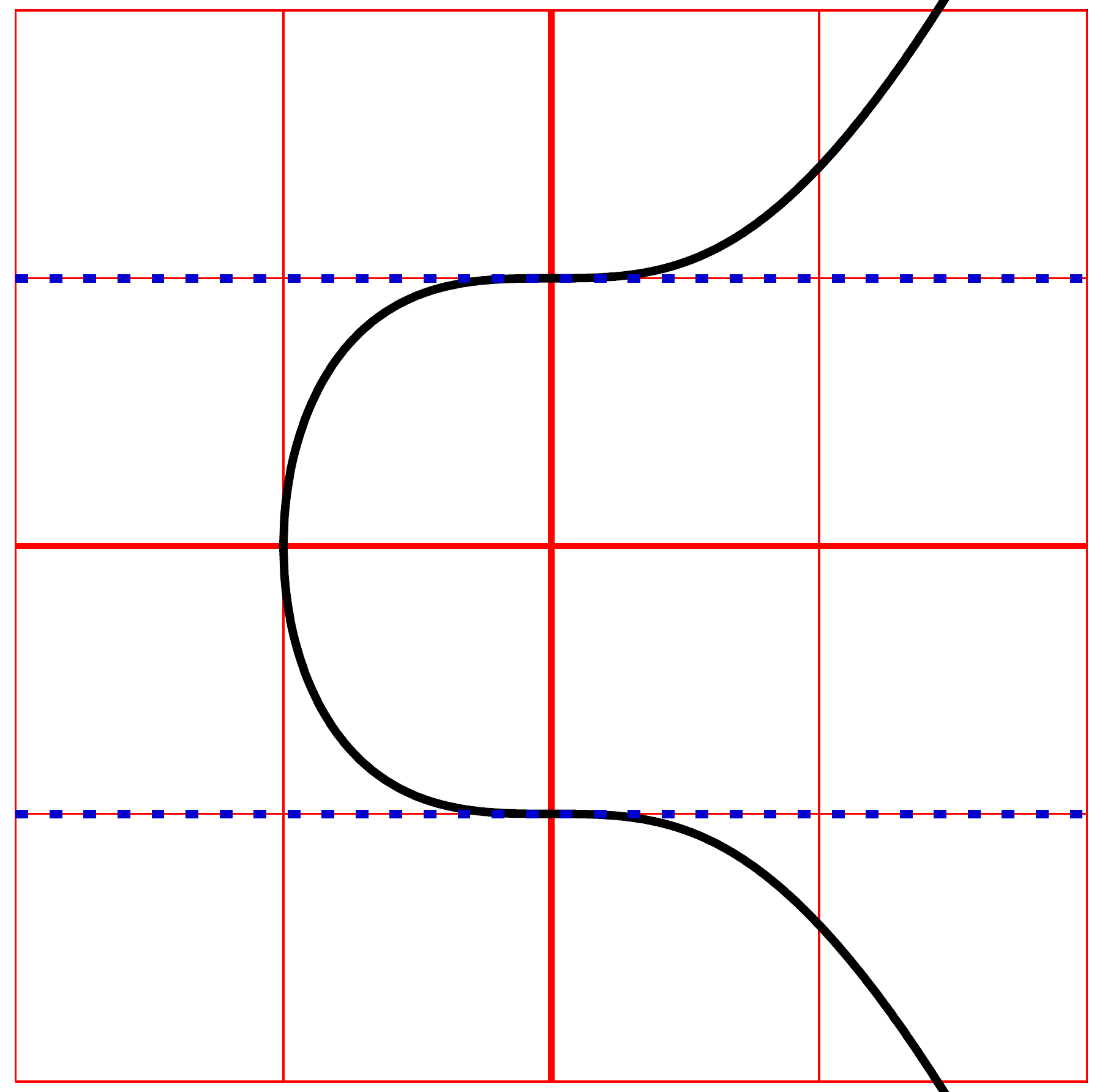} 
\includegraphics[height=1.5in]{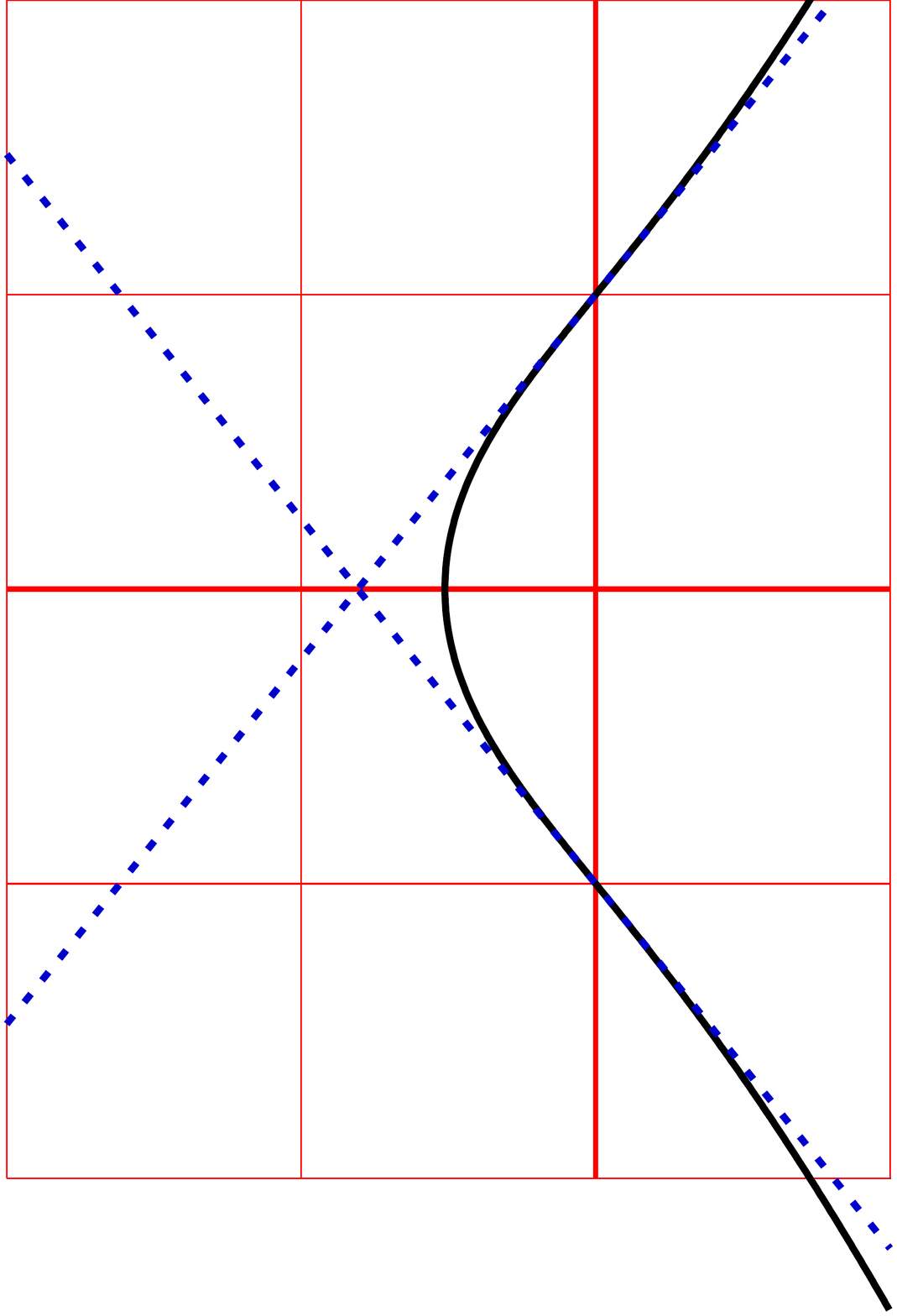} 
\includegraphics[height=1.5in]{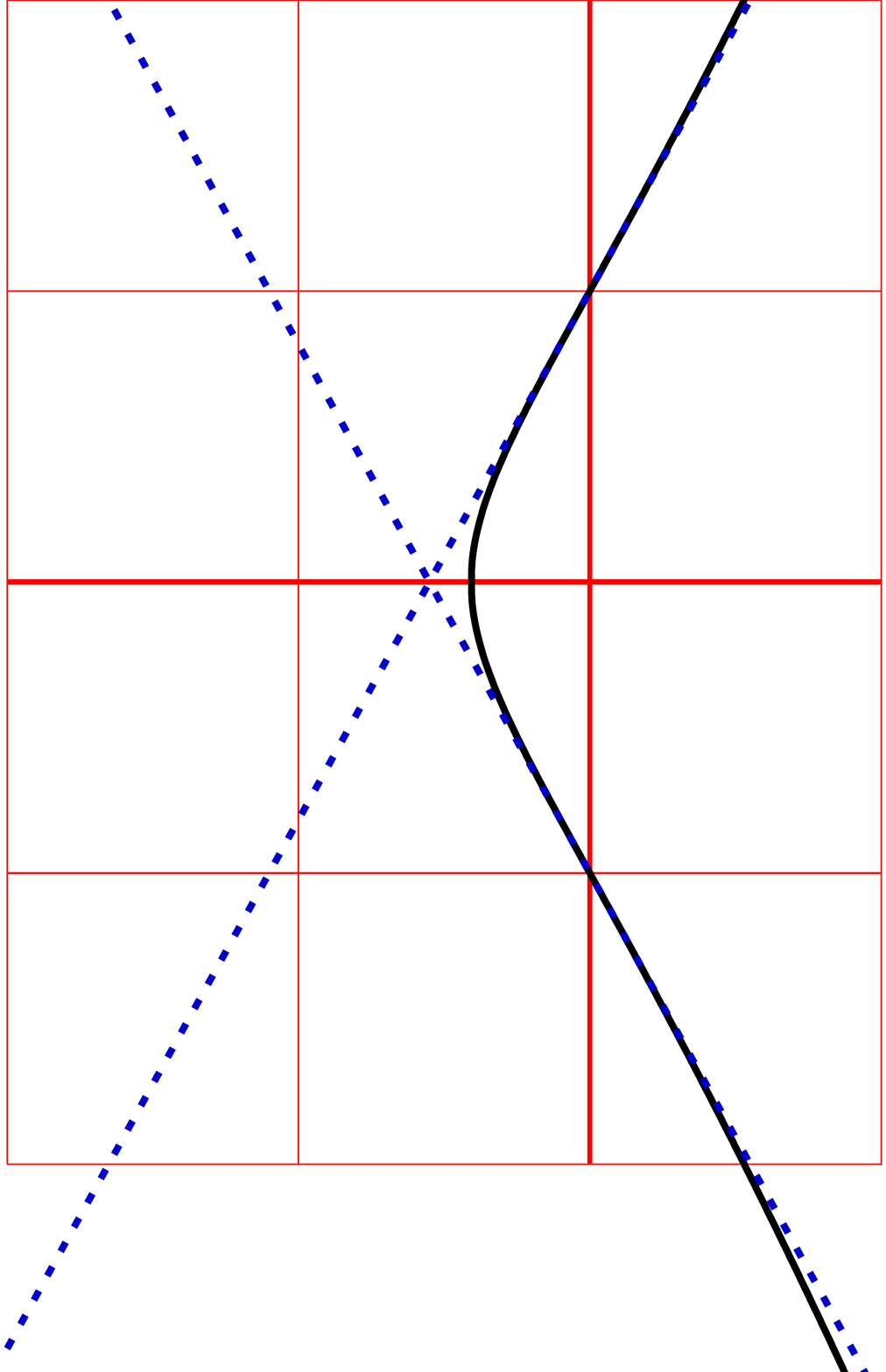}} 
\centerline{$s=-1.7$\qquad\qquad ($\bbet$) $s=0$\qquad\qquad 
($\bgam$) $s=1.238$\qquad\qquad ($\bdel$) $1.817$}
\medskip
\centerline{
\includegraphics[height=1.5in]{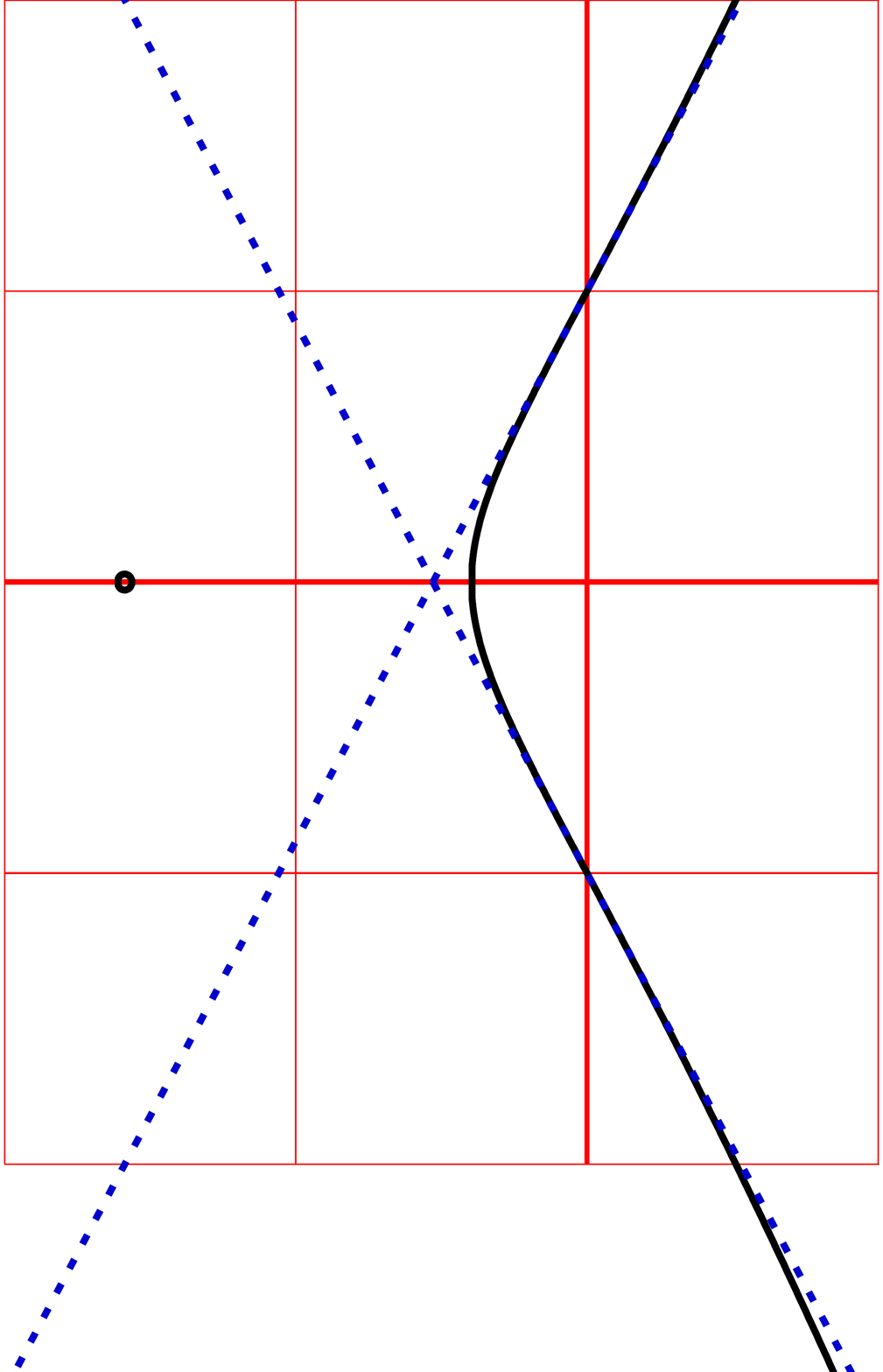} 
\includegraphics[height=1.5in]{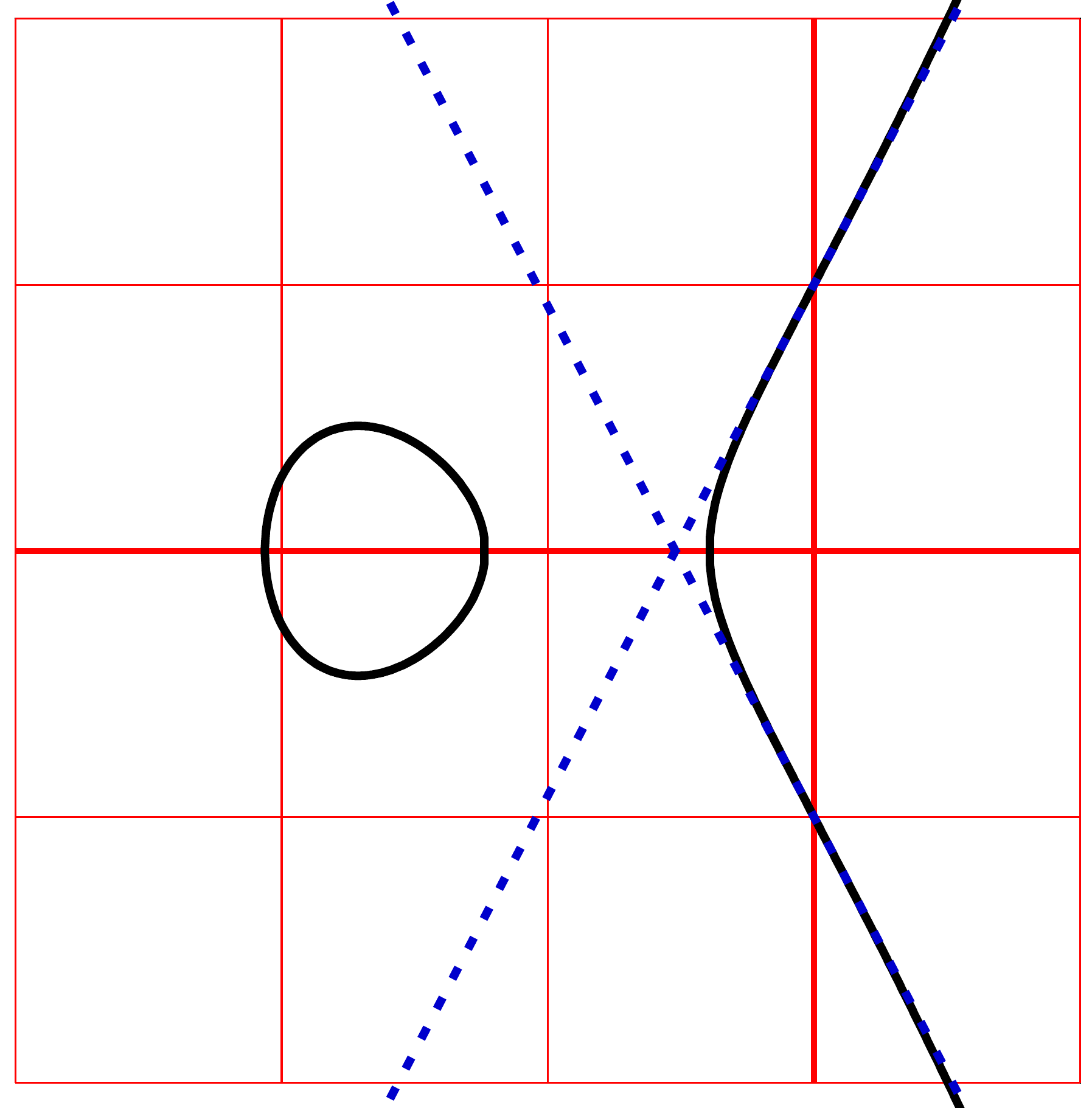} 
\includegraphics[height=1.5in]{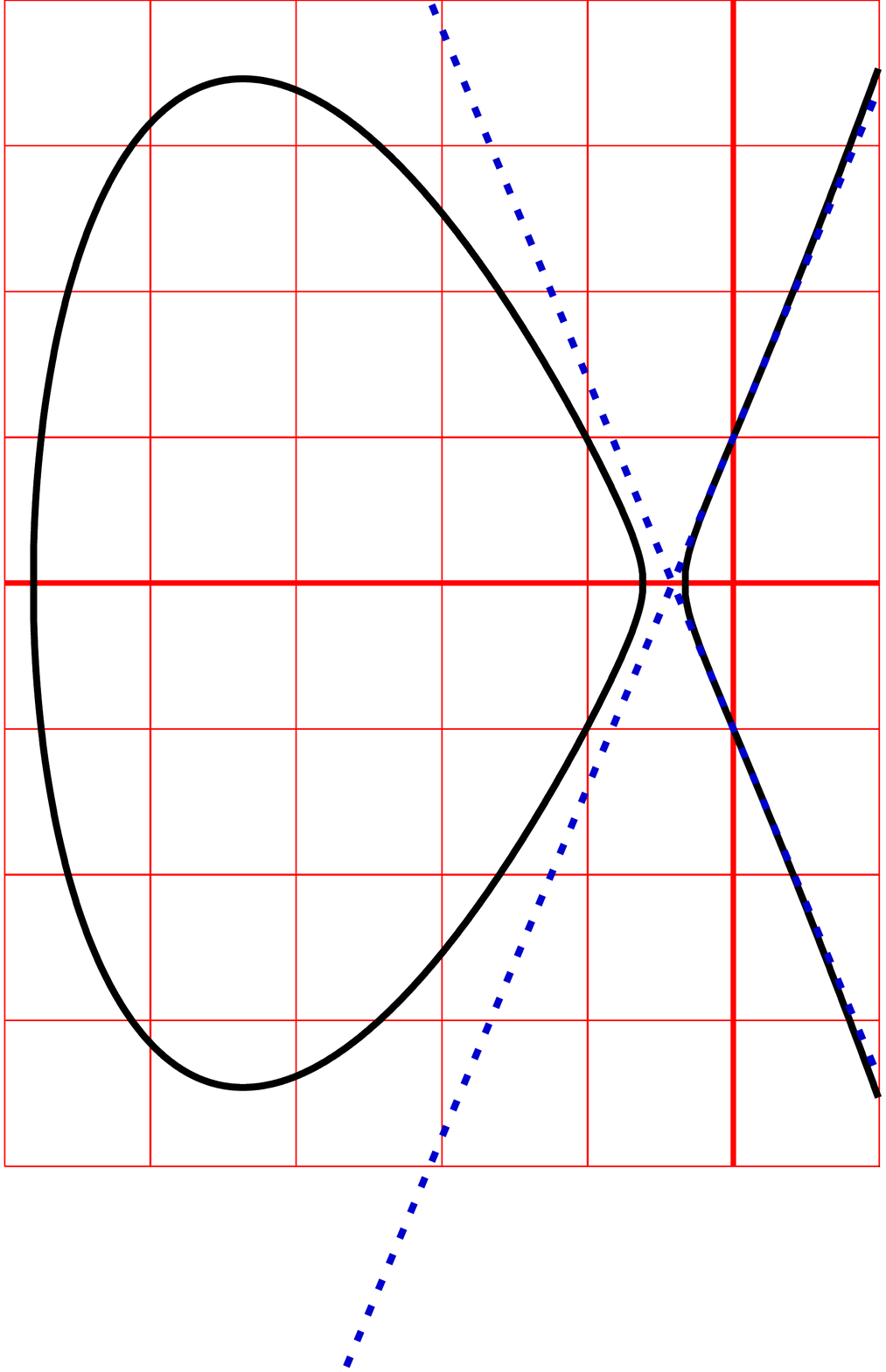}} 
\centerline{($\bep$) $s=1.89$\quad \qquad ($\bzet$) 
$s=1.921$\quad\qquad $s=2.4$\qquad}

\caption{\label{F1}\sf Graphs of seven curves ${\mathcal F}(s)$ 
in flex-slope normal form, so that the finite flex points are at $(0,\,\pm 1)$.
Here the slope $s$ ranges from $-1.7$ to
$+2.4$ The tangent lines at the flex points are indicated by dotted lines.
The
grid of points with an integer coordinate is also shown. Note the isolated
singular point which appears at $s=1.88988\ldots$ and immediately expands to a
circle. The figures blow up as $s\to\pm\infty$.\break (See Figure  \ref{F-alf}
for the limiting behavior.)
The five middle curves, labeled as $(\bbet)$ through ($\bzet$), correspond
to the points with the same labels in Figure \ref{F-rcirc}. \bigskip}
\bigskip
\end{figure}

\begin{figure}[h!]
\centerline{\includegraphics[width=1.5in]{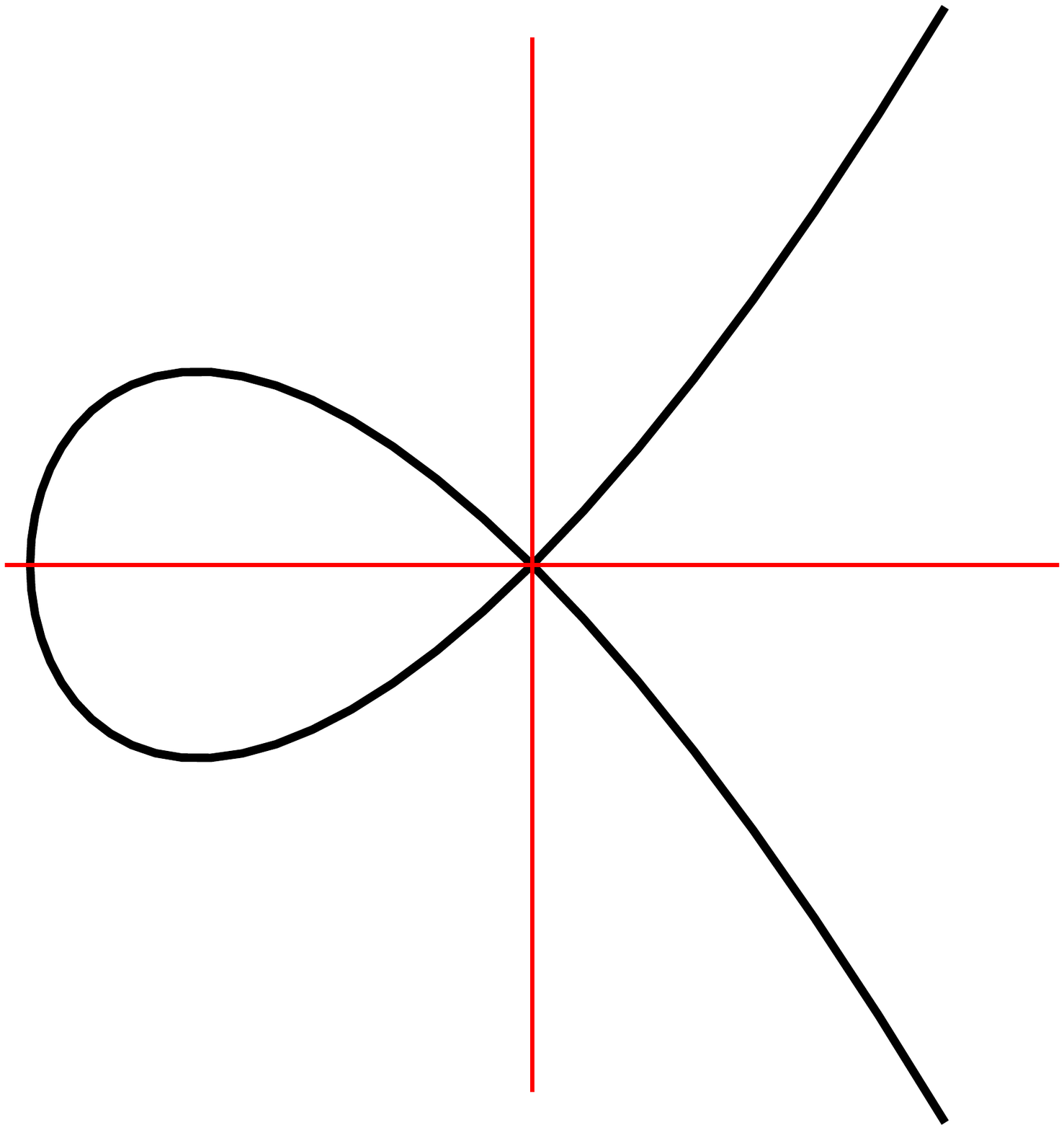}} \vspace{-.2cm}
\caption{\label{F-alf} \sf Although the flex-slope normal form blows up
as $s$ tends to $\pm\infty$, a carefully rescaled version, with
$Y^2=X^3+(X+1/s^3)^2$, tends to the illustrated curve, 
with a simple self-crossing at the limit of the two finite flex
points. This limit belongs to the leaf $\balph$ in Figure \ref{F-rcirc}.}
\bigskip
\end{figure}

\begin{proof}[Proof of Theorem \ref{T-FS}] The proof proceeds in three steps
as follows.\ssk

\noindent{\it Proof that ${\bf(a)}\Longrightarrow{\bf(b)}$.\/}
In the case of a smooth real or complex cubic curve, choosing one flex point as
base point, there is a classical additive
group structure, with the points of order
three as the remaining flex points. In the real case, this group is
isomorphic to either $\R/\Z$ or $\R/\Z\times(\Z/2)$. Thus it has a unique
subgroup of order three, and hence exactly three flex points.

There are just two real cubic curve-classes made up of singular curves.
(Compare Figure \ref{F1}$\bep$ and Figure \ref{F-alf}.) 
For the curve with an isolated singular point, 
there are still two finite flex points: Note that the slope $dy/dx$ of 
the upper branch $y>0$ of this curve tends to $+\infty$, both as $y\to 0$ 
and as $x,y\to+\infty$. Therefore the slope must take on a minimum value,
necessarily at a flex point, somewhere on this branch. It follows easily 
that there are three flex points altogether.\qed\ssk
\bigskip

(On the other hand, as we converge towards 
 the rescaled curve of Figure \ref{F-alf},
the equation $Y^2=X^3+(X+1/s^3)^2$ converges to $Y^2=X^3+X^2$, and
the two finite flex points  at $(0,\,\pm 1/s^3)$ converge to the 
singular point at the origin.)\medskip

\noindent{\it Proof that ${\bf(b)}\Longrightarrow{\bf(c)}$.\/}
Now assume that there are three flex points. According to Proposition
\ref{P-snf}, we can put the curve into standard normal form
$y^2=x^3+a\,x+b$, with one flex point on the line at infinity. Let
$(x_1,\,\pm y_1)$  be the two finite flex points. Translating the
$x$-coordinate appropriately, we can move these flex points to $(0,\,\pm y_1)$,
replacing the defining equation by $y^2=p(x)$, where the polynomial
$p(x)=f(x-x_1)$ is again monic.  Now replacing the coordinates 
$x, \,y$ by $X=c^2x$ and $Y=c^3y$ for some constant $c\ne 0$,
 the defining equation will be $$Y^2=P(X)=c^6p(X/c^2)~,$$ where the polynomial
$P(X)$ is again monic. Choose $c$ so that $c^3y_1=1$. Then the  flex points
will be at $(0,\,\pm 1)$. Let $s$ be the slope $dY/dX$ at the upper
flex point $(X,\,Y)=(0,\,1)$.
Then we can write $Y=1+s\,X+O(X^3)$ as $X\to 0$ with $Y>0$, hence 
$$Y^2=P(X)=(s\,X+1)^2+O(X^3)~.$$ 
 Since the polynomial $P(X)$ is monic of degree three, 
it follows that $P(X)$ has the required
form $P(X)=X^3 +(s\,X+1)^2$. Furthermore, since the construction 
is uniquely specified, it follows that the parameter $s$ is uniquely determined
by the curve-class.
\qed \bigskip

\noindent{\it Proof that ${\bf(c)}\Longrightarrow{\bf(a)}$.\/}
Finally, assuming that the curve is in flex-slope normal form~(\ref{E-FS}),
 we must show that it is either smooth everywhere, or else has just one 
isolated point as  singularity. It is not hard to show that any singularity
must lie on the $x$-axis, and correspond to a double or triple
root of the polynomial $x^3+(s\,x+1)^2$. Using the standard formula
for the discriminant of a cubic polynomial 
 (see for example \cite{BMac}), we can check that the discriminant
of this polynomial is given by
$\Delta~=~4s^3-27$. Therefore, the corresponding curve is singular
if and only if
 $$s~=~3/\root 3\of 4 = 1.88988\cdots .$$
For this value of $s$, it
 is not hard to check that the associated polynomial factors as
$$ x^3+(s\,x+1)^2~=(x+r)^2(x+r/4)\qquad{\rm where}\quad r=\root 3\of 4~,$$
with a double root at $-r$ and a simple root at $-r/4$. It follows easily
that the associated curve has an isolated point at $(-r,\,0)$.
Thus we have proved that \break
${\bf(a)}\Rightarrow{\bf(b)}\Rightarrow{\bf(c)}\Rightarrow{\bf(a)}$,
completing the proof of Theorem \ref{T-FS}.
\end{proof}

\bigskip

\begin{figure} [h!]
\centerline{\includegraphics[width=3.5in]{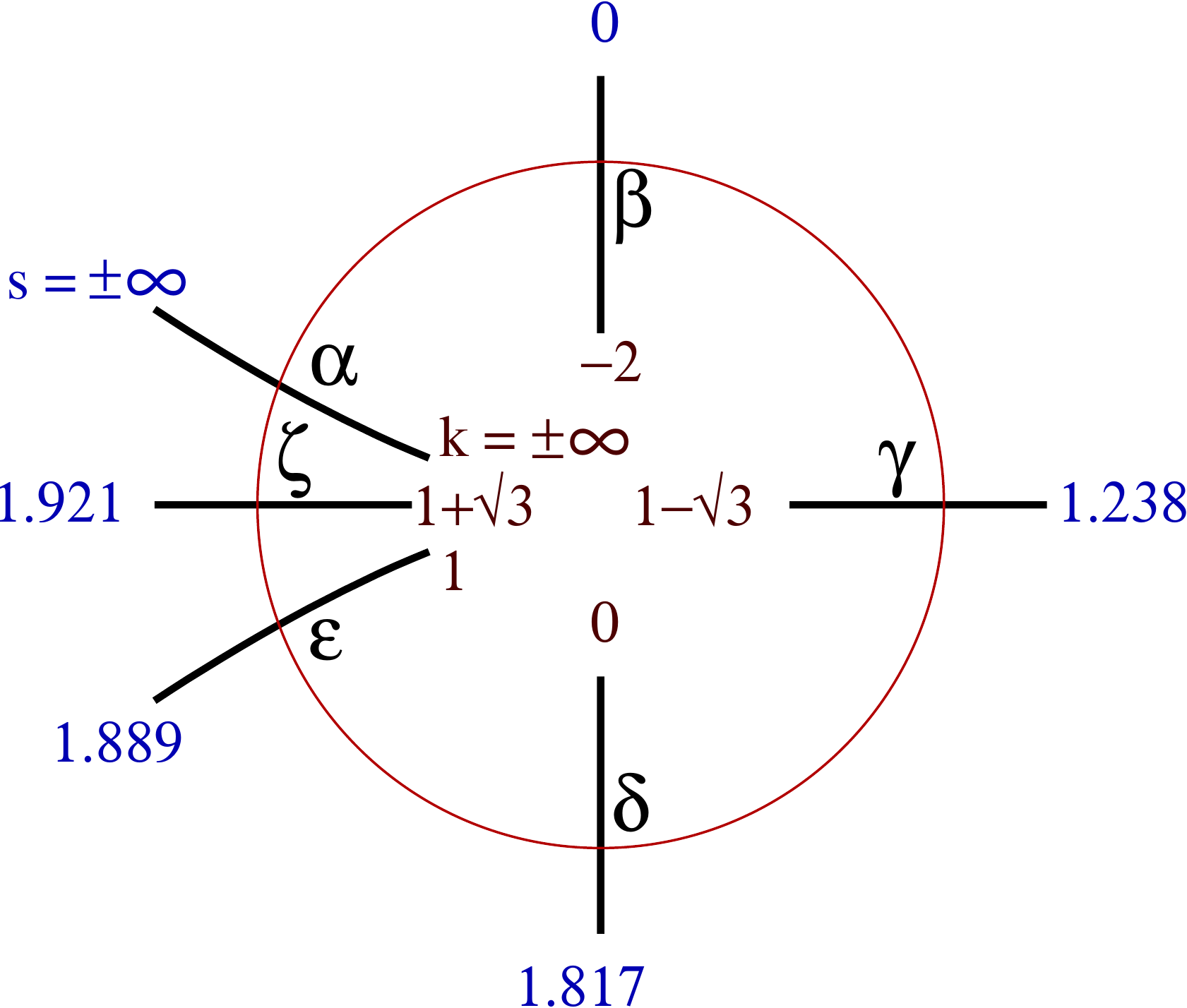}}
\caption{\label{F-rcirc} \sf Showing the unit circle in the $(a,b)$-plane
together with six leaves from the foliation of Figure \ref{F-abplane}.
The labels $\boldsymbol\alpha$ through $\boldsymbol\zeta$ on
 these leaves
correspond to the labels on the cubic curves of Figures
 \ref{F1} and \ref{F-alf}. The numbers outside the circle in this figure give
the flex-slope  invariant $s$ for the corresponding curves, and the numbers
 inside the circle  give
the Hesse invariant $k$. Note that both $s$ and $k$ increase monotonically 
from $-\infty$ to $+\infty$
as we follow the circle clockwise from $\boldsymbol\alpha$ back around to 
$\boldsymbol\alpha$. The curve associated with any point of this plane
 has two connected components if and only if\break
$\qquad\qquad k\ge 1~~\Longleftrightarrow~~s\ge 1.88988\ldots$.}
\end{figure}
\bigskip

\begin{rem}\label{R-s}
More generally, for any curve of the form
$$Y^2~=~X^3+AX^2+BX+C $$
the flex-slope invariant can be computed as  
$$s~=~\frac{dY/dX}{\root 3\of Y}~,$$
to be evaluated at either of the finite flex points $(X_0,\,\pm Y_0)$.
In fact, if we set $x=\lambda^2X$ and $y=\lambda^3Y$, then the slope will be
$dy/dx=\lambda\, dY/dX$. Choosing $\lambda=1/{\root 3\of Y_0}$, the $y$
 coordinate at the finite flex points with be $\pm 1$, and we can translate
 the $x$ coordinate so that it will be zero at these points.
\end{rem}
\medskip

Thus we have three different possible normal forms for real cubic curves:
the unit circle normal form~(\ref{E-cnf}), the Hesse normal form~(\ref{E-He}),
 and the flex-slope normal form~(\ref{E-FS}).
These are compared in Figure \ref{F-rcirc}. Two of the leaves in this figure
correspond to singular curves, and separate the connected cubic curves
from curves with two components: The $\balph$-leaf is the set
 of curves with a self-crossing point, giving rise to improper group action;
while the $\bep$-leaf is the set of curves with an isolated 
(necessarily singular) point. A cubic curve has two components if and only if 
$~~1\le k<\infty~~~\Longleftrightarrow~~~ 1.889\ldots\le s<\infty$

The $\bbet$ leaf also corresponds to cubic curves with a distinctive geometry. 
These are the only real cubics such that the tangent lines at the three flex
points all pass through a common point. Note that the $\bbet$ leaf and the
$\bdel$ leaf both lie in the coordinate line $a=0$, with  shape invariant 
$\bJ=0$. These correspond to complex curves with six-fold rotational symmetry.
Similarly both the $\bgam$-leaf and the $\bzet$-leaf lie in the coordinate line
$b=0$ with $\bJ=1$, corresponding to complex cubics with four-fold rotational 
symmetry.\ssk

It is noteworthy that the one singular curve-class in $\M_3(\C)$ splits
into two distinct singular curve-classes $\balph$ and $\bep$ in $\M_3(\R)$.
However, only the singularity of type $\balph$,
 corresponding to curves with a real self-crossing point, is improper. 
All other curves in $\fC^\fs_3$ have an
automorphism group $\fS_3$ of order six; but real curves of class $\balph$ 
have an automorphism group of order two. The proof of weakly proper action
in this case is completely analogous to the proof in the complex case.

\bigskip

\setcounter{lem}{0}
\section{Degree $n\ge 4$: the Complex Case. }\label{s-cc} 
Recall from Section \ref{s-mod} that the moduli space $\M_n=\M_n(\C)$
is defined to be the quotient space $\wfC_n^\fs/\bG$, where 
$\wfC_n^\fs=\wfC_n^\fs(\C)$
 is the open subset consisting of 1-cycles with finite
stabilizer in the  complex projective
 space $\wfC_n$ consisting of all 1-cycles of degree $n$, 
and where $\bG$ is the projective linear group $\PGL_3(\C)$.
Here is a preliminary statement.

\begin{prop}\label{P-nh}
  For $n\ge 7$ the moduli space $\M_n$ is not a Hausdorff space.
\end{prop}

(It seems likely that $\M_4,\,\M_5,\, \M_6$ are 
also non-Hausdorff; but we don't know.)

\begin{proof}[Proof of Proposition \ref{P-nh}] Consider the subspace
of $\M_n$ consisting of formal sums
$$ m_1\cdot L_1+\cdots+ m_k\cdot L_k\quad{\rm with}\quad n={\textstyle \sum}
 m_j~, $$
where the $L_j$ are lines. Each line $L_j$ in the plane $\bP^2$
is dual to a point $\p_j$ in the dual plane $\bP^{2*}$, yielding an
associated zero-cycle $~~m_1\<\p_1\>+\cdots+m_k\<\p_k\>~~$ in the dual plane.
The argument is now similar to the proof of Lemma \ref{L-D5}, but
with suitable modification since we are now working in $\bP^2$
rather than $\bP^1$.

By definition, four points of $\bP^2$ are in \textbf{\textit{general position}}
if no three are contained in a common line. Note that the action of
$\bG=\PGL_3$ on $\bP^2$ is simply transitive on 4-tuples
$(\p_1,\,\p_2,\,\p_3,\,\p_4)$ which are in general position. In fact
there is one and only one group element $\g$ such that 
$$ \g(\p_1)= (1:0:0)\,,~~\g(\p_2)=(0:1:0)\,,~~\g(\p_3)=(0:0:1)\,,
~~{\rm and}~~\g(\p_4)=(1:1:1)~.$$
It follows that any zero-cycle which includes four points in general
position will have finite stabilizer.

We will make use of automorphisms of the form $\g_t(x:y:z)=(t^{-1}x:y: tz)$
so that $\g_t(x:y:z)\to (0:0:1)$ as $t\to\infty$  if $z\ne 0$, 
and $\g_t^{-1}(x:y:z)\to(1:0:0)$ if $x\ne 0$.

Now let $A_3$ and $A_4$ be zero-cycles of degree three and four in general
position and with all points satisfying $xz\ne 0$. Then the cycles
$$ A_3+\g_t(A_4)\qquad{\rm and}\qquad \g_t^{-1}(A_3)+A_4 $$
belong to the same $\bG$-orbit. But as $t\to\infty$ the first tends to
$A_3+4\<(0:0:1)\>$ while the second tends to $3\<(1:0:0)\>+A_4$. Here the
second limit clearly has finite stabilizer, and we can choose $A_3$ so
that the first does also. Since these two limits evidently
 do not belong to the same $\bG$-orbit, it follows
that the quotient space is not Hausdorff.
\end{proof}
\bigskip

In this section and the next one
 we will describe moderately large open subsets of $\M_n$ which are Hausdorff. 
First note that $\M_n$ is a ${\rm T}_1$-space.
\smallskip

\begin{theo}[\bf Theorem of Ghizzetti; Aluffi, and Faber] For any
\hbox{$\bG$-orbit} 
$\((\cC\))\subset\wfC_n$, every limit point in the complement
$\overline{\((\cC\))}\ssm\((\cC\))$ has infinite stabilizer. 
\end{theo}
\smallskip

See \cite{Ghi} and \cite{AF2}. Although this statement is not emphasized in
these papers, it is clearly stated; see 
for example  \cite[p. 35]{AF2}. The proof
involves a detailed case by case analysis. \qed
\medskip

As an immediate Corollary, it follows that:

\begin{coro}
Every $\bG$-orbit which is contained in $\wfC_n^\fs$ is a closed subset of 
$\wfC_n^\fs$. In other words, every point
in the quotient space \hbox{$\M_n=\wfC_n^\fs/\bG$} is closed.
\end{coro}  
\medskip

The constructions in this section will be
 based on the concept of ``virtual flex point''.
\ssk

\begin{definition} \label{D-vfp}
Let $\U_n\subset\fC_n$ be the open set consisting of 
all \textbf{\textit{line-free curves}} $\cC$ of degree $n$. In other words,
for $\cC\in\U_n$ we  assume:
\begin{itemize}
\item[{\bf (1)}] that $\cC$ contains no line, and 

\item[{\bf (2)}] that every  
irreducible component of $\cC$ has multiplicity one.
\end{itemize}

\noindent For $\cC\in\U_n$, a point $\p\in|\cC|$ will be called a
\textbf{\textit{virtual flex point}} if it is either a singular point or
 a flex point.
\end{definition}

\begin{lem}\label{L-vfp} Every virtual flex point $\p$ for a curve in $\U_n$
can be assigned a \textbf{\textit{flex-multiplicity}} 
$~~\bph(\p)\ge 1~$ with the following properties:

\begin{itemize}
\item[{\bf(1)}] The sum of $\bph(\p)$ over all virtual flex points is equal
to $3n(n-2)$.

\item[{\bf(2)}] Under a generic small perturbation, each virtual flex point
 $\p$ splits into $\bph(\p)$ distinct nearby simple flex points.
\end{itemize}\end{lem} \ssk

\begin{figure}[h!]
\centerline{\includegraphics[width=4.1in]{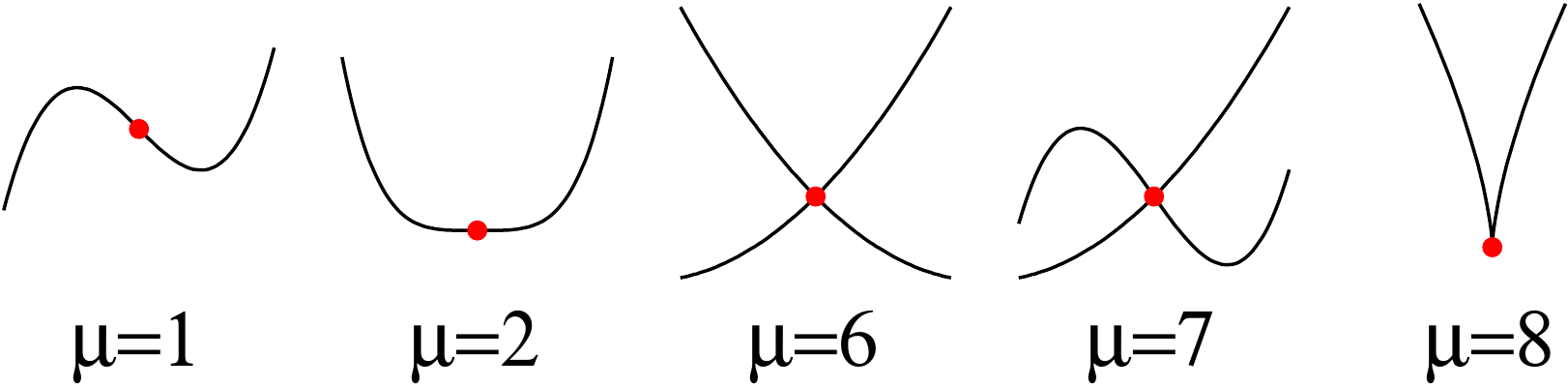}}
\caption{\label{F-vfp}
\sf Examples of virtual flex points with small flex-multiplicity. The first is
a simple flex point, the second is a double flex point\break $(y=x^4)$,
and the remaining three are singular points.
}\end{figure}

\begin{proof}[Proof of Lemma \ref{L-vfp}]
Define $\bph(\p)$ as the local intersection multiplicity
between the curve $\cC$ of degree $n$ and its associated Hessian curve 
$\cH_\cC$ of degree $3(n-2)$.  (See for example \cite{Gr}, \cite{Kir},
 \cite{Kun}, or \cite{Sha}.)
The statement then follows from B\'ezout's Theorem.
In order to apply B\'ezout, it is first necessary 
to check that $\cC$ and $\cH_\cC$ have no common sub-curve.
But such a sub-curve would have to be either a line
or a component of 
multiplicity $\ge 2$; and both possibilities have been excluded.
To prove that $\bph(\p)\ge 1$ at every singular point, we proceed as
follows. Taking the singular point of a curve of degree $n$ to be $(0:0:1)$
the defining equation must have the form
$$\Phi(x,y,z)~=~\sum_{j=2}^n~\Phi_j(x,y)\,z^{n-j}$$ where the $\Phi_j$ are 
homogeneous of degree $j$. It is then easy to check that the last row
$(\Phi_{x\,z},~\Phi_{y\,z},~\Phi_{z\,z})$ of the Hessian matrix is identically
zero at $(0,0,1)$; so that
 this point is a common zero of $\Phi$ and the Hessian determinant.
\end{proof}

\ssk

\begin{ex}\label{EX-mu} Suppose that there are $k$ smooth local branches
($=$ curve germs) 
$\cB_1,\,\ldots,\,\cB_k$ of the curve $\cC$ passing through $\p$.
(Compare the cases $\bph=6,\,7$ in Figure \ref{F-vfp}  and for other
 examples see  Figure~\ref{F-deg2x2}.)  Then
$$\bph_\cC(\p)~=~\sum_j \bph_{\cB_j}(\p)~+~ 6\,\sum_{i<j} \cB_i\cdot \cB_j~,$$
where $\cB_i\cdot \cB_j$ is the local intersection multiplicity. (Here 
$\bph_{\cB_j}(\p)$ makes sense, since $\bph$ can be 
defined as a local analytic invariant.) This equation can be proved
choosing small generic translations of the $\cB_j$ so that they intersect
transversally, and then noting that a simple intersection has 
flex-multiplicity six.\footnote{Consider for example a cubic curve with
a single double point, it follows from the proof of Lemma \ref{L-1flex}
that there are three flex points, which must certainly have multiplicity
$\bph=1$. Since $\sum\bph=9$, it follows that $\bph=6$ at the double point.
(Alternatively, see Figure \ref{F-deg2x2}{\bf(4)} for an example of degree
$n=4$ with four
simple double points and no flexes, and with $\sum\bph=3n(n-2)=24$.)}
\end{ex}\medskip

It will be convenient to consider the probability measure on $\bP^2$ defined by
$$ \widehat\bph(S)~=\frac{~\sum_{\p\in S}\, \bph(\p)}{3n(n-2)}~~\in~~[0,\,1]
\qquad{\rm for~every~set}\quad S\subset\bP^2~.$$
Here it is understood that $\bph(\p)=0$ unless $\p$ is a virtual flex point
 in $|\cC|$. 
\smallskip

\begin{definition}
To every $\cC\in\U_n$, we can assign the two  rational numbers
$$ {\bf p_{max}}={\bf p_{max}}(\cC)=\max_\p~ \widehat\bph_\cC(\p)\qquad{\rm and}
\qquad {\bf L_{max}}={\bf L_{max}}(\cC)=\max_L~\widehat\bph_\cC(L)~,$$
where $\p$ ranges over all points in $|\cC|\subset\bP^2$, and
where $L$ ranges over all lines in $\bP^2$. Evidently, since there are
 at most $n$ points of $\cC$ on $L$
\begin{equation}\label{E-M}
 0~<~ {\bf p_{max}}~\le~ {\bf L_{max}}~\le~ 1\,,\quad {\rm and}\quad 
{\bf L_{max}}\le n\;{\bf p_{max}}~.\end{equation}
(Compare the emphasized triangle in Figure \ref{F-Msquare}.)
\end{definition}

\begin{figure}[h!]
\centerline{\includegraphics[width=2.5in]{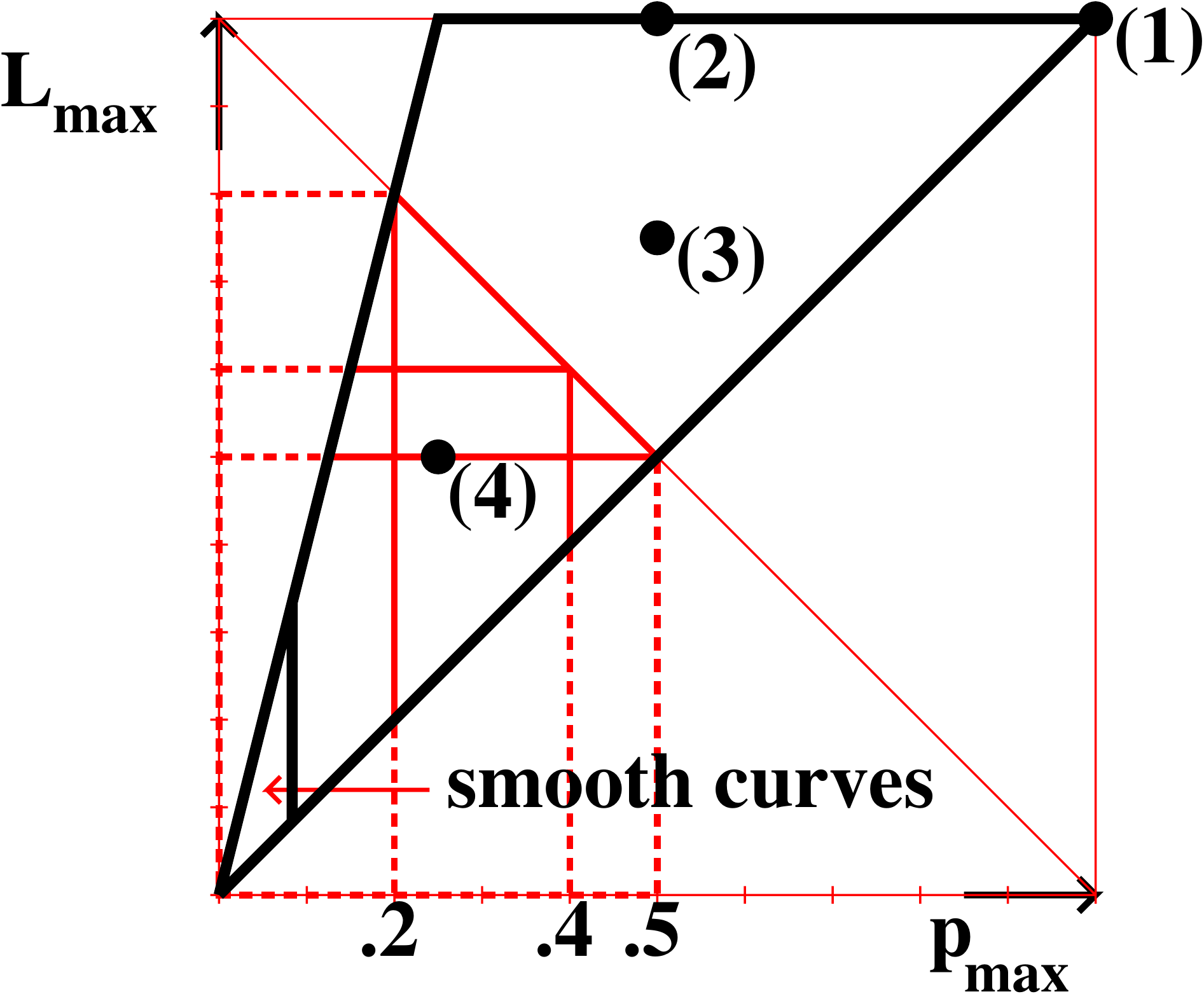}}
\caption{\sf The large black triangle encloses all possible pairs
 $\big({\bf p_{max}}(\cC),\, {\bf L_{max}}(\cC)\big)$ for curves of degree $n=4$,
 while the small  black triangle encloses the possible pairs for smooth degree
 four curves. By Theorem \ref{T-dis}, the $\bG$-action is locally proper for 
curves below the\break
diagonal line ${\bf p_{max}}+{\bf L_{max}}=1$.  As defined in 
Theorem~\ref{T-dis}, the boundaries of  three typical rectangles
 $\U_4(1/5)\,,~\U_4(2/5)$, and $\U_4(1/2)$ of provably proper action
are also shown. The heavy dots correspond to the four
examples shown in  Figure~\ref{F-deg2x2}.\label{F-Msquare}}
\end{figure}
\medskip

\begin{theo}\label{T-dis} \it If
  $~~{\bf p_{max}}(\cC_0)+{\bf L_{max}}(\cC_0)<1~~$
for a curve $\cC_0\in\U_n$, then the action
of $\bG$ on $\U_n$ is locally proper at $\cC_0$, hence the moduli space $\M_n$
is locally Hausdorff at the corresponding point $\bpi(\cC_0)\in\M_n$. 
In fact, choosing a real number $\kappa$ 
so that $~~{\bf p_{max}}(\cC_0)<\kappa\quad{\rm  and}\quad {\bf L_{max}}(\cC_0)<1-\kappa~,$
the action of $\bG$ is proper throughout the entire open subset 
$\U_n(\kappa)\subset\U_n$
consisting of curves which satisfy $~~{\bf p_{max}}(\cC)<\kappa$ and
${\bf L_{max}}(\cC)<1-\kappa$.
\end{theo}

\msk

\begin{coro}\label{C-dis} The action of $\bG$ is locally proper at every
smooth curve in $\fC_n$. In fact, every smooth curve belongs to $\U_n(\kappa)$
for every $\kappa$ between $1/(n+1)$ and $1/2$. Furthermore,
if $\cC$ is any curve in $\U_n$ which satisfies ${\bf p_{max}}+{\bf L_{max}}<1$, 
then $\cC$ belongs to $\U_n(\kappa)$ for some $\kappa$ between $1/(n+1)$
and $1/2$. Thus every 
curve-class with ${\bf p_{max}}+{\bf L_{max}}<1$ belongs to a Hausdorff
orbifold open subset of $\M_n$ 
which also contains every smooth curve-class.
\end{coro}
\smallskip

\begin{proof}[Proof of Corollary \ref{C-dis}
$($assuming  Theorem \ref{T-dis}$)$] If $\cC$ is a smooth
 curve,  then every virtual flex point is an actual flex point, with
 flex-multiplicity
$$\bph\le n-2\,,\qquad{\rm hence}\qquad \widehat\bph\le 1/3n~.$$
Thus ${\bf p_{max}}\le 1/3n$; and therefore  ${\bf L_{max}}\le 1/3$ 
 since there are at most $n$ points of $\cC$ on any line. 
It follows easily that
\hbox{ $\cC\in
\U_n(\kappa)$} for every $\kappa$ between $1/(n+1)$ and $1/2$. 

Note that if we parametrize the line ${\bf p_{max}}+{\bf L_{max}}=1$ by setting
 $$({\bf p_{max}},\,{\bf L_{max}})
=(\kappa,\,1-\kappa)~,$$ then this line intersects the triangle defined by the
inequalities (\ref{E-M}) precisely in the interval $1/(n+1)\le \kappa\le 1/2$.
(Compare Figure \ref{F-Msquare}.)
Further details are easily supplied. 
\end{proof}
\msk

\begin{figure}[h!]
  \centerline{\includegraphics[width=3.8in]{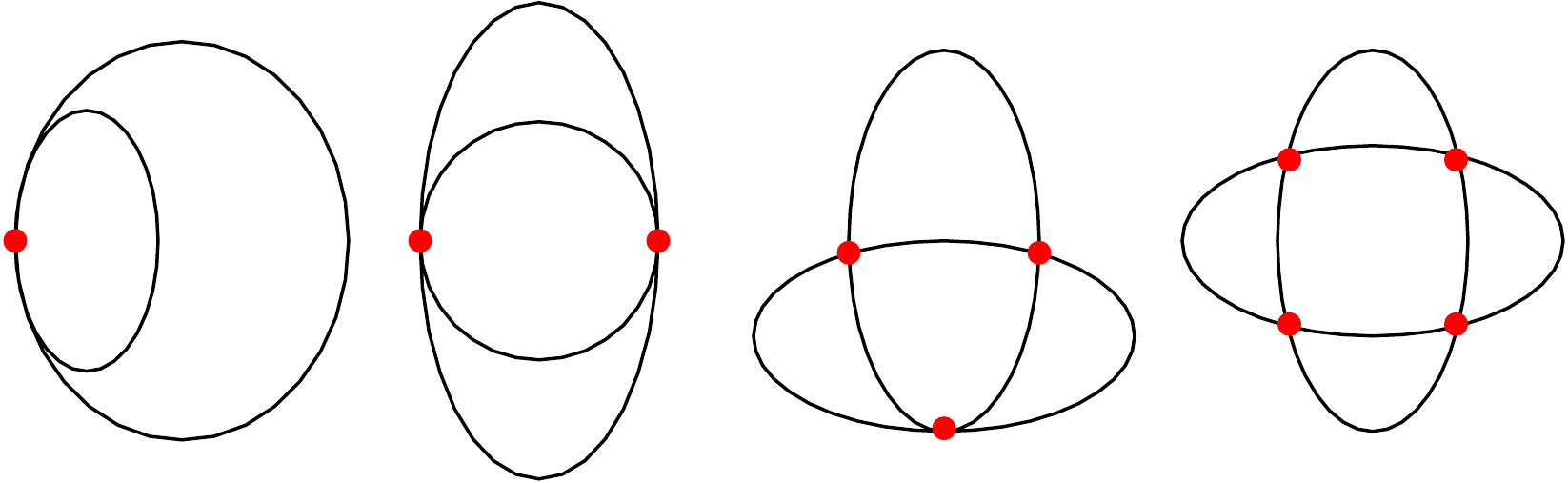}}
\centerline{$\bf(1)\qquad\qquad(2)\qquad\qquad\quad(3)\qquad\qquad\quad(4)$}
\caption{\label{F-deg2x2} \sf Representatives for the four equivalence
classes of curves which are unions of two smooth quadratic curves. 
  The first two are W-curves. (See \S\ref{s-aut}.) In
each case there are four intersection points, counted with multiplicity.}
\end{figure}\medskip

\begin{ex}\label{EX-deg2x2}
Let $\cC$ be a curve of degree four which is the union 
of two smooth curves of degree two. Then $\cC$ has
no flex points, but may have either one, two, three, or four singular points,
as illustrated in Figure~\ref{F-deg2x2}. The corresponding values of 
${\bf p_{max}}$ and ${\bf L_{max}}$ can be tabulated as follows. 
(Compare the four labeled points in Figure \ref{F-Msquare}.)
$$\begin{matrix}
{\rm \# singular~ points:}& (1)&(2)&(3)&(4)\\
{\bf p_{max}}+{\bf L_{max}}:& 1+1=2& .5+1=1.5& .5+.75=1.25& .25+.5=.75\\
\end{matrix}$$\ssk

\noindent
Thus Theorem \ref{T-dis} implies that  the last curve represents a
point in moduli space which is proper, and hence locally Hausdorff.
On the other hand,
the first two are \hbox{W-curves}, and hence do not represent any point of
 moduli space. (Compare Figure~\ref{F-W4} in Section~\ref{s-aut}.) 
The point of $\M_4$ corresponding to the third curve is more interesting:
\end{ex}\smallskip

\begin{prop}\label{P-nh  } Let $\cC_3$ be a curve in $\bP^2(\C)$
which is the union of two smooth quadratic curves which have three
intersection points $($as in Figure
$\ref{F-deg2x2}{\bf(3)})$. Then the action of $\bG$ is not even
weakly proper at $\cC_3$.
\end{prop}

\begin{proof}  We will make use of the statement that the automorphism
  group (= stabilizer) of a generic curve of degree four is trivial.
  (See Theorem \ref{T-noaut} below.)

Note that any 
curve $\cC$ which is the union of two smooth quadratic curves
with four distinct intersection points has a group of projective automorphisms
which is transitive on these four points.

To see this, note that the
 four points must be in general position, since a line can intersect
a smooth quadratic in at most two points. Thus, after a projective 
transformation, we can put the intersection points at $(\pm 1,\,\pm 1)$. 
The general quadratic equation in affine coordinates can be written
as $Q(x,y)+L(x,y)=c$; where $Q$ is homogeneous quadratic and $L$ is linear.
If the equation is to hold at all of the points $(\pm 1,\,\pm 1)$, then it is
easy to check that the linear term $L(x,y)$ must be zero, and that the
coefficient  of $xy$  in the quadratic term must be zero. Thus we are reduced
to an  equation of the form  $a\,x^2+b\,y^2=c$; where evidently $c$ must equal
$a+b$.

Any curve defined by an equation of the form $a\,x^2+b\,y^2=a+b$ 
 is clearly invariant under the four element group
$$(x,y)\mapsto (\pm x,\,\pm y)~.$$
It follows that any union $\cC$ of two such curves is also invariant
 under this four element group.
 Let $\g_\cC$ be the element of the stabilizer $\bG_\cC$ corresponding
 to the involution $(x,y)\leftrightarrow(-x,\,-y)$. We can choose a curve
 $\cC_0$ arbitrarily close to $\cC$ which has no non-trivial automorphism.
 It follows that the curves $\cC_0$ and $\g_\cC(\cC_0)$ represent the same
 point of $\M_4$; but that $\g_\cC$ is the only group element carrying
 one to the other.

 Now, as we move two of the intersection points together to obtain
 the third curve {\bf(3)} in Figure \ref{F-deg2x2}, the corresponding 
 sequence of involutions clearly cannot lie in any compact subset of $\bG$.
Thus the action of $\bG$ is not weakly proper at this curve.
\end{proof}

We don't know whether the moduli space $\M_4$ is locally Hausdorff near
the corresponding point.\medskip

The proof of Theorem \ref{T-dis} will be based on the following.
\smallskip

\begin{lem}[{\bf Distortion Lemma for $\bP^2$}]\label{L-dis2}
For any $\varepsilon>0$ there exists a compact set $K_\varepsilon\subset \bG
=\PGL_3(\C)$ such that, for any $\g\in \bG\ssm K_\varepsilon$, one or both 
of the following two conditions is satisfied. Either:

\begin{itemize}
\item[{\bf (1)}] there exists a line  $L^+\subset\bP^2$ and
 a point $\p^-\in\bP^2$ such that
$$ \g\big(N_\vep(L^+)\big)~\cup~ N_\vep(\p^-)~=~\bP^2\,,~{\it or}$$

\item[{\bf (2)}] there exists a point $\p^+\in\bP^2$ and a line 
$L^-\subset\bP^2$ such that
$$ \g\big(N_\vep(\p^+)\big)~\cup~ N_\vep(L^-)~=~\bP^2\,.$$
\end{itemize}
\end{lem}
\smallskip

\noindent(Note that we can interchange the two cases simply by replacing $\g$ by
 $\g^{-1}$.) 
\smallskip

\begin{proof}[Proof of Lemma \ref{L-dis2}]
In order to prove 
 this Lemma, we will discuss first  the special situation in
which the group action is diagonalizable of the form
$$  {\bf d}(x:y:z)~=~(x':y':z')=(rx:sy:tz)\quad{\rm with}\quad 
|r|\ge|s|\ge|t|>0~.$$
The ratio $\delta=|r/t|\ge 1$ can be thought of as a measure of the distortion
of $\bf d$. Note that  either 
$$ |r/s|~\ge~\sqrt{\delta}\qquad {\rm or} \qquad |s/t|~\ge~\sqrt{\delta}~.$$
To fix our ideas, suppose that 
$|r/s|\ge\sqrt{\delta}$, and set $k=\root 4\of{\delta}$. Then 
$$|r/t|~=~k^4~\ge~|r/s|~\ge~ k^2~\ge 1~.$$
 We are interested in estimates when $k$ is large.
It will be convenient to set
$$ X=\frac{|x|}{|x|+|y|+|z|}\,,~~Y=\frac{|y|}{|x|+|y|+|z|}\,,
~~Z=\frac{|z|}{|x|+|y|+|z|}\,,$$
so that $X+Y+Z=1$, with $X',~Y',~Z'$ defined similarly. Then
$$ X'/Z'=|r/t|X/Z\ge k^2X/Z\quad{\rm and}\quad X'/Y'=|r/s|X/Y\ge k^2X/Y~.$$
In particular, if $X/Z>1/k$ then $X'/Z'>k$,
 and similarly if $X/Y>1/k$ then $X'/Y'>k$.
It then follows easily, as illustrated in Figure \ref{F-dis2}, that
$$ X\ge 1/k \qquad\Longrightarrow\qquad X'>1-2/k~.$$
If $k$ is large, this means that everything out of a small neighborhood
of the line $X=0$ is mapped into a small neighborhood of the point $X=1$,
where $Y=Z=0$. This proves the first case of Lemma~\ref{L-dis2} in the 
special case of diagonal action.\msk

\begin{figure}[h!]
\centerline{\includegraphics[width=2.9in]{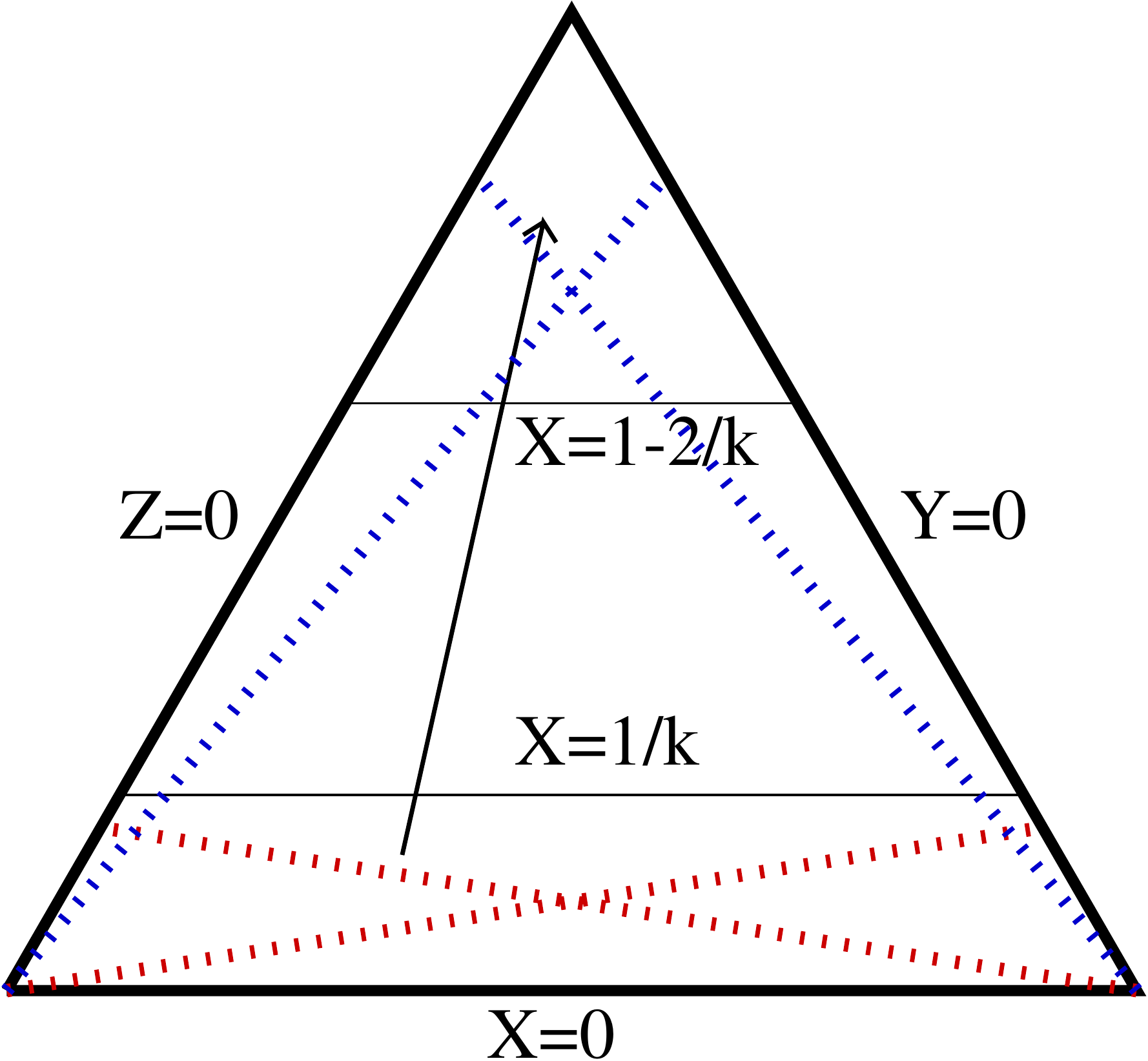}}
\caption{\label{F-dis2} \sf Showing the triangle of real numbers $X,\,Y,\,Z$
with \hbox{$X+Y+Z=1$}. The lower dotted lines indicate the loci
$X/Z=1/k$ and $X/Y=1/k$ for the case $k=5$, while the upper
dotted lines indicate the loci $X/Z=k$ and  $X/Y=k$.
Everything above
the lower dotted lines is pushed above the upper dotted lines by $\g$;
hence the region $X\ge 1/k$ is pushed into the region $X'>1-2/k$.}
\end{figure}
\medskip

But according to Lemma \ref{L-linalg}, any element of $\bG=\PGL_3$ can be 
written uniquely as a product ${\bf r}\circ{\bf d}\circ{\bf r'}$ where $\bf r$
and $\bf r'$ are unitary rotations, and where $\bf d$ is diagonal.
Suppose for example that for the diagonal  transformation
$\bf d$, everything outside of a small neighborhood of the line $x=0$ is 
pushed into a small neighborhood of the point $y=z=0$. Then setting
$$L^+={\bf r'}^{-1}\!\big(\{x=0\}\big)\qquad{\rm and}\qquad 
\p^-={\bf r}\big(\{y=z=0\}\big)~,$$
we obtain the required line and point, with

$$\xymatrixrowsep{.01in}
\xymatrix{\bP^2 \ar[r]^{\bf r'}& \bP^2 \ar[r]^{\bf d}& \bP^2 \ar[r]^{\bf r}& \bP^2\\ 
L^+\ar[r]^\cong& \{x=0\}& \{y=z=0\}\ar[r]^\cong&\p^-~.}$$

 This completes the proof of the
first case of Lemma \ref{L-dis2}. The second case is completely
analogous (or follows by replacing $\g$ by $\g^{-1}$).
\end{proof}\ssk

\begin{proof}[Proof of Theorem~$\ref{T-dis}$] 
Recall that $\U_n(\kappa)$ is the set of 
line-free curves \hbox{$\cC\in\fC_n$}
satisfying ${\bf\p_{max}}(\cC)<\kappa$ and ${\bf L_{max}}(\cC)<1-\kappa$.
 Given two curves $\cC_0$ and $\cC'_0$ in $\U_n(\kappa)$, 
first choose $\varepsilon$ small enough so that
the $\varepsilon$-neighborhoods $N_\varepsilon(\p)$ of the various
virtual flex points of $\cC_0$ are disjoint, and so that 
there exists a line intersecting any three of the 
neighborhoods $N_{2\varepsilon}(\p_j)$ only if the center points 
$\p_j$ are collinear. (Compare Figure~\ref{F-dis}. Such an $\varepsilon$ must 
exist since, if there were such a line for arbitrarily small
$\varepsilon$, then the center points would have to be collinear.) Furthermore, 
we also require the corresponding conditions for~$\cC'_0$.

Now let $\mathfrak N$ be the neighborhood of $\cC_0$ consisting of all curves
 $\cC\in \U_n$ such that, for every virtual flex point 
$\p$ of $\cC_0$,  the number of virtual flex points of $\cC$ in
$N_\varepsilon(\p)$, counted with flex-multiplicity,
is exactly the flex-multiplicity 
of $\p\in\cC_0$. Construct the neighborhood
 ${\mathfrak N}'$ of $\cC'_0$ in the  analogous way.
\medskip

\begin{figure}[h!]
\centerline{\includegraphics[width=2.2in]{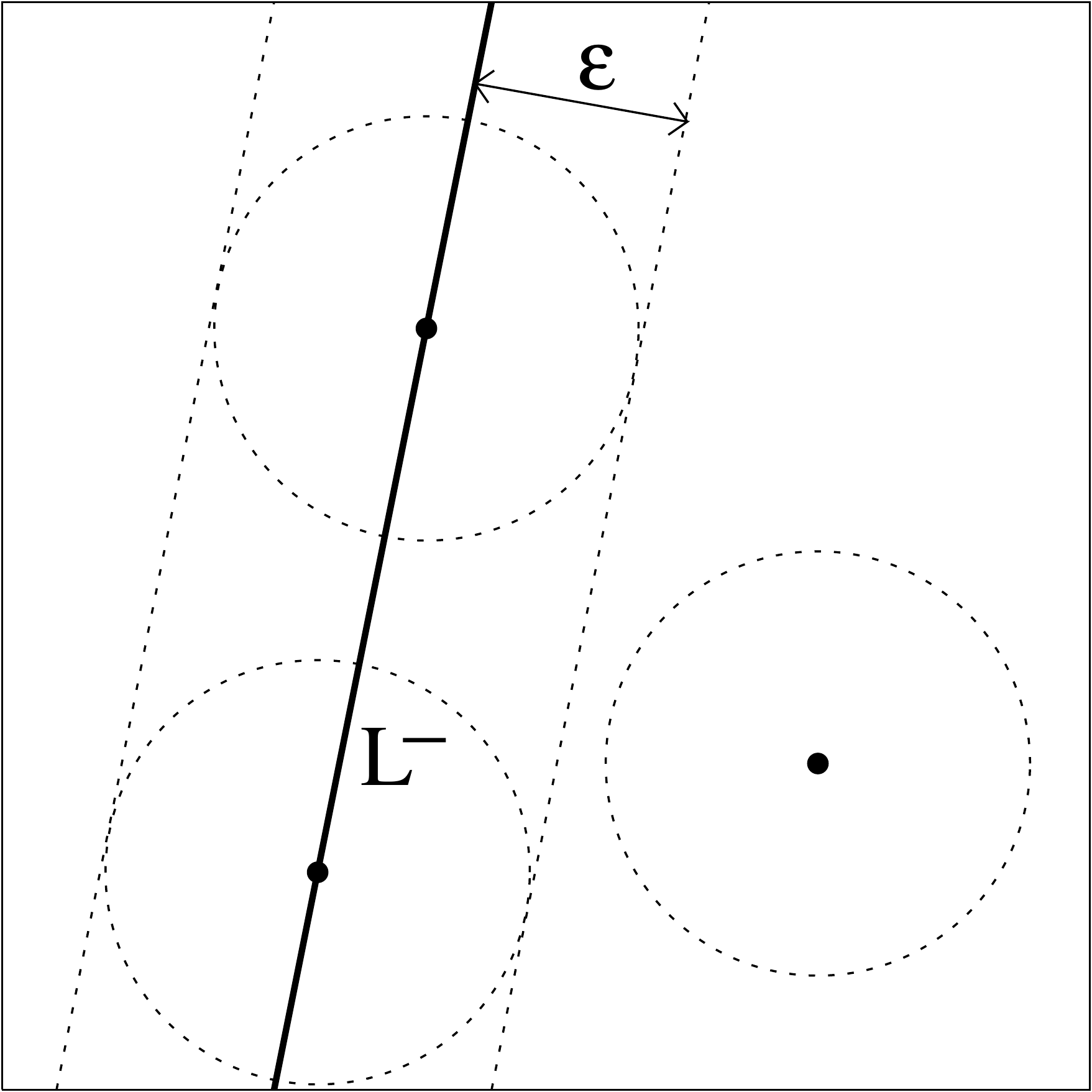}}
\caption{\label{F-dis} \sf Illustrating the proof of Theorem \ref{T-dis}.
For $\cC'\in{\mathfrak N}'$,
every virtual flex point of $\cC'$ must have distance
less than $\varepsilon$ from the corresponding virtual fixed point of $\cC'_0$.}
\end{figure}
\medskip

Choosing $K_\varepsilon$ as in Lemma \ref{L-dis2}, 
we must show that there cannot be any $\cC\in{\mathfrak N}$ and any 
$\g\not\in K_\varepsilon$ such that $\g(\cC)=\cC'$
belongs to ${\mathfrak N}'$. Replace  $\g$ by $\g^{-1}$ 
if necessary, so that we are in Case (2) of the Distortion 
 Lemma \ref{L-dis2}. Then 
$\widehat\bph_\cC\big(N_\varepsilon(\p^+)\big)<\kappa$, and it follows
from the Distortion Lemma that 
the $\varepsilon$-neighborhood of $L^-$ contains more that $1-\kappa$
virtual flex points of $\g(\cC)=\cC'$. This contradiction completes the 
proof of Theorem \ref{T-dis}.
\end{proof}
\bigskip

\begin{rem}[{\bf The Classical Moduli Space $\cM_\fg$}]
Since a smooth curve of degree $n$ in $\bP^2(\C)$ has genus
$\fg(n)={n-1\choose 2}$,
it is natural to compare the moduli space $\M^\sm_n(\C)$ for smooth curves
in $\bP^2$ with the classical moduli 
space $\cM_{\fg(n)}$, consisting\footnote{For $\fg\ge 2$ the moduli space
 $\cM_\fg$ can be considered as a quotient space ${\mathcal T}_\fg
/{\rm MCG}_\fg$, with the associated orbifold structure.
 Here ${\mathcal T}_\fg$ is the 
 $(3\fg-3)$-dimensional Teichm\"uller space, and ${\rm MCG}_\fg$, the
mapping class group, is a discrete group which acts on ${\mathcal T}_\fg$. 
See for example \cite{Hub}.}
 of conformal isomorphism
classes of closed Riemann surfaces of genus $\fg(n)$.
The dimension of this classical moduli space is given by
$$ {\rm dim}(\cM_\fg)~=~3\fg-3\quad{\rm for}\quad \fg\ge 2\,;\quad{\rm but}\qquad {\rm dim}(\cM_1)=1~.$$
(Compare \cite[p.28]{ACGH} or \cite[Ch. 5]{Mu}.) 
For every $n\ge 3$ there is a natural map $\M^\sm_n(\C)\to \cM_{\fg(n)}$. The
case $n=3$ is exceptional.  In this case, we obtain an isomorphism
$$ \M^\sm_3(\C)~\stackrel{\cong}{\longrightarrow}~\cM_1~,$$
where both spaces are isomorphic to $\C$, using the shape
invariant $\bJ$ of \S\ref{s-deg3} (which is just the classical
$j$-invariant, up to a multiplicative constant). Compare
 \cite{BM}. \medskip

Now let us assume that $n\ge 4$. For $n=4$ the map
$$ \M_4^{\sm}~\to~\cM_{3}~. $$
is fairly well understood: Any closed Riemann surface 
$\cC$ of genus $\fg$ has $\fg$ linearly independent holomorphic 1-forms, 
say $\omega_1,~\cdots,~\omega_\fg$, for any point $\p\in\cC$, the ratio
$$\big(\omega_1(\p):\cdots:\omega_\fg(\p)\big)$$ can be interpreted as a point
in the projective space $\bP^{\,\fg-1}(\C)$. Thus there is a canonical map 
from any Riemann surface of genus $\fg>1$ into $\bP^{\,\fg-1}(\C)$, well
defined  up to automorphisms of $\bP^{\,\fg-1}(\C)$. Furthermore any conformal
automorphism of $\cC$ corresponds to a change of basis 
for the vector space of 1-forms, and hence to a projective automorphism
of its image in the projective $(\fg-1)$-space.

By definition, a smooth complex curve $\cC$ of genus $\fg>1$ is called 
\textbf{\textit{hyperelliptic}} if it admits a meromorphic function
$\cC\to\bP^1(\C)$ of degree two. The following is proved for example
in \cite{Be} or \cite{Gr}. 
\medskip

\begin{prop} \label{P-hyper}
\it If $\cS$ 
is a Riemann surface of genus three which is not hyperelliptic,
then the canonical map $\cS\to \,\bP^2(\C)$ is a smooth embedding.
 Furthermore, every embedding of a Riemann surface of genus 3 into $\bP^2(\C)$
 can be obtained by this construction; and every conformal automorphism of
the Riemann surface corresponds to a projective automorphism of the
embedded curve. On the other hand, a hyperelliptic Riemann surface of genus
three cannot be embedded in $\bP^2(\C)$.
\end{prop}
\medskip

\begin{coro}\label{C-hyper}
The moduli space $\M^{\sm}_4(\C)$ for smooth projective curves of
degree four maps bijectively onto the open subset of $\cM_3$ consisting
of conformal equivalence classes of non-hyperelliptic Riemann surfaces
 of genus three.
Furthermore,  any conformal  automorphism of a smooth
projective curve of degree four extends to a projective automorphism of 
$\bP^2(\C)$.
\end{coro}
\smallskip

\begin{proof} This follows since a  smooth curve of genus 
three in $\bP^2(\C)$ necessarily has degree four;
and since any conformal automorphism of a Riemann surface  corresponds to 
a projective automorphism of its image in $\bP^{\fg-1}(\C)$.\end{proof}

For degrees $n\ge 5$ and hence $\fg(n)\ge 6$, we don't know whether the
map $\M_n\to\cM_{\fg(n)}$ is always injective. 
However, the dimension 
$$\dim\,\cM_{\fg(n)}~=~3\fg(n)-3~=~3(n^2-3n)/2$$
is larger than 
$$\dim\,\M_n=\dim\,\wfC_n-8~=~(n^2+3n-16)/2$$
in this case.  Therefore, only a very thin
set of curves of genus $\fg\ge 6$ can be embedded in $\bP^2(\C)$. And of course 
if $\fg$ is not of the form $n-1\choose 2$, then no curve of genus $\fg$ can be
 embedded as a smooth algebraic curve in $\bP^2(\C)$.  (However every 
Riemann surface can be immersed into $\bP^2(\C)$ with only simple double 
points. See for example \cite[Corollaries 3.6 and 3.11, Chapter~IV]{Ha}.)
\end{rem}
\bigskip 

\setcounter{lem}{0}
\section{Singularity Genus and Proper Action.}\label{s-G}\bsk

This Section will describe another criterion
for proper action, based on the genus invariant for a singular point.
However, we must first understand the genus of an arbitrary 
surface.\ssk

\subsection*{\bf The Genus of a Surface.}
 It will be convenient to work with the homology group
$H_1=H_1(\cS;~\Q)$, using rational coefficients.\footnote{One could equally
well use coefficients in any field. The field $\Z/2$ is particularly convenient
since then it is no  longer necessary to restrict to orientable surfaces. 
 However the resulting genus may be a half-integer. For example, the real
projective plane $\bP^2(\R)$, has genus 1/2 in this sense; 
and the hyperbolic dodecahedron of  Figure \ref{F-dodec} has genus 5/2.}
Choosing an orientation for $\cS$,
any two homology classes $\alpha,\,\beta\,\in\,H_1$ have a
well defined intersection number, yielding a skew-symmetric bilinear pairing
$$(\alpha,\,\beta)~\mapsto \alpha\cdot\beta~=~-\beta\cdot\alpha~\in~\Q~.$$

By a \textbf{\textit{surface}} $\cS$ we will mean a $C^1$-smooth,\footnote{
Smoothness is not really necessary, but makes proofs easier.}
oriented 2-dimensional Hausdorff manifold, possibly with $C^1$-smooth boundary. 
(In particular, any Riemann surface is also a surface in this sense.) The genus
$\fg(\cS)\ge 0$ of a surface is a fundamental topological invariant,
taking integer values (or  the value $+\infty$ in some non-compact cases).
It is defined as follows.
\msk

\begin{definition} The \textbf{\textit{genus}} $\fg(\cS)$ is defined to be the 
rank of this intersection pairing, divided by two. In other words, $2\,\fg$
is the dimension of the ``\textbf{\textit{reduced homology group}}''
$H_1/N$, where $N$  is the null space, consisting  of all
$\alpha$ such that $\alpha\cdot\beta=0$ for all $\beta$. As an example,
if $\cS$ is a compact surface with boundary, then $\fg$ is finite, and
 $N$ is generated by
the homology classes of the boundary circles.\footnote{By  
abuse of language, in this section
 we will use the word ``circle'' for any manifold
which is homeomorphic to the standard circle.}
Whenever $\fg$ is finite, 
it is not difficult to choose a basis for $\HG$ so that the matrix for this
 pairing consists of $\fg$ blocks of 
$\left(\begin{matrix}0&1\\-1&0\end{matrix}\right)$ along the diagonal, with
zeros elsewhere.
\end{definition}

As a classical basic example, if $\cC\subset\bP^2$ is a 
smooth complex curve of degree $n$, then the genus is given by 
\begin{equation}\label{E1}
\fg(\cC)~=~{n-1\choose 2}~.
\end{equation}
\bigskip

\begin{theo}\label{T-gen}
 Here are six basic properties of the genus.\msk

 \begin{itemize}
\item[\bf (1)] {\bf Additivity.} If $\cS$ is the disjoint union of two open
subsurfaces\footnote{It is customary to define genus only for connected 
surfaces, but this extension to non-connected cases is often convenient.}
$\cS_1$ and  $\cS_2$, then $\qquad \fg(\cS)~=~\fg(\cS_1)+\fg(\cS_2)$.\msk

\item[\bf (2)] {\bf Monotonicity.} If $~\cS\subset\cS'$, then 
$\quad \fg(\cS)~\le~\fg(\cS')~.$\msk

\item[\bf (3)] {\bf Puncture Tolerance.} The genus of $\cS$
is unchanged if we remove any finite subset from $\cS$. Similarly, it
is unchanged if we remove  the interior of a closed disk which is embedded
in the interior of $\cS$.\msk

\item[\bf (4)] {\bf Compact versus Non-Compact.} If $\cS$ is non-compact, then
$\fg(\cS)$ is equal to the supremum of the genera of compact
sub-surfaces. On the other hand, 
if $\cS$ is compact with boundary $\partial\cS$, then the genus of $\cS$ 
is equal to the genus of the interior $\cS\ssm\partial\cS$.
\msk

\item[\bf (5)] {\bf The Closed Surface Case.} For a compact surface $\cS$
with empty
boundary  $($or more generally for one such that each connected component 
has at most one boundary component$)$, 
the doubled genus  is equal to the first Betti number:
\begin{equation}\label{E2}
2\,\fg(\cS)~=~{\rm dim}\,H_1(\cS)~.
\end{equation}

\item[\bf (6)] {\bf Cutting and Pasting.} Let $\cS$ be a compact surface 
of genus $\fg$ with $\ell$ connected components  and with $b$ boundary 
circles. Then the Euler characteristic can be computed as
\begin{equation}\label{E-chi}
\chi(\cS)~=~ 2\ell~ -~2\,\fg~-~b ~.\end{equation}
If we form a new surface $\cS'$ by pasting together two  boundary 
circles of $\cS$, then $\chi(\cS)=\chi(\cS')$. $(\!$More precisely, 
the number of boundary circles decreases by two, but either the genus 
increases by one or the number of components decreases by one.$)$
\end{itemize}
\end{theo}
\msk

\begin{rem}
Felix Klein defined the genus
as the maximal number of disjoint non-separating circles which can be placed
on the surface. (Compare \cite[p. 230]{Gra}. In particular, a surface has
genus zero if and only if has the Jordan property of being separated
by any embedded circle.) 
It is an non-trivial exercise, using properties {\bf(4)} and {\bf(6)}, to
 check that this agrees with our definition of genus.
\end{rem}
\msk

\begin{proof}[Proof of Theorem~\ref{T-gen}]
The first five statements follow easily from corresponding
statements for  the ``reduced homology group'' $\HG$;
which are not difficult to check.\ssk

\noindent 
For {\bf(1)}: If $\cS=\cS_1\uplus\cS_2$,  then evidently  
$$ H_1(\cS)/N(\cS)~=~H_1(\cS_1)/N(\cS_1)~\oplus~ H_1(\cS_2)/N(\cS_2)~.$$

\noindent For {\bf(2)}: If $\cS\subset\cS'$, then $H_1(\cS)/N(\cS)$ maps 
injectively into $H_1(\cS')/N(\cS')$.\ssk

\noindent For {\bf(3)}: If $\cS'=\cS\ssm\{\p\}$, then $H_1(\cS')/N(\cS')$
maps isomorphically onto  $H_1(\cS)/N(\cS)$.\ssk

\noindent For {\bf(4)}: This follows since  $H_1(\cS)$ is the direct limit
of the homology groups of the compact subsets of $\cS$.\ssk

\noindent For {\bf(5)}, the Poincar\'e duality theorem for a compact oriented
$n$-manifold without boundary implies that the intersection pairing
$(\alpha,\,\beta)\mapsto \alpha\cdot\beta$ from $H_j\times H_{n-j}$ to $\Q$
is non-singular,\footnote{See for example \cite{GH}. The intersection
pairing in homology corresponds to the cup product pairing in the 
dual cohomology groups.}
 giving rise to an isomorphism from $H_j$ to
${\rm Hom}(H_{n-j},~\Q)\cong H^{n-j}$. In particular, for the special case 
$j=n-j=1$, it follows that the null space $N$ is trivial, so that  
$\HG=H_1$.\ssk

The proof of  {\bf (6)} will make use of properties of the Euler characteristic.
Note first the  additive property
\begin{equation}\label{E-chisum}
 \chi(X\cup Y)~=~\chi(X)+\chi(Y)-\chi(X\cap Y)~,\end{equation}
which clearly holds whenever $X$ and $Y$ are finite complexes with
$X\cap Y$ as a subcomplex. In the special case where $X\cap Y$ is a finite
union of circles, since the Euler characteristic of a circle is zero,
this simplifies to $\chi(X\cup Y)~=~\chi(X)+\chi(Y)~.$
As an example, suppose that $X$ is a compact connected surface of genus $\fg$
bounded by $b$ circles. Then we can choose $Y$ to be a union of $b$ closed
disks so that $X\cap Y$ is the union of these circles, and so that $X\cup Y$ 
is a closed surface of genus $\fg$. Then
$$\chi(X\cup Y)=2-2\fg~=~\chi(X)+\chi(Y)=\chi(X)+b~,$$
yielding the standard formula $\chi(X)=2-2\fg-b~.$ Now if we take the disjoint
union of $\ell$ such manifolds, since both $\fg$ and the number of boundary 
circles are additive, we obtain the required identity
$$ \chi(\cS)~=~ 2\,\ell(\cS)~-~2\,\fg(\cS)~-~b(\cS) $$
for any compact surface.  (Where $\ell(\cS)$ is the number of components
of $\cS$.)
\end{proof}

\begin{rem}\label{R-(6)} Here is a convenient consequence of property {\bf(6)}.
Suppose that a compact connected surface $\cS$ of genus $\fg$ can be obtained
 from a disjoint union of $\ell$ connected 
surfaces $\cS_j$  of genera $\fg_1,\,\ldots,\,\fg_\ell$ by pasting
together $k$ pairs of boundary circles. Then 
\begin{equation}\label{E-6'}
\fg~~=~~k\,+1-\ell +\,\sum_{j=1}^\ell \fg_j~.\end{equation}
In fact it follows from (\ref{E-chi}) that 
$$\chi(\cS_j)=2-2\fg_j-b_j~,\qquad{\rm and~that}
\qquad \chi(\cS)=2-2\fg-\Big(\sum b_j-2k\Big)~;$$
where the expression in parentheses is the number of boundary circles of $\cS$.
Equation (\ref{E-6'}) then follows easily 
since $\quad\chi(\cS)=\sum\chi(\cS_j)$.
\end{rem}
\msk

\subsection*{\bf The Genus of a Singularity.}
Let $\p$ be a (necessarily isolated) singular point of a complex  
curve $\cC\subset\bP^2$. If $N_\p$ is the $\vep$-ball centered at $\p$,
using the standard Study-Fubini metric, and if $\vep$ is small enough,
then for every smooth curve $\cC'$ of the same degree
which approximates $\cC$ closely enough (depending on $\vep$),
the intersection $\cS_\p=\cC'\cap\overline N_\p$ is a smooth compact
connected surface with $b$ boundary components, where $b\ge 1$ is the number
of local branches of $\cC$ through the point $\p$; and where the genus
of $\cS_\p$ is independent of the choice of $\cC'$. This is proved
\footnote{To see that
such an $\vep$ exists, note that the set of $\vep'$ for which the intersection
is not transverse has measure zero by Sard's Theorem. In fact, it is a
semi-algebraic set, and hence must be finite.} for example in \cite{Mi1} or
\cite{Wa}.
\msk

\break

\begin{figure} [t!]
\centerline{\includegraphics[width=4.5in]{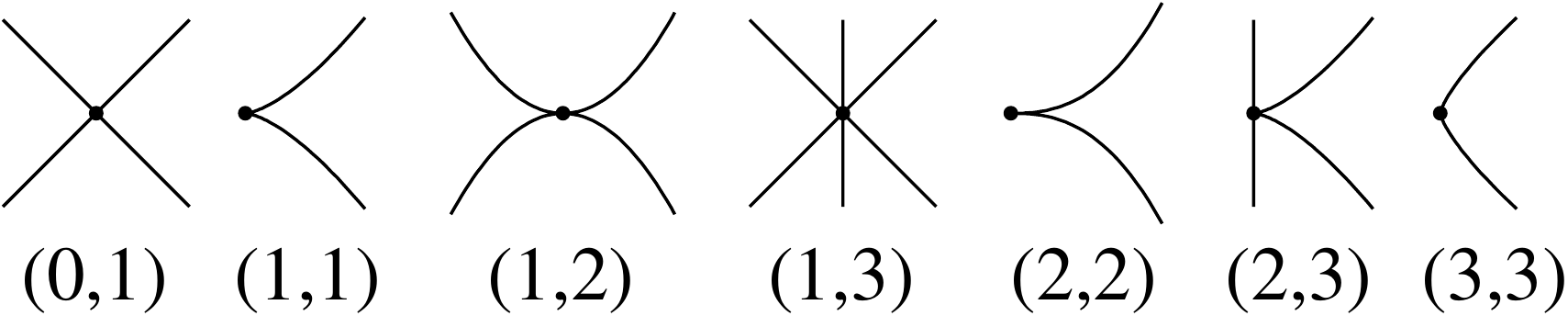}}
\caption{Showing the pair of invariants $(\fg,\,\fg^+)$ for seven examples
of singular points.\protect\footnotemark[29]
\label{F-7s}}
\end{figure}

\footnotetext[29]{These singularities are respectively a simple node, 
a $(2,3)$-cusp, a tacnode, a triple crossing, a $(2,5)$-cusp, 
a $(2,3)$-cusp together with a non-tangent line, and a $(3,4)$-cusp. For the
computation of these numbers, see Example \ref{ex-cuspcomp}.}
\setcounter{footnote}{29}
\def\gmg{{{\fg}_{\sf geom}}}

\begin{definition} By the \textbf{\textit{genus of the  singularity}}, 
denoted by $\fg_\cC(\p)$, we will mean the genus of this surface  $\cS_\p$. 
It will  also be useful to consider the \textbf{\textit{augmented genus}}
(or $\boldsymbol\delta$-\textbf{\textit{invariant}}\footnote{
See for example
\cite[pp.59-65]{Ser} or \cite[p.126]{Nam}. We have preferred
 the $\fg^+$ notation since
it makes the connection with genus clearer.}) 
$$ \fg^+_\cC(\p)~=~ \fg_\cC(\p) + b-1~.$$
\end{definition}
\ssk

\begin{definition} Let $\cC$ be a singular projective curve 
with  singular points $\p_1,\,\ldots\p_m$. The
 genus of the non-singular open subset 
$$ \cC\ssm\{\p_1,\,.\ldots,\, \p_m\}~.$$
is called\footnote{This is the usual definition in the case of an irreducible
  curve. Our $\gmg(\cC)$ is just the sum of the geometric genera of the
  irreducible components of $\cC$.}
the \textbf{\textit{geometric genus}} $\gmg(\cC)$. 
\end{definition}

Here is an example to illustrate these definitions.

\begin{lem}[\bf Degree-Genus Formula]\label{L-dg} With $\cC$ as above, we have
\begin{equation}\label{E-dg}
 {\gmg}(\cC)~+~\sum_{j=1}^m\fg^+_\cC(\p_j)~=~{n-1\choose 2}~+~r-1~,
\end{equation}
where $n$ is the degree of $\cC$ and $r\ge 1$  is its number of 
irreducible components.\end{lem}

\begin{proof}
Choose a small closed ball around each $\p_j$ and let $\cC'$ be a 
smooth degree $n$ curve which closely approximates $\cC$. Let $\cS_j$ be the 
intersection of $\cC'$ with the ball around $\p_j$ and let $\cS'$ 
be the closure
of $\cC'\ssm\big(\cS_1\cup\cdots\cup\cS_m\big)$. If the balls are small enough
and the approximation is close enough, then each $\cS_j$ will have genus 
$\fg(\p_j)$, and will have $b_j$ boundary circles, where $b_j$ is the number of
local branches of $\cC$ at $\p_j$. Furthermore, $\cS'$ will be a smooth curve
with $\sum_j b_j$ boundary circles, and with $r$ connected components $\cS'_k$,
where $r$ is the number of irreducible components of $\cC$; and with
$\fg(\cS')$ equal to the geometric genus ${\gmg}(\cC)$. Now applying Equation
(\ref{E-6'}) to  the surface $\cC'$, which is the union of the $\cS_j$
together with the $r$ components $\cS'_k$of $\cS'$, pasted together along
$\sum b_j$ boundary circles, we see that
$$ \fg(\cC')~=~\sum_{j=1}^m b_j \,+\,1\,-\, (m+r) \,+\,
\Big(\sum_{j=1}^m \fg(\cS_j) +\sum_{k=1}^r \fg(\cS'_k)\Big)~. $$
Here the left side is equal to $n-1\choose2$, while the right side 
can be rearranged as
$$ \sum_1^m\Big(\fg(\cS_j)+ b_j-1\Big)~+~\gmg(\cC) +1-r~~=~~
\sum_1^m\fg^+(\p_j)~+~\gmg(\cC) +1-r~.$$
The required equation (\ref{E-dg}) now follows easily.
\end{proof}
\msk

\begin{rem} \label{R-mn}
 The numbers $\fg_\cC$ and $\fg^+_\cC$ are closely related
  to the ``Milnor number'' $\bmu$. (See for example \cite{Mi1},
  \cite{Wa}, \cite{Ghy}, \cite{Sea}.) Using affine coordinates $(x,\,y)$,  
this number $\bmu$ for the curve $F(x,\,y)=0$ at a point $\p$ can be defined 
as the intersection multiplicity between the curves $F_x=0$ and $F_y=0$ at
$\p$, where the subscripts indicate partial derivatives. 

If $\p=(0,\,0)$,
then $\bmu$ can be computed as the dimension of the quotient algebra 
$\C[[x,\,y]]/(F_x,\,F_y)$, where $\C[[x,\,y]]$ is the ring of formal 
power series in two variables and $(F_x,\,F_y)$ stands for the ideal 
generated by these two partial derivatives.\footnote{Compare \cite[p.9]{Fu}.}
It follows easily that $\bmu>0$ if and only if $\p$ is a singular point of
$\cC$.
\end{rem}
\ssk

\begin{ex}\label{ex-cuspcomp} For a cusp curve with
  $~F(x,\,y)=x^p-y^q=0~,$
the  quotient algebra $\C[[x,\,y]]/(F_x,\,F_y)$ has an additive 
basis consisting of the $(p-1)(q-1)$ monomials $x^jy^k$
with $0\le j<p-1$ and $0\le k<q-1$. Therefore $\bmu=(p-1)(q-1)$.
\end{ex}
\bigskip

\begin{lem}\label{l-Mil-num}
  The Milnor number $\bmu$ is the sum of the genus $\fg$, and
  the augmented genus $\fg^+$. That is,
  \begin{equation}\label{E6} \bmu~=~\fg\,+\,\fg^+~.\end{equation}
  
\end{lem}
\ssk

\begin{proof} According to  \cite[Theorem 7.2]{Mi1}, $\bmu$ is equal to
  the first Betti number ${\rm dim}\big(H_1(\cS_\p)\big)$ of the surface $
  \cS_\p$. We must show that the sum
$$\fg+\fg^+~=~2\,\fg+ b-1$$
is equal to this Betti number ${\rm dim}\big(H_1(\cS_\p)\big)$.
Recall from Theorem \ref{T-gen}(6) that the Euler characteristic
of a connected surface of genus $\fg$ with $b$ boundary components
is $2-2\fg-b$. Comparing this with the standard expression
$$ \dim(H_0)-\dim(H_1)+\dim(H_2) $$
for the Euler characteristic, we obtain
$$ 2-2\fg-b~=~ 1-\dim(H_1)+0~,$$
and hence $2\fg+b-1=\dim(H_1)$. The equation (\ref{E6}) follows.
\end{proof}
\medskip

Since $0\le\fg^+\le\bmu\le 2\,\fg^+$, it also follows that $\fg^+>0$
if and only if $\p$ is a singular point. \bigskip

Consider again the cusp curve $x^p=y^q$ of Example~\ref{ex-cuspcomp}.  If $p$
and $q$  are relatively prime so that
the number of branches is $b=1$, 
then $\fg^+=\fg$, and it follows that $$\fg=\fg^+=(p-1)(q-1)/2~.$$
On the other hand, if $p$ and $q$ have greatest common divisor $\delta>1$,
then  there are $\delta$ branches, and a similar argument shows that
$$\fg~=~\frac{(p-1)(q-1) +1-\delta}{2}\quad{\rm and}
\quad\fg^+~=~\frac{(p-1)(q-1) +\delta-1}{2}~.$$
For the simplest case $p=q=\delta=2$, the curve $~x^2-y^2=(x+y)(x-y)=0~$
has a simple crossing point at the origin, and we obtain 
$(\fg,\,\fg^+)=(0,1)$ as listed in Figure \ref{F-7s}. Similarly for $p=2,~
q=4$, the equation $(x^2-y^4)=(x+y^2)(x-y^2)=0$ defines a tacnode, as shown
in the figure. This takes care of five of the examples in the figure, and
the remaining two can be checked by similar computations.
\bsk

\begin{rem}[{\bf Erratum}] In \cite[p.~60]{Mi1}, it was stated incorrectly
that the invariant $\bmu$ is equal to the classical
multiplicity of the singularity. In fact the \break
\textbf{\textit{multiplicity}} $\bf m$  of a singular point  $\p\in\cC$ 
is defined to be the intersection multiplicity at $\p$ between $\cC$
and a generic line through $\p$. The following examples show that neither
of these two invariants at the point $x=y=0$ determines
 the other.
$$\begin{matrix} {\bf curve} &&{\bf m}&& {\bmu}\\
x^3=y^5 && 3 && 8\\
x^3=y^7&& 3 && 12\\x^4=y^5&& 4 && 12\\
\end{matrix}$$
Note that the multiplicity $\bf m$ for a singular point of a curve 
of degree $n$  satisfies $$2\le {\bf m}\le n~.$$
The set of all curves in ${\mathfrak C}_n$ which have a singularity of 
multiplicity $\bf m$ or larger forms a closed algebraic subset of codimension 
${{\bf m}+1\choose 2}-2$ in ${\mathfrak C}_n$. 
The proof is similar to the proof of Proposition \ref{p-ass} in 
\S\ref{s-genR}. 
\end{rem}

\msk

\begin{ex}\label{exa1-ap3} Let $\cC$ be a curve of degree $n=4$ consisting
of a smooth cubic curve together with its tangent line at a flex point $\p$.
Since \hbox{$\gmg(\cC)=1+0$} and $r=2$, it follows from Equation (\ref{E-dg})
that
$$ \fg^+_\cC(\p)~=~{3\choose 2} +r-1-\gmg(\cC)~=~3~,$$
and hence that $\fg_\cC(\p)=2$. We can check this statement 
by a different argument as follows.
Let $F(x,y)=y(x^3-y)$, so that the locus $F=0$  is locally the
union of a smooth cubic curve and the tangent line at a flex point.
Then $F_x=3\,x^2y$ and $F_y=x^3-2\,y$. Therefore, modulo the ideal 
$(F_x,\,F_y)$ we have $x^2y\equiv 0$ and $x^3\equiv 2y$. 
It follows easily that
the quotient algebra is generated by $x$, with $x^5\equiv 0$, so that the 
dimension is $\bmu=5$. Since the number of local branches is $b=2$, it 
follows again that $\fg=2$ and $\fg^+=3$.
\end{ex}\msk

\begin{lem}[{\bf Multi-Branch Lemma}]\label{L-mb} \it The augmented
 genus of a singularity with 
$k$ local branches $\cB_1\,~\ldots\,,~\cB_k$ is given by the formula
$$ \fg^+_\cC(\p)~=~\sum_j\fg_{\cB_j}(\p)
~+~\sum_{i<j} \cB_i\cdot\cB_j~,$$
where $ \cB_i\cdot\cB_j$ is the intersection number between the two
branches. As an example, if there are $k$ smooth branches intersecting 
pairwise transversally, then  $\fg^+_\cS(\p)={k\choose 2}$.
\end{lem}\msk

(Compare the analogous formula for flex-multiplicity
in Example~\ref{EX-mu}.)
\ssk

\begin{proof}[{\bf Outline Proof}]$\!\!$\footnote{For a detailed proof of
an equivalent statement, see \cite[Th.~6.5.1]{Wa}.} 
First choose a fixed small round neighborhood $N$ of $\p$, and choose
generic small 
translations $\cB_j+{\bf v}_j$ of the various branches so that each
one still intersects $\partial N$ transversally, and so that any two
translated branches intersect transversally in $\cB_i\cdot\cB_j\ge 1$
distinct points. 
 Then approximate each translated branch very closely by
 a smooth curve. Thus we are reduced to the case of smooth curves 
intersecting transversally.   The disjoint union of the resulting smooth 
curves will have $k$ components, each with one boundary curve, and will
have genus $\sum \fg_{\cB_i}(\p)$. A smooth curve which is close to the 
actual union of these transversally intersecting curves will be homeomorphic 
to the object obtained by removing a small round neighborhood
of each transverse intersection point, and then gluing the 
$2\sum\cB_i\cdot\cB_j$ resulting boundary circles together in pairs.
By Theorem \ref{T-gen}{\bf(6)}, each such pasting must either increase the
genus by one or decrease the number of components by one. Since the total 
effect is to decrease the number of components from $k$ to one, the final
genus must be
$$ \fg_\cC(\p) ~=~1-k~+~\sum_i\fg_{\cB_i}(\p) ~+~
\sum_{i<j} \cB_i\cdot\cB_j~.$$
Adding $k-1$ to both sides, the conclusion follows.
\end{proof}
\medskip

\subsection*{\bf Proper Action.} 

First, as in Section \ref{s-cc},
 consider only line-free curves. Let $U\subset\fC_n$ be some $\bG$-invariant
open set consisting of curves which contain no lines. Let 
$\max_U\fg\le\max_U\fg^+$ be the maximum values of $\fg_\cC(\p)$ and
$\fg^+_\cC(\p)$ as $\cC$ ranges over $U$ and $\p$ ranges over $\cC$.

\begin{prop}\label{P-pa1}
Suppose that the following two conditions are satisfied:\ssk

{\bf(1)} $\qquad
\max_U \fg\,+\,\max_U\fg^+~<~{n-1\choose 2}~.$\ssk

{\bf(2)} No curve in $U$ is separated by a single point.\ssk

\noindent Then the action of $\bG$ on
$U$ is proper, and hence the open set $U/\bG\subset\M_n$ is 
a Hausdorff orbifold.
\end{prop}

As an example, since $\fg\le\fg^+$, Condition {\bf(1)} will be satisfied
if the maximum value of $\fg^+_\cC(\p)$ satisfies the following,
$$\begin{matrix} n=& 3&4&5&6&7&8&9&10\\
\max\fg^+\le & 0&1&2&4&7&10&13&17
\end{matrix}$$
(with somewhat sharper results when the maximum value of $\fg$ is less
than the maximum value of $\fg^+$). Thus as the degree increases, we can 
allow more and more complicated singularities.\msk

\begin{rem} 
The conditions {\bf(1)} and {\bf(2)} are independent of each other when
$n$ is large enough. To see this,
note that a curve $\cC$ is disconnected by a single
point $\p$ only if it can be described as a union $\cC=\cC_1\cup\cC_2$, where
 the curves $\cC_1$ and $\cC_2$ intersect only at $\p$, necessarily with
 intersection
multiplicity equal to the product $n_1n_2$ of degrees. As an example,
let $\cC_1$ be the smooth curve 
$$ y^{n-2} ~=~(y\,z-x^2)f(x,y,z) $$
of degree $n-2$, where $f(x,y,z)$ is a homogeneous function of degree $n-4$
 with $f(0,0,1)\ne 0$;  and let $\cC_2$ be the curve $y\,z=x^2$ 
 of degree 2. Then it is easy 
to check that the intersection $\cC_1\cap\cC_2$ consists of the single 
 point $\p$ with coordinates
$(x:y:z)=(0:0:1)$, and that both $\cC_1$ and $\cC_2$ are 
smooth near this point. Thus there are just two local branches of $\cC$ at 
$\p$.  
By B\'ezout's Theorem  the total intersection number of $\cC_1$ and $\cC_2$ is
 the product of degrees $2(n-2)$. Since there is only one intersection point, 
the local intersection number at $\p$ is precisely $2(n-2)$. 
Thus it  follows from Lemma~\ref{L-mb} that $\fg^+_\cC(\p)=2(n-2)$. 
Since there are two local branches, it follows that $\fg_\cC(\p)=2(n-2)-1$.
It is then easy to check 
 that $\fg+\fg^+ < {n-1\choose 2}$ whenever $n\ge 9$; so that
 we obtain curves which satisfy {\bf(1)} but not {\bf(2)}. On the other hand,
it is not hard to find curves which satisfy {\bf(2)} but not {\bf(1)}.
\end{rem}\msk

The proof of Proposition \ref{P-pa1} will make use of the Distortion 
Lemma~\ref{L-dis2},
which involves not only points but also lines. In order to apply it,
 we will need the following.
\msk

\begin{definition}
Given a line $L\subset\bP^2$ and given $\vep>0$, let $N_\vep(L)\subset\bP^2$
be the open $\vep$-neighborhood, using the standard Study-Fubini metric.
Then for any smooth curve $\cC\in\fC_n$,
the intersection $\cC\cap N_\vep(L)$ has a well defined genus $0\le \fg\le
{n-1\choose 2}$.
Similarly, given a singular curve $\cC_0\subset\bP^2$, there is a well
defined number $$ \limsup_{\cC\to \cC_0}~ \fg\big(\cC\cap N_\vep(L)\big)~,$$
where $\cC$ varies over smooth curves converging to $\cC_0$ within the space
$\fC_n$. Therefore the monotone limit
\begin{equation}\label{E-mlim}
  \fg_{\cC_0}(L)~=~\lim_{\vep\to 0}\big(
\limsup_{\cC\to \cC_0}\, \fg(\cC\cap N_\vep(L)\big) \end{equation}
is also well defined. In fact for  a generic choice of $\vep$, the boundary
of $N_\vep(L)$ is transverse to $\cC_0$, so that it doesn't matter whether 
we take the lim sup or the \hbox{lim inf} in equation (\ref{E-mlim}).
\end{definition}

\begin{lem}\label{L-Lmax}
Let $\max_L\, \fg_{\cC_0}(L)$ be the maximum of the expression
$(\ref{E-mlim})$ over all lines
$L\subset \bP^2$. Then there exists a number $\vep_0=\vep_0(\cC_0)>0$ such that
$$\fg\big(\cC\cap N_{\vep_0}(L_0)\big) ~\le~ \max_L \fg_{\cC_0}(L)$$
for every line $L_0$ and every smooth $\cC$ which is sufficiently close 
to $\cC_0$.
\end{lem}

\begin{proof}
Otherwise for every $\vep>0$ there would be a line $L_\vep$ and curves $\cC$
arbitrarily close to $\cC_0$ for which $\fg\big(\cC\cap N_\vep(L_\vep)\big)>
 \max_L \fg_{\cC_0}(L)$. Choose a sequence $\{\vep_j\}$ converging to zero 
so that the corresponding lines $L_{\vep_j}$ converge to some limit $L_0$.
Then $N_{\vep_j}(L_{\vep_j})\subset N_{\vep_0}(L_0)$ for large $j$; and
we obtain a contradiction, using monotonicity of the surface genus.
\end{proof}
\medskip

If the curve $\cC_0$ contains no line, then we can sharpen this statement
as follows.

\begin{lem} \label{L-nl}
If $\cC_0$ is line-free, then for every $\vep>0$ there exists  $\delta>0$
with the following property. For any line $L\subset\bP^2,$ each connected 
component of the intersection $~\cC_0\cap N_\delta(L)~$ has diameter less 
than $\vep$. In practice, we will choose $\vep$ less than the smallest distance 
between two singular
points of $\cC_0$. It then follows that each such connected component 
contains at most one singular point. Hence it follows
that $\fg_{\cC_0}(L)$ is just the sum of $\fg_{\cC_0}(\p)$ as $\p$ ranges over
all singular points of $\cC_0$ in $L$.
\end{lem}

\begin{figure}[t]
\centerline{\includegraphics[width=2.3in]{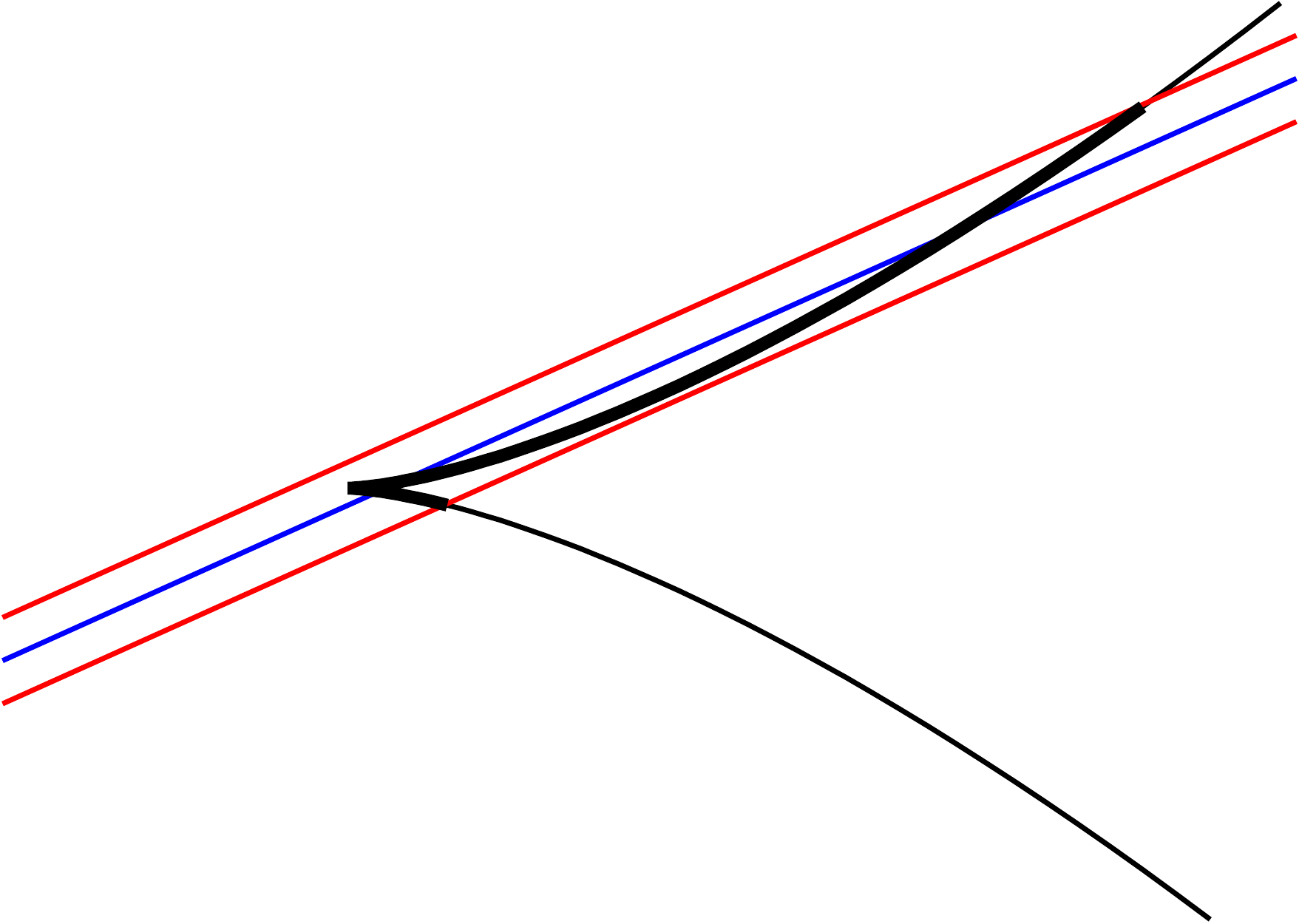}}
\caption{Illustrating Lemma \ref{L-nl}}
\end{figure}

\begin{proof}[Proof of Lemma \ref{L-nl}] Otherwise, for some fixed
 $\vep_0>0$, we could choose a sequence $\{\delta_j\}$ converging to zero,
and an associated sequence of lines $L_j$,
such that for each $j$ some component of $~\cC_0\cap N_{\delta_j}(L_j)~$
has  diameter $\ge\vep_0$. After passing to an infinite subsequence, we
may assume that $\{L_j\}$ converges to a limit line $L'$.
It then follows that the  
intersection of any neighborhood of $L'$ with $\cC_0$ has one or more 
components of diameter $\ge\vep_0$.
Since any nested intersection of compact connected sets is again connected,
it would follow that $L'\cap\cC_0$ has a component of length $\ge\vep_0$,
which is impossible since $\cC_0$ is line-free.
\end{proof}
\bigskip

Next we must look at the intersection of $\cC_0$ with a small round ball.

\begin{lem} \label{L-CcapB}
Given any singular curve $\cC_0\in\fC_n$ there exist numbers\break
$\vep_1>\vep_2>0$ such that 
the $\vep$-sphere centered at any singular point
of $\cC_0$ intersects $\cC_0$ transversally whenever $0<\vep\le \vep_1$;
and furthermore such that any open ball of radius $<\vep_2$ either:
 \ssk
 
 \begin{itemize}
 \item[{\bf(a)}] is contained in the $\vep_1$-ball about some singular point; or
   else

\item[{\bf(b)}] intersects $\cC_0$ in  a topological disk or in the empty set.
\end{itemize}
\end{lem}

\begin{figure} [t]
\centerline{\includegraphics[width=2in]{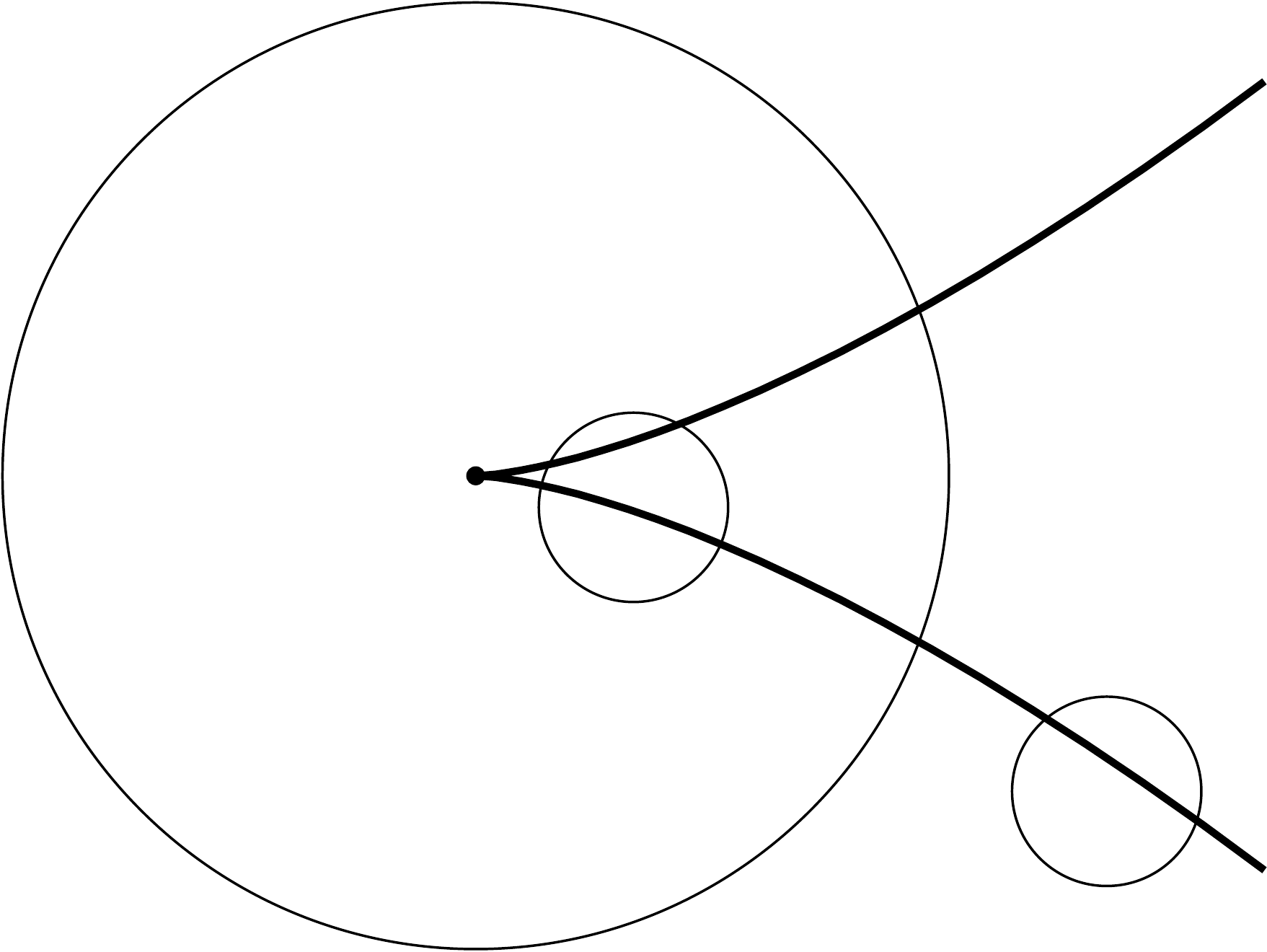}}
\caption{Illustration for Lemma \ref{L-CcapB}, showing a (larger) $\vep_1$-ball
and two (smaller) $\vep_2$-balls.\label{F-ball-fig}}
\end{figure}

The proof is not difficult. (Compare Figure \ref{F-ball-fig}.) \qed 

\begin{proof}[Proof of Proposition \ref{P-pa1}]
  If the action of $\bG$ on the open set $U\subset\fC_n$ were not proper, then
we could find curves $\cC_1(k)$ converging to some $\cC_1\in U$, and curves
$\cC_2(k)$ converging to some $\cC_2\in U$, and group elements $\g_k$ diverging
to infinity in $\bG$ so that $\g_k\big(\cC_1(k)\big)=\cC_2(k)$.
Since the number $\vep_1$ of Lemma \ref{L-CcapB} can be arbitrarily small,
we can assume without loss of
generality that $\vep_1$ is smaller than the number $\vep_0$ 
of Lemma \ref{L-Lmax} and the number $\delta$ of Lemma \ref{L-nl} both for
$\cC_1$ and for $\cC_2$. Then we can choose
$k$ large enough so that $\g_k\not\in K_{\vep_2}$. According to the Distortion
Lemma~\ref{L-dis2} (interchanging the roles of $\cC_1$ and $\cC_2$ if
necessary), we can find a ``repelling'' neighborhood $N_{\vep_2}(\p^+)$ and an
``attracting'' neighborhood $N_{\vep_2}(L^-)$ so that
every point outside of $N_{\vep_2}(\p^+)$ maps into 
$N_{\vep_2}(L^-)$ under the action of $\g_k$. 

First suppose that we are in Case {\bf(a)} of Lemma \ref{L-CcapB} . Then we can
replace the disk $N_{\vep_2}(\p^+)$ by a larger disk $N_{\vep_1}(\p')$,
where $\p'$ is a singular point of $\cC_1$. Note that $\cC_1$
 intersects this larger disk transversally, hence the same is true for any
$\cC$ which is sufficiently closed to $\cC_1$. Then the boundary of
$N_{\vep_1}(\p')$ cuts $\cC$ into:

\begin{quote}
\begin{enumerate}
\item a part $\cC_{\sf in}$ inside this disk which is connected, with
  $$ \fg(\cC_{\sf in})+b(\cC_{\sf in})-1$$
(genus plus number of boundary curves minus one)
equal to the augmented genus $\fg^+_{\cC_1}(\p')$; and\ssk

\item a part $\cC_{\sf out}$ outside
of $N_{\vep_1}(\p')$ which is connected and can be embedded into
 $\g_k(\cC)\cap N_{\vep_2}(L^-)$.
\end{enumerate}
\end{quote}

\noindent It follows from Lemma \ref{L-Lmax} that this second part 
 has genus at most equal to the
 maximum genus of the singular points of $\cC_2$ within $N_{\vep_2}(L^-)$.
Since both parts are
 connected, it follows  from 
 the cutting and pasting formula~(\ref{E-6'}) that
 $$ \fg(\cC)~\le~\fg^+_{\cC_1}(\p')+ \max_\p\fg_{\cC_2}(\p)~.$$
 where $\p$ ranges over singular points of $\cC_2$ within $N_{\vep_2}(L^-)$.
 Since $\cC$ is a smooth curve of 
degree $n$, this contradicts Hypothesis {\bf(1)} of Proposition \ref{P-pa1}.
This completes the proof of this Proposition in Case {\bf(a)}.
The proof in Case {\bf(b)} is similar but easier.
\end{proof}
\medskip

\begin{rem}
If we allow curves which contain lines, then the following slightly weaker 
statement still holds: If

\begin{equation}\label{E-pa2}
\fg_{\cC_1}^+(\p)+\fg_{\cC_2}(L)~<~{n-1\choose 2} \end{equation}
for every $\cC_1$ and $\cC_2$ in $U$  and every $\p\in\cC_1$ and 
$L\subset\bP^2$, then the action of $\bG$ on $U$ is proper. 
The proof is similar to the argument above.
\end{rem}
\medskip

Here is an easy consequence.
\medskip

\begin{coro}\label{C-simpsing} \it Let $U'\subset\fC_n$ be the open set
 consisting of curves with no singularities other than simple double points 
and cubic cusps $($or equivalently with $\fg^+(\p)\le 1$ for all singular
points$)$. If $n\ge 4$, then the action of $\bG$ on $U'$ is proper, 
and hence the open set
$U'/\bG\subset\M_n$ is a Hausdorff orbifold. 
\end{coro}

\begin{proof} First consider the open subset of $U'$ consisting of 
curves $\cC_0$ which contain no lines. Since we have assumed that 
$\fg(\p)\le\fg^+(\p)\le 1$ for all singular points,
and since ${n-1\choose 2}\ge 3$ for $n\ge 4$, 
the conclusion in this special case 
follows easily from Proposition \ref{P-pa1}. 

To prove the full Corollary,
we must also show that $\fg_{\cC_0}(L)\le 1$ for every line
which is contained in $\cC_0$. In fact 
it follows from the hypothesis that the only
 singularities which can be contained in $L$ are simple double points.
Therefore, it is not hard to check that the 
surface $\cS_L=\cC\cap N_\vep(L)$, with $\vep$ small and 
\hbox{$\cC\approx \cC_0$},
is homeomorphic to the line $L$ itself with each singular point removed. Thus 
\hbox{$\fg_{\cC_0}(L)=0$}, and the conclusion follows.
\end{proof}
\bigskip

Note that the condition $n\ge 4$ is essential.  
For a cubic curve with a cusp point, the quotient is not even a $T_1$-space;
while for a cubic curve with a double point, the action is not proper 
(although the quotient is Hausdorff). Note also that the Corollary applies to a 
union of four lines in general position; but not to a union of four lines 
where they pass through a common point. (Compare  Figure \ref{F-W4}.)
\msk

\begin{rem}[{\bf 1-cycles}]\label{R-cyc} 
If we consider 1-cycles rather than curves, with
multiplicities  allowed, then the arguments become more difficult
 since every point of a curve of
multiplicity two or more is singular. As a consequence,  in the definition
of the genus associated with a point of $\cC$ or a line through $\cC$ we must
 take the lim-sup over all possible smooth approximating curves.

 As a simplest example, consider a 1-cycle
of the form $\cC=\cC_{n-2}+2\cdot L$ where $\cC_{n-2}$ is a generic
smooth curve of degree $n-2$ and $L$ is a generic line counted twice.
Then one can check that the largest value 
of $\fg^+$ at a point is $\fg^+(\p)=3$, corresponding to an intersection
point in $\cC_{n-2}\cap L$. Similarly, the largest value of $\fg$ on a line
(when $n\ge 4$)  
is $~\fg(L)=n-2~$ for the doubled line $L$.
The required inequality
$$\max\fg^+({\rm point})~+~\max\fg({\rm line})~<~{n-1\choose 2}$$
then reduces to $~~3+(n-2)<(n-1)(n-2)/2\,$; which is satisfied if and only if
$n\ge 6$. Thus we can conclude that $\M_n$ is locally Hausdorff at
$\((\cC_{n-2}+2\cdot L\))$ whenever $n\ge 6$. Details of the argument will be 
omitted.
\end{rem}
\bigskip

\setcounter{lem}{0}
\section{Automorphisms and W-curves.} \label{s-aut}
Following Klein and Lie, a complex curve\footnote{Klein and Lie \cite{KL}  
also considered transcendental curves (such as the logarithmic spiral) which
are invariant under a one-parameter group; but we consider only
algebraic curves.}
 $\cC\in\fC_n$ is called a W\textbf{\textit{-curve}} if it is
invariant under a one-parameter group of projective transformations
(or equivalently, if it has infinite stabilizer\footnote{Recall from
Remark \ref{R-stab} that every infinite stabilizer is a Lie group, and
hence contains a one-parameter Lie group.}).
We will use the notation $\fW_n\subset\wfC_n$ for the
algebraic set consisting of all curves or cycles
in $\wfC_n$ which have infinite stabilizer. 
This algebraic set $\fW_n$ is reducible for all $n\ge 2$. (It is equal to the
entire space $\wfC_n$ for $n\le 2$.)

Recall from Remark~\ref{R-stab} that a curve or cycle
$\cC$ has finite  stabilizer whenever its orbit
(= projective equivalence class)
$$  \((\cC\))~=~\{g(\cC)~;~g\in \bG=\PGL_3\}~\subset~\wfC_n $$
has dimension equal to $\dim(\bG)=8$, and has infinite stabilizer whenever
its orbit has dimension strictly less than $8$. It will be enough
to study curves, since it is easy to check that a cycle $\cC$ has finite
stabilizer if and only if its support $|\cC|$ has finite stabilizer. A detailed 
 classification of curves with infinite stabilizer has been provided by 
\cite{AF1}. (See also \cite {Ghi}, \cite{Pop}.)

Since many different W-curves may be invariant under the same group,
it is convenient to first list the possible connected Lie groups
which can serve as the identity component $\bG_\cC^0$ of some
stabilizer. The largest groups, with dimension two or more, are relatively
 easy to describe.
\ssk

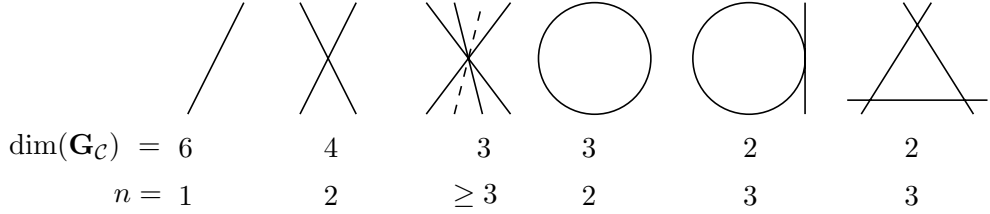
\begin{figure}[h!]
\labellist
\pinlabel $\dim(\bG_\cC)~=$ [l] at -115 -20
\pinlabel $6$ [l] at   -10 -20
\pinlabel $4$ [l] at    80 -20
\pinlabel $3$ [l] at   175 -20
\pinlabel $3$ [l] at   240 -20
\pinlabel $2$ [l] at   340 -20
\pinlabel $2$ [l] at   440 -20

\pinlabel $n=$ [l] at -50 -50
\pinlabel $1$ [l] at  -10 -50
\pinlabel $2$ [l] at   80 -50
\pinlabel $\geq 3$ [l] at 160 -50
\pinlabel $2$ [l] at  240 -50
\pinlabel $3$ [l] at  340 -50
\pinlabel $3$ [l] at  440 -50
\endlabellist
\centerline{\qquad\qquad\includegraphics[width=4.2in]{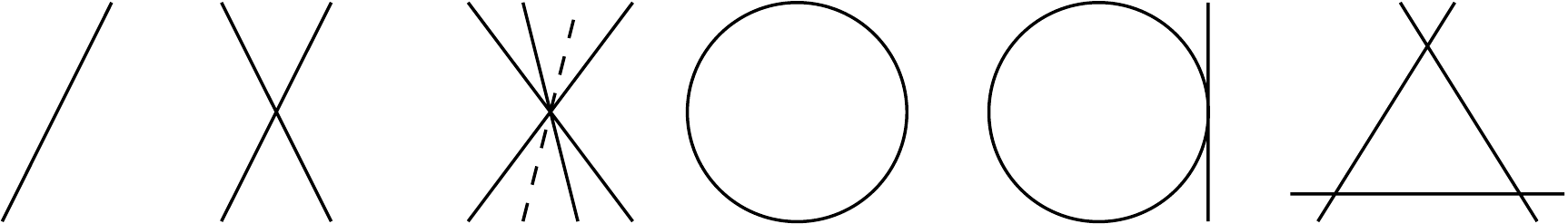}\vspace{1.2cm}}
\caption{\sf Six highly symmetric curves.\label{F-hsym}}
\end{figure}
\medskip

\begin{theo}\label{T-Big}
 There are only six connected Lie groups of dimension two or more
which can occur as the component of the identity  $\bG^{\,0}_\cC$ for some
curve in $\bP^2$. The
corresponding curves can be listed as follows. $($ {\rm Compare 
Figure~\ref{F-hsym}}.$)$
 \ssk

{\bf One Line.} If $\cC$ is a line, the 
stabilizer $\bG_\cC$ has dimension six.\footnote{This is
the unique example for which the action of the stabilizer
 $\bG_\cC$ on $|\cC|$ is not effective.
The group of automorphisms of the line (counted with any multiplicity)
is the 3-dimensional group $\PGL_2$, which is a quotient group of
the stabilizer $\bG_\cC$.}
Putting this line at infinity, $\bG_\cC=\bG^{\,0}_\cC$ can be identified with the
 group consisting of all non-singular affine transformations
\begin{equation}\label{E-aff-iso}
 (x,\,y)~\mapsto \big(\alpha x+\beta y +\sigma,\;\; \gamma x+\delta y+\tau\big)
\qquad{\rm with} \qquad \alpha\delta-\beta\gamma\ne 0~.\end{equation}
\smallskip

{\bf Two Lines.} If $\cC$ is the union of two distinct lines, the group 
$\bG_\cC$ is  four-dimensional, and $\bG^0_\cC$ can be identified
with the solvable  subgroup of $(\ref{E-aff-iso})$ consisting of
 transformations $(x,y)\mapsto (\alpha x+\beta y+\sigma,\;\;\delta y)$ 
 $($preserving the line $y=0$, as well as the line at infinity$)$.\ssk
\medskip

{\bf Concurrent Lines.} If $\cC$ is the union of three {\textit or more} lines
passing through a common point, the group $\bG_{\cC}$
is three-dimensional, and $\bG^0_\cC$can be identified with the subgroup of
$(\ref{E-aff-iso})$ consisting of transformations
$$(x,y)\mapsto (\alpha x+\beta y+\sigma,\;\; y)$$ which preserve every line 
$y={\it constant}$. $($This is the only case which includes curves of every 
degree $n\ge 3$. For $n\ge 4$ note that it includes infinitely many 
\hbox{$\bG$-equivalence} classes, since
any four lines through a point have a $\bG$-invariant cross-ratio.$)$
\msk

{\bf Conic Section.} If $\cC$ is a smooth degree two curve, the group $\bG_\cC$
 is a three-dimensional simple group, isomorphic to $\PGL_2$.\ssk
\medskip

{\bf Conic plus Tangent Line.} For the union of a smooth degree two
 curve with a tangent line,
the stabilizer has dimension two, isomorphic to the group of affine 
automorphisms $z\mapsto \alpha z+\beta$ of $\C$.\msk

{\bf Three Non-concurrent Lines.} For a triple of lines in general position, 
the group $\bG_\cC$ has dimension two, and $\bG^0_\cC$ can be 
identified with the abelian group consisting of non-singular diagonal 
transformations  $$(x,y)\mapsto(\alpha x,\;\beta y)~. $$
\end{theo}\msk

The proof will depend on the following catalog of one-dimensional stabilizers.

\msk

\begin{theo}\label{T-inf} 
A curve $\cC$ has infinite stabilizer if and only if, after a projective
change of coordinates, it is invariant under one of the following two
kinds of one-parameter subgroup of $\bG$:\ssk

{\bf(1) Diagonalizable of type $D(p,\,q,\,r)$:} Here the integers $$p\ge q\ge r
\ge 0$$ should be pairwise relatively prime with $p=q+r$.
The automorphism takes the form
\begin{equation}\label{E-spq}
 (x:y:z)~\mapsto ~(t^qx:t^py:z)~, \end{equation}
where $t$ varies
over all non-zero complex numbers. In this case, the invariant curve $\cC$
can be any union of finitely many irreducible
curves of the form $x=0$ or $y=0$ or $z=0$ or 
\begin{equation}\label{E-Dpq}
 x^p ~=~ a\,y^qz^r~,\qquad{\rm with}~~~ a\ne 0~.\end{equation}
\smallskip

{\bf(2) Non-Diagonalizable,\footnote{There is also a simpler non-diagonalizable
    family $(x:y:z)\mapsto(x+ty:y:z)$; but we will ignore this one
    since it occurs only as a subgroup of the  
3-dimensional group $\bG_\cC$ where $\cC$ is a union of concurrent lines.
This $\cC$ is included under type $D(1,\,1,\,0)$.}
of type ND,} with automorphism
$$ (x:y:z)~\mapsto~(x+ty+(t^2/2) z:y+tz: z)$$
where $t$ varies over all complex numbers. In this case $\cC$ can be any
union of curves of the form $z=0$ or
\begin{equation}\label{E-ND}
 x\,z~=~y^2/2 +a\,z^2~,\qquad{\rm with}\quad a\quad{\rm constant}~.\end{equation}
\end{theo}
\smallskip

\begin{rem}[{\bf Catalog of curves in $\fW_n$}]\label{R-infs}
Before proving Theorem~\ref{T-inf}, we will describe these curves in
 more detail. 
\smallskip

$\bullet$ {\bf Type $D(1,1,0)$.\;} The curves of type $D(1,1,0)$ are the easiest
to describe.
To be invariant under the action $(x,y,z)\mapsto (tx,\,ty,\,z)$ a curve
must be a union of lines $\quad(x:y)={\rm constant}\quad$  through the point
$(0:0:1)$, possibly together with the ``line at infinity'' $z=0$.
In other words, a curve $\cC$ of degree $n$ has type $D(1,1,0)$
if and only if it is a union of $n$ lines, at least $n-1$ of which pass through
a common point. To compute the dimension of the corresponding subset
of $\fW_n$, note that we need
two parameters in order to specify the intersection point, two parameters
to specify the free line, and then one-parameter for each additional line.
Hence the dimension of the corresponding irreducible subset
of $\fW_n$ is  $n+3$ (provided that $n\ge 3$). If $n\ge 4$ then this component
contains infinitely many different projective equivalence classes. In fact, for
$n>4$ there are $n-4$ invariant cross-ratios; while for $n=4$ the algebraic
 subset  consisting of lines through a common point has one invariant 
cross-ratio.
\msk

$\bullet$ {\bf Type $D(2,1,1)$.}\; By definition
each irreducible non-linear curve of type $D(2,1,1)$ can be put in the form
 $x^2=a\,y\,z$ with $a\ne 0$. Any two  curves in this form intersect
in the two points $(0:0:1)$ and $(0:1:0)$. For example, in the region $z\ne 0$
we can use affine coordinates with $z=1$. The curves are then
parabolas $x^2=a\,y$ which are tangent to each other at the origin. Thus any
automorphism which maps each curve to itself and fixes the origin must also map
the tangent line $y=0$ to itself. Similarly the tangent line $z=0$ 
at the point $(0:1:0)$ must map to itself,
and the line $x=0$ joining the two
intersection points must map to itself. A union of $k$ such curves,
 with $k\ge 2$, can be determined by $k+6$ independent parameters: namely
6 parameters to determine the three coordinate lines, 
and one more for each curve. Thus the corresponding irreducible
variety in $\fW_{2\,k}$ has dimension $6+k$. Note that we can obtain
varieties of higher degree, but the same dimension, by adjoining one or more
of the three coordinate lines to the curve.

It is interesting to note that a union of concentric circles
$u^2+v^2=\rho^2w^2$ looks superficially different, but is also of type 
$D(2,1,1)$. In fact it can be put in the required form  $x^2=a\,y\,z$
 by setting
$$x=w,\;\;\; y=u+iv,\;\;\; z=u-iv, \;\;\;\;{\rm and}\;\;\;\; a=1/\rho^2~.$$
\ssk

$\bullet$ {\bf Type $D(p,q,r)$ with $q\ge 2$.} \; For a curve of the form 
$$x^p~=~a\,y^qz^r ,\qquad{\rm with}\quad q\ge 2~,$$
the point $(x:y:z)=(0:0:1)$ is  a cusp-point of the form $x^p=a\,y^q$, 
using affine coordinates with $z=1$. On the other hand, using
affine coordinates with $y=1$, the point  $(0:1:0)$ is either a cusp point
of the form $x^p=az^r$ if $r>1$, or a flex-point of the form $x^p=az$ if
$r=1$. In either case these two points are distinguished. Hence, as in 
 Case~$D(2, 1, 1)$ it follows that one, two, or all three of the
 coordinate  lines $x=0,~y=0$ and $z=0$ can be adjoined to the curve without
 increasing the  dimension of  the associated irreducible components.
(Compare the last three curves on the top line of Figure \ref{F-W4}.) 
 As in Case~$D(2, 1, 1)$, this dimension is 
$k+6$ where $k$ is the number of non-linear components; but now we need only 
require that $k\ge 1$.\msk

$\bullet$ {\bf Non-Diagonalizable Type.} \;
The most transparent example in this case 
is the family of parallel parabolas $y=x^2+k$, 
each invariant under the automorphism
$$(x,\,y)~\mapsto(x+c\,,\;\;2\,c\,x+y+c^2)~.$$
Writing the defining equation in homogeneous form as $yz=x^2+k\,z^2$,
 with the line $z=0$ as a common tangent line,
note that two such parabolas intersect only at the point $(x:y:z)=(0:1:0)$.
More generally, it is not hard to check that
a union of $k\ge 2$  smooth curves of degree two can be put 
simultaneously  into the non-diagonalizable
normal form (\ref{E-ND}) if and only if these curves all are mutually tangent
at a common point of intersection and have no other intersection. (Thus the
pairwise intersection multiplicity at this point must be four.) Equivalently,
these curves must belong to the pencil consisting of all sums
$$\{\alpha\Phi_1+\beta\Phi_2\}$$  where
$\Phi_1=0$ defines a smooth degree two curve
and $\Phi_2=0$ is one of its tangent lines, counted with multiplicity two.
The corresponding irreducible component of $\fW_{2k}$
 has dimension $5+k$, assuming that $k\ge 2$. (It takes six parameters to
 specify a quadratic curve plus distinguished point, and one more for each
additional curve.)  We can also adjoin the common tangent line  without
increasing the dimension of the locus in the appropriate space $\fC_n$.
\bsk

\begin{figure}[h!]
\centerline{\includegraphics[width=4in]{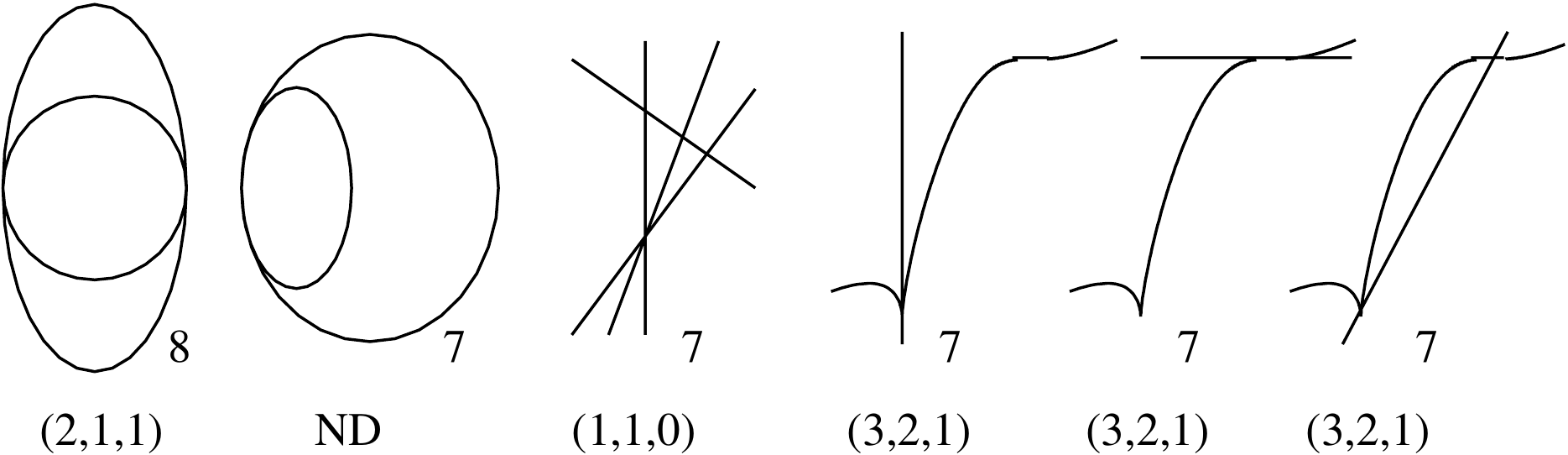}}\smallskip
\centerline{\includegraphics[width=2in]{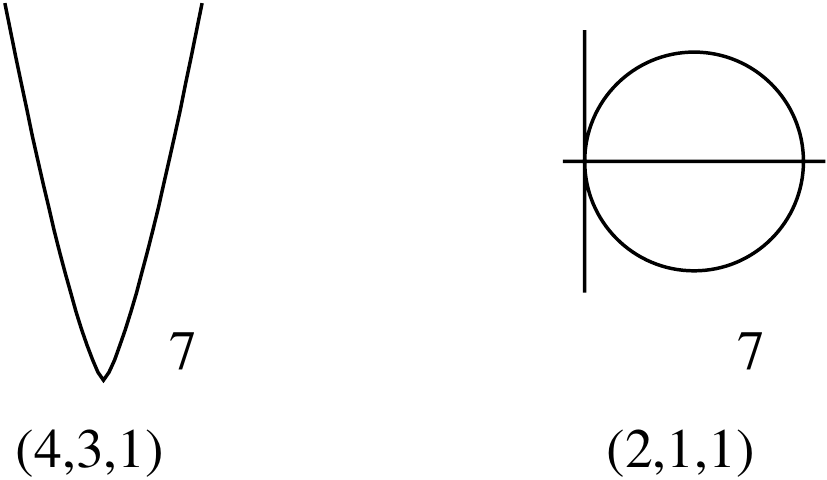}\qquad\qquad}
\caption{\label{F-W4}\sf The algebraic set $\fW_4\subset\widehat\fC_4$ 
consisting of curves or cycles of degree four 
with infinite stabilizer is the union of eight maximal
irreducible subvarieties. Representative generic curves from each of 
these subvarieties are shown.
In each case, the dimension of the algebraic subset is listed,
as well as the automorphism type indicated  
by the appropriate indices $p\ge q\ge r$ or by ND (for non-diagonalizable).}
\end{figure} 
\medskip

{\bf Examples.} 
For $n=3$, it is not hard to check that the algebraic set 
\hbox{$\fW_3\subset\C_3$}
is the union of two maximal irreducible subvarieties, both of dimension seven
and codimension two. One consists of the $\bG$-equivalence class of the 
cusp curve \hbox{$x^3=y^2z$}, together 
with the classes of $x^3=0$ and $y^2z=0$.  A generic curve in the other
is a smooth quadratic curve together with a line which intersects it 
transversally. There are five  subvarieties having the following as generic
elements: (1) a smooth degree two curve plus
tangent line, (2) three lines in general position, (3) three distinct
lines through a common point, (4) two lines, one with multiplicity two,
and (5) one line with multiplicity three.

For $n=4$, there are eight different maximal  irreducible subvarieties,
as illustrated in Figure \ref{F-W4}, and again there are many subvarieties.
(Each of these maximal subvarieties has a curve as generic element; but for 
higher degrees, a maximal irreducible subvariety may consist entirely of
cycles.) 
It follows from this figure that the dimension of
the algebraic set $\fW_4$ is $8$. 
\bigskip

{\bf Caution.} Of course not every curve in a maximal irreducible subvariety
is generic, so that there are many other curves in $\fW_4$ which are not shown.
As examples, in the $(2,1,1)$ example, the outer ellipse can expand and
converge to a union of two vertical
tangent lines. Similarly the inner ellipse can shrink and converge to a
horizontal line counted with multiplicity two.
\bigskip

For $n>4$ the dimension of $\fW_n$ is $n+3$,
with the irreducible component of type  $D(1,1,0)$ as one component
 of dimension $n+3$.
\end{rem}
\smallskip

\begin{proof}[Proof of Theorem \ref{T-inf}]
We know that every infinite stabilizer must contain a one-parameter Lie group.
Every one parameter subgroup of ${\rm PGL}_3(\C)$  can be parametrized as
$$ t~\mapsto ~\exp(tA)~=~I+tA+(tA)^2/2!+(tA)^3/3!+\cdots~,$$
where $A$ is a $3\times 3$ matrix. We can simplify this matrix in three
different ways:
\begin{itemize}
\item[$\bullet$] We can put $A$ into  Jordan normal form by a linear
change of coordinates.

\item[$\bullet$] We can add a constant multiple of the identity matrix to $A$,
or in other words multiply $\exp(tA)$ by a non-zero constant, since this will
not affect the image in $\PGL_3$.

\item[$\bullet$] We can multiply the matrix $A$ itself by a non-zero constant;
  this is just equivalent to multiplying the parameter $t$ by a constant.
\end{itemize}
\medskip

We will first show, using these three transformations, that the matrix $A$
can be reduced to one of the following, which we will refer to as Cases 1
through 4.

$$\left(\begin{matrix}a&0&0\\0&b&0\\0&0&c\end{matrix}\right),\qquad
\left(\begin{matrix}0&1&0\\0&0&1\\0&0&0\end{matrix}\right),\qquad
\left(\begin{matrix}0&1&0\\0&0&0\\0&0&0\end{matrix}\right),\quad{\rm and}\quad
\left(\begin{matrix}0&1&0\\0&0&0\\0&0&1\end{matrix}\right)~. $$
\smallskip

In fact if $A$ has three linearly  independent eigenvectors, then
we are certainly in Case 1. In particular, if 
the eigenvalues $a,\,b,\,c$ are all distinct, then we are
in Case 1. At the opposite extreme, if the eigenvalues are all equal
then subtracting a multiple of the identity matrix we may assume that they
are all zero. The Jordan normal form will then correspond to either Case 2 or 3.
(Evidently $A$
 cannot be the zero matrix.) Finally, suppose that just two of the eigenvalues
are equal. Then we can assume that two are zero and the third is one,
so that the Jordan normal form is either a diagonal matrix, so that we are in
Case 1, or corresponds to the matrix of Case 4.
The four corresponding matrices $\exp(tA)$  can now be listed as follows.

$$\left(\begin{matrix}e^{ta}&0&0\\0&e^{tb}&0\\0&0&e^{tc}\end{matrix}\right),
\qquad\left(\begin{matrix}1&t&t^2/2\\0&1&t\\0&0&1\end{matrix}\right),\qquad
\left(\begin{matrix}1&t&0\\0&1&0\\0&0&1\end{matrix}\right),\quad
{\rm and}\quad
\left(\begin{matrix}1&t&0\\0&1&0\\0&0&e^t\end{matrix}\right)~.  $$
\medskip

{\bf Case 1.} Suppose that a curve is invariant under the transformation
$$(x:y:z)\mapsto (e^{at}x:e^{bt}y:e^{ct}z)~.$$
Clearly this maps each of the three coordinate axes to itself. 
In affine coordinates with $z=1$, we can write this as
$$~~(x_0,\,y_0)~\mapsto~(x,y)~=~(e^{a't}x_0,~e^{b't}y_0)
\qquad{\rm where}\quad a'=a-c\,,\quad b'=b-c~.$$ 
In other words, if $x_0$ and $y_0$ are non-zero, we can write
\begin{equation}\label{E-x/x0}
 x/x_0\,=\,e^{a't}\,,\quad y/y_0\,=\,e^{b't}~.\end{equation}
Since $a,~b,~c$ cannot all be equal, we may assume (after
permuting the coordinates if necessary) that $a'$ and $b'$ are non-zero.

First suppose that the ratio $b'/a'$ is a rational number, which we can write
as a fraction in lowest terms as  $\pm \ell/m$ with $\ell,\,m>0$. 
Then \hbox{$m\,b'=\pm \ell\,a'$,} so that
$$ (y/y_0)^m~=~e^{mb't}~=~e^{\pm \ell a' t}~=~ (x/x_0)^{\pm\ell}~.$$
In other words, we have an equation of the form either 
$$ y^m~=~ax^\ell\qquad{\rm or}\qquad y^mx^\ell=a $$
for a suitable constant $a$. After permuting the coordinates appropriately,
this takes the required form (\ref{E-Dpq}).

On the other hand, if $b'/a'$ is irrational or imaginary, then the invariant
curve cannot be algebraic. Choosing $t$ so that $a't$ is an integral
multiple of $2\pi i$ in the equation (\ref{E-x/x0}), we see that $x=x_0$
but that $y$ takes a countably infinite  number of distinct values, 
which is impossible for any algebraic curve.
\medskip

{\bf Case 2.} Using affine coordinates $(x:y:1)$, the automorphism will
take the form
$$ (x_0,\,y_0)~\mapsto ~(x,y)~=~(x_0+ty_0+t^2/2, \;\;y_0+t)~.$$
Eliminating $t=y-y_0 $ from this equation, the curve through $(x_0,~y_0)$ takes
the form $x=y^2/2 +(x_0-y_0^{\,2}/2)$,
which agrees with the required normal form (\ref{E-ND}) 
except for the factor of $1/2$
which is easily eliminated by a change of variables.\medskip

{\bf Case 3.} In this case the transformation takes the simpler form
$$(x_0:y_0:z_0)~\mapsto (x_0+ty_0: y_0:z_0)~, $$
so that the invariant curves are just parallel lines $y=y_0$, $z=z_0$,
 or in other words
lines which pass through the point $(1:0:0)$ on the line at infinity.
These can easily be put in the Case 1 normal form, of type $D(1,1,0)$.
\medskip

{\bf Case 4.} Here the transformation takes the more ominous form 
$$ (x_0,\,y_0,\,z_0)~\mapsto~ (x,y,z)~=~(x_0+ty_0,\;y_0,\; e^tz_0)~.$$
Thus $y=y_0$ is constant. If $y_0=0$, then the invariant lines are
just parallel lines with $x=x_0$, again of type $D(1,1,0)$. But if $y_0\ne
0$, then we can solve for $t=(x-x_0)/y_0$, so that
the invariant curves have the form
$$ z~=z_0\,\exp\big((x-x_0)/y_0\big)~.$$
Since this cannot be the equation of any algebraic curve, this completes the
proof of Theorem \ref{T-inf}
\end{proof}
\smallskip

\begin{proof}[Proof of Theorem \ref{T-Big}]
  We must show that every curve with stabilizer of\break
  dimension two or more is
  contained in the list which is illustrated in Figure \ref{F-hsym}.

Let $\cC=\bigcup_j\cC_j$ be a curve with irreducible components $\cC_j$.
Then the intersection $\bigcap_j \bG_{\cC_j}$ is a subgroup of finite index
in $\bG_\cC$. 
If $\dim\,\bG_\cC\ge 2$, then  it follows that every irreducible component
must have $\dim\, \bG_{\cC_j}\ge 2$. Therefore every irreducible component
must have degree one or two. In fact, any irreducible curve of degree three or
more is either a cusp curve, with $\dim\,\bG_\cC=1$, or else has finite
stabilizer.

Thus we need only consider unions of lines and smooth quadratic curves.
Similarly, we can ignore curves of type $D(p,q,r)$ with $p\ge 2$, since they
either contain a cusp curve, or consist of at most three lines. (Note that any
union of at most three lines is already included in the list represented by
Figure \ref{F-hsym}.) We can also ignore curves of type ND with two or more
components, since it is easy to check that they have a one-dimensional
stabilizer. Thus we only need to consider curves of type $D(2,1,1)$ or
$D(1,1,0)$.

It is easy to check that a curve of type $D(1,1,0)$ has stabilizer of
dimension two or more if and only if it either consists of concurrent lines,
or consists of exactly three lines.
Thus we are left with curves of type $D(2,1,1)$. Note that the group of
automorphisms of a quadratic curve which fix two points is one-dimensional.
For example, any automorphism of $\bP^2$ which fixes the two points $(1:0:0)$ and
$(0;1:0)$ must take the form
$$(x,y)~~\mapsto~~(ax+b,~~cy+d)$$
in affine coordinates. Such an automorphism maps the hyperbola $xy=1$ to
itself only  if $ac=1$ and $b=d=0$, yielding a one dimensional group.
Thus, if a curve $\cC$ with a quadratic component $\cC_1$ has
$\dim\,\bG_\cC\ge 2$, then there must be at most one singular (or intersection)
point on $\cC_1$. There can be one tangent line intersecting $\cC_1$ in one
point; but nothing more. (Another quadratic curve intersecting at one point
would yield a curve of type ND, which has already been excluded.
\end{proof}
\bigskip

\subsection*{\bf Automorphism Groups of Smooth Curves.}
This subsection will first answer 
 the following question. 

\begin{quote} \it For which
degrees $n$ and which primes $p$ does there exist a smooth curve of degree $n$
which admits a projective
 automorphism of period $p$?\end{quote}
\noindent(It also contains a discussion of the
corresponding question for conformal automorphisms of more general
Riemann surfaces.) Using this result, we show that a generic curve of degree 
$\ge 4$ has trivial stabilizer. Finally, we show that every finite 
subgroup of $\PGL_3(\C)$ can occur as the stabilizer for some smooth curve.
\smallskip

\begin{theo}\label{T-aut1}
Given a degree $n$ and a prime $p$ there exists a smooth curve of degree
$n$ in $\bP^2(\C)$ with a projective automorphism of period $p$ if and only
if $n$ is congruent to either $0,~1,$ or $2$
modulo $p$. If   $n\ge 3$, an equivalent condition is that $p$ must be a
divisor of either  $n,\,n-1$ or $n-2$. 
\end{theo}
\smallskip

As examples, smooth curves of degree $n\le 2$ have automorphisms of all prime
orders.  
For an arbitrary degree $n$,
the primes 2 and 3 can occur; and for $n=3$ or $4$,
these are the only possible primes. For any odd prime $p$, the smallest
$n\ge 3$ for which $\fC^\sm_n$ contains a curve with a period $p$ 
orbit is $n=p$ \ssk

For a much more detailed study of finite stabilizers, see  \cite{Harui}.
\smallskip

\begin{proof}[Proof of Theorem $\ref{T-aut1}$]
For each of the three cases, an appropriate curve $\Phi(x,y,z)=0$, and a
 corresponding automorphism of period $p$, can be listed as follows, where 
$\alpha$ is a primitive $p$-th root of unity, and $\alpha\beta=1$. 
\begin{eqnarray*}
 n\equiv 0~({\rm mod}~p)\,,&\quad \Phi=x^n+y^n+z^n\,,\quad&
(x:y:z)\mapsto(\alpha\,x:y:z)~;\\
 n\equiv 1~({\rm mod}~p) \,,&\quad \Phi=x^{n-1}y+y^n+z^n\,,\quad&
(x:y:z)\mapsto(\alpha\,x:y:z)~;\\
 n\equiv 2~({\rm mod}~p)\,,&\quad \Phi=x^{n-1}y+ x\,y^{n-1}+z^n\,,\quad&
(x:y:z)\mapsto(\alpha\,x:\beta\,y:z)~.
\end{eqnarray*}Each of these three curves is smooth, since in each case
it is not difficult to check that the only solution to
the equations $\Phi_x=\Phi_y=\Phi_z=0$ is $x=y=z=0$. Furthermore, 
 the indicated mappings are clearly period $p$ automorphisms of $\bP^2(\C)$, 
and it is not 
difficult to check that each one maps the corresponding
 locus $\Phi=0$ to itself. (For example in the last case, 
since $n-1\equiv 1 ~({\rm mod}~p)$, the monomials $x^{n-1}y$ 
and $xy^{n-1}$ are both multiplied by $\alpha\beta=1$.)

Conversely, let $\cC\subset \bP^2(\C)$ be a smooth curve of arbitrary
degree $n\ge 3$ which has a projective automorphism of prime order $p$.
The corresponding linear automorphism of $\C^3$ necessarily\footnote
{In fact, using the Jordan normal form, one sees easily that an automorphism
which does not have three independent eigenvectors, 
can never have finite order.}
has three linearly
independent eigenvectors,  which we can place so that the automorphism has
the form $(x,y,z)\mapsto (\alpha x, \,\beta y, \gamma z)$. Here the three
eigenvalues cannot all be equal, since our map of projective space is not the 
identity. Therefore, permuting the coordinates if necessary, we may assume
that $\gamma\ne\alpha$ and $\gamma\ne\beta$. Hence after dividing by
a common constant, we may assume that $\gamma=1$, and that both $\alpha$ and
$\beta$ are primitive $p$-th roots of unity.

The defining equation for any curve which is invariant under this 
transformation must be a linear
combination of monomials of the form $x^iy^jz^k$ with $i+j+k=n$ and with
$\alpha^i\beta^j=1$. Since $\cC$ is smooth, there
must be at least one such monomial with $i>n-2$ (or in other words of the
form $x^n$ or $x^{n-1}y$ or $x^{n-1}z$). For otherwise, it is not hard to check
that $(1:0:0)$ would be a singular point. Similarly there must be at least one
with $j>n-2$ and at least one with $k>n-2$.

If one of these monomials is $x^n$, then $\alpha^n=1$ hence $n\equiv 0~
({\rm mod}~p)$, and the same conclusion follows if one of the monomials is
$y^n$. Similarly, if one of the monomials is $x^{n-1}z$ or $y^{n-1}z$,
then $n\equiv 1~({\rm mod}~p)$. The only other possibility is that the two
monomials  $x^{n-1}y$ and $y^{n-1}x$ both occur, so that
$$ \alpha^{n-1}\beta~=~\beta^{n-1}\alpha~=~1~.$$
Dividing by $\alpha\beta$, it follows that $\alpha^{n-2}=\beta^{n-2}$.
There are then two possibilities:
Either $\alpha^{n-2}=\beta^{n-2}=1$ hence $n\equiv 2~~{\rm mod}~p)$,
or else $\alpha=\beta$ hence $n\equiv 0~({\rm mod}~p)$. This completes 
the proof.\end{proof}
\medskip

The analogous question for conformal automorphisms of arbitrary
Riemann surfaces has an explicit but more complicated answer:
\smallskip

\begin{theo}\label{T-RSaut} Given a prime $p$ and an integer $\fg\ge 2$, 
there exists a closed Riemann surface $\cS$ of genus $\fg$ with a
conformal automorphism of period $p$ if and only if, for some
\hbox{$0\le \fg'< \fg$,} the ratio 
\begin{equation}\label{E-confaut}
k~=~\frac{2\fg-2\,-\,(2\fg'-2)p}{p-1}~.
\end{equation}
is an integer, with $k\ge 0$ and $k\ne 1$. It follows from this condition
 that $p\le 2\fg+1$.
\end{theo}
\smallskip

\begin{proof}[Proof Outline]
Let $\bG$ be a group of automorphisms of $\cS$ of order $p$ with $k$
fixed points, and let $\cS'=\cS/\bG$ be the quotient surface, with genus
$\fg'$. Then the Riemann-Hurwitz formula can be written as 
\begin{equation}\label{E-RH}
 2\fg-2~=~(2\fg'-2)p\,+\, (p-1)k~.\end{equation}
 (See for example \cite[Theorem 7.2, Pg. 70]{Mi2}.) 
Solving for $k$, we obtain the formula~(\ref{E-confaut}).

For the converse construction, choose a surface $\cS'$ of genus $\fg'$,
and choose a finite subset $K\subset\cS'$ consisting of $k$ points.
A $p$-fold cyclic covering of $\cS'\ssm K$ is determined by a homomorphism
from the fundamental group $\pi_1(\cS'\ssm K)$, or equivalently from
the abelianized fundamental group $H_1(\cS'\ssm K)$, onto the cyclic group
of order $p$. Such a homomorphism always  exists if $\fg'>0$, but in the
 case $\fg'=0$ it exists only if $k\ge 2$. However, for this cyclic covering to
 extend to a branched covering, branched over each point of $K$, we need the
 extra condition that a small loop around each point 
of $K$ maps to a generator
of the cyclic group. This condition is easily satisfied if $k\ge 2$.
However, it can never be satisfied when $k=1$ since a small loop around a
single puncture point represents the zero element of $H_1(\cS'\ssm K)$.

We can also solve the equation (\ref{E-RH}) for
$$ p~=~\frac{2\fg-2+k}{2\fg'-2+k}~,$$
where $k$ must be large enough so that the denominator is positive.
This ratio is monotone decreasing as a function of $k$, so for fixed
$\fg>\fg'$ it takes the largest value when the denominator $2\fg'-2+k$ is $+1$, 
so that $p=2(\fg-\fg')+1\le 2\fg+1$. (Of course
 the largest {\it prime} solution will often be smaller than this.)
This proves Theorem~\ref{T-RSaut}.
\end{proof}\medskip

As an example, 
for a conformal automorphism of a Riemann surface of genus $\fg=3$,
we see from Theorem \ref{T-RSaut} that the possible 
 primes are  $2,\,3$ and \hbox{$7~(=~2\fg+1)$.}
On the other hand, for a  projective automorphism of a 
smooth curve of genus 3 (and hence degree~4)
in $\bP^2(\C)$, by Theorem \ref{T-aut1}
the only possible primes are $2$ and $3$.
Since we know by Proposition  \ref{P-hyper} 
that every conformal automorphism of a curve of degree 4 
is actually projective, it follows that a genus 3 curve with a period 7
automorphism cannot be embedded in $\bP^2(\C)$, and hence (again by
Proposition~\ref{P-hyper}) must be hyperelliptic.

\smallskip

\begin{theo}\label{T-noaut} For $n\ge 4$, a generic real or complex
curve of degree $n$ in $\bP^2$
has no projective\footnote{It seems likely that it also has no 
conformal automorphisms; but we don't know how to settle this question.}
 automorphisms other than the identity map.
\end{theo}
\smallskip

  On the other hand, for degree $n=3$
  the projective automorphism group of a generic curve has order six 
  in the real case, and eighteen in the complex case.  (See \cite[\S3]{BM}.)
  In terms of the additive group structure on an elliptic curve, taking a flex
  point as zero element, the automorphisms have the form 
  $\p\mapsto \pm \p +\p_0$, where $\p_0$ can be any one of the flex points
(three in the real case or nine in the complex case.)
\smallskip

\begin{rem}[{\bf Conformal Automporphisms of Riemann Surfaces}]\label{R-RSaut1}
The corresponding statement for arbitrary Riemann surfaces 
is that a generic Riemann surface of genus $\fg\ge 3$
has no non-trivial conformal automorphism. (See \cite{Ba} 
as well as \cite{Po}.) However, 
every Riemann surface of genus two is hyperelliptic, and hence has
an automorphism of period two.
The group of all conformal automorphisms of a Riemann surface of genus
$\fg\ge 2$ has been much studied 
since the time of Hurwitz \cite{Hur}, who proved that such a group 
has at most $84\,(\fg-1)$ elements. In particular, there are now many
examples of groups which realize this Hurwitz maximum. 
As an extreme example, Wilson \cite{Wi} has shown that the ``monster group'' 
of order roughly $8\times 10^{53}$ is one such group.
\end{rem}
\medskip

\begin{proof}[Proof of Theorem \ref{T-noaut}.] 
Clearly it suffices to consider the complex case. Since a generic curve
is smooth, it will suffice to work in the moduli space\break
$\M^{\sm}_n=\fC^{\sm}_n/\bG~~$ for smooth curves.
Furthermore, since all such curves have finite stabilizer, we need only
consider possible finite groups.

For each prime $p$, let
$$ \M^\sm_n(p)~\subset~\M^\sm_n$$
be the subset consisting of all smooth curve-classes which have
 an automorphism of period $p$. 
The proof will simply require counting dimensions. The dimension of the
 moduli space $\M^\sm_n$ can be computed as 
$${\rm dim}(\M^\sm_n)~=~\left(n+2\atop 2\right)-9~=~(n^2+3n-16)/2~\qquad
{\rm for}\quad n\ge 3~.$$
We will show that the subspace $\M^\sm_n(p)$ 
has strictly smaller dimension, provided that $n\ge 4$.
(Of course this subspace may be empty, in which case we assign it the
dimension $-1$. For  example by Theorem \ref{T-aut1}, this is the case
for all primes $p> n$.)
\ssk

Any projective automorphism of $\bP^2(\C)$ lifts to
a linear automorphism of $\C^3$, with three eigenvalues. However, we can
multiply these eigenvalues by any common non-zero constant, so only their
ratios have an invariant meaning. For an automorphism of finite
order, the linear transformation is necessarily diagonalizable, as noted
in the proof of Theorem \ref{T-aut1}.
 Thus we can choose coordinates so that the automorphism is given by
 $$(x:y:z)~ \mapsto~(\alpha x:\beta y:\gamma z)~,$$
 where $\alpha,\,\beta,\,\gamma$ are roots of unity.
There are now two possibilities. Either only two of these three eigenvalues
are distinct, or all three are distinct. The corresponding subsets of
$\M^\sm_n$ will be denoted by  $\M'_n(p)$ and $\M''_n(p)$ respectively.
\smallskip

{\bf Case 1.} Suppose that only two of the eigenvalues are distinct.
(In the special case $p=2$,  this condition is always satisfied.)
After permuting coordinates and multiplying by a constant, we may assume that
$\alpha=\beta=1$ and that $\gamma$ is a primitive $p$-th
root of unity, so that $$(x:y:z)\mapsto(x:y:\gamma\, z)~.$$
If $\Phi=0$ is the defining equation for the invariant curve, then $\Phi$
must be a linear combination of monomials $x^iy^jz^k$
with $i+j+k=n$ and with $\gamma^k$ equal to some constant,
or in other words with $k$ congruent to some fixed $k_0$
modulo $p$. For each choice of $k$, there are $~n+1-k~$ possible
choices for $i$ and $j$. Thus the number of such monomials
is equal to the sum
$$ s(k_0,\,p)~=~\sum_{0\le k\le n\,;~~~ k\equiv k_0~~({\rm mod}~p)} n+1-k ~.$$
For each $p$, it is not hard to check that this sum $s(k_0,\,p)$ will take its
 largest value if we choose $k_0$ to be zero. That is:
\begin{equation}\label{E-s0}
  s(k_0,\,p)~\le~s(0,\,p)~=~\sum_{0\le m\le n/p} (n+1-mp)\end{equation}
for every $k_0$ mod $p$. Similarly, it is even easier to check that
\begin{equation}\label{E-s0p}
    s(0,\,p)~\le ~s(0,\,2)\qquad{\rm for~ every}~~~p~.\end{equation}
Thus it will suffice to concentrate on the case $p=2$.\ssk

For $p=2$, the number $s(0, \,2)$
of such monomials is
$$ 1+3+5+\cdots+(n+1) 
\quad {\rm for} \quad
n\quad{\rm ~even,~~and}$$
$$ 2+4+6+\cdots+(n+1) 
\qquad {\rm for} \quad n\quad{\rm ~odd~.}$$
A brief computation shows that
$$~s(0,2)~=~\floor\left((n+2)^2/4\right)~$$
in both cases;
where $\floor(\xi)$ denotes the largest integer $\le \,\xi$.\ssk
  
In order to find the corresponding dimension in moduli space, we must first
subtract one, since all of the coefficients of $\Phi$ may be multiplied by
a constant. Then we subtract another four since the general linear group
${\rm GL}_2(\C)$ acts by the transformation
$$(x,\,y,\,z)~\mapsto~(ax+by,\;cx+dy,\;z) ~,$$
mapping each eigenspace to itself. Thus the set of all curve-classes of
degree $n$ with an automorphism of period 2 has dimension equal to
$~\floor\big((n+2)^2/4-5\big)~$ for $n\ge 3$.
Here is a table. 
\smallskip

\begin{center}
\begin{tabular}{rccccc}
$ n=$& 3&4&5&6&7\\
${\rm dim}(\M^\sm_n)=$ & 1&6&12&19&27\\
${\rm dim}\big(\M'_n(2)\big)=$& 1& 4 &7& 11&15\\
\end{tabular}
\end{center}
In general, as we pass from $n$ to $n+1$ the dimension of $\M^\sm_n$ increases
by $n+2$, while the dimension of $\M'_n(2)$ increases by 
$$ \dim\big(\M^\sm_{n+1}(2))- \dim\big(\M'_{n}(2))
~=~\floor((n+1)/2)+1~<~ n+2~; $$
It follows easily that 
$$~\dim(\M_n)>\dim\big(\M_n(2)\big)\quad{\rm for~ all}\quad n\ge 4~.$$
The analogous inequality for odd primes follows easily from the inequality
(\ref{E-s0p}). In fact $~\dim\big(\M'_n(p)\big)<\dim\big(\M'_n(2)\big)~$
for $p>2$.
\medskip

{\bf Case 2.} 
Now suppose there are three distinct eigenvalues,
so that the transformation can be put the form
$$(x:y:z)~\mapsto (x:\beta y: \gamma z)$$
where $\beta$ and $\gamma$ are distinct primitive $p$-th roots of unity.
(This case can only occur if $p\ge 3$.)
Thus we can set $\gamma=\beta^m$ for some $1<m<p$, so that the
transformation will multiply each monomial $x^iy^jz^k$ by
$\beta^{j+km}$. We must now estimate the number of monomials for which
$~j+km~$ is congruent to some constant $j_0$ modulo $p$. The estimate
 will be based on the following remark.

\begin{quote} \it Given a sequence of $\ell+1$ consecutive integers, the number
of these integers which are congruent to some constant modulo $p$ is at most
$~~~\floor(\ell/p)+1$.
\end{quote}

\noindent
For example, if $\ell+1\le p$ (so that $\floor(\ell/p)=0$)
there is at most one solution; but if $\ell+1> p$ there
 may be two solutions. The proof will be left to the reader.\smallskip

It will be convenient to set $\ell=n-k$, with $0\le j\le \ell \le n$.
It follows that the number of monomials satisfying the required conditions
that $i+j+k=n$ and $j+km\equiv j_0~~({\rm mod}~p)$
is at most
$$\sum_{\ell=0}^n\floor\big(1+\ell/p\big)~=
~(n+1)+\sum_{\ell=0}^n\floor\big(\ell/p\big)~.$$
In order to find the dimension of the corresponding subset $\M''_n(p)$
of moduli space, we must subtract two from this sum, since we can multiply
$x$ by an arbitrary non-zero constant and also multiply $(y,z)$ by an
arbitrary non-zero constant without changing the projective equivalence class.
Therefore
$$\dim\big(\M''_n(p)\big)~\le 
~(n-1)+\sum_{\ell=0}^n\floor\big(\ell/p\big)~.$$
(This is a rather crude upper bound, but will suffice for the proof.)
Clearly the expression on the right is monotone decreasing as a function
of $p$, so it suffices to consider the case $p=3$.
Here is a table. \smallskip

\begin{center}
\begin{tabular}{rcccc}
$ n~=$& 4&5&6&7\\
${\rm dim}(\M^\sm_n)~=$ & 6&12&19&27\\
${\rm dim}\big(\M''_n(3)\big)\le$ &  5 &7& 10&13\\
\end{tabular}
\end{center}
Since the difference between the  dimension bounds for $\M''_{n-1}(3)$ and 
$\M''_{n}(3)$ is\break $~1+\floor(n/3)<n+1$,  it follows easily that
$~\dim\big(\M''_n(3)\big)<\dim(\M^\sm_n)$ for all $n\ge 4$.
This completes the proof of Theorem \ref{T-noaut}.
\end{proof}
\medskip

\begin{prop} \label{P-allstab}
For any subgroup 
$\Gamma\subset \bG=\PGL_3(\C)$ with finite order $m$, 
 there exists a smooth curve $\cC\in\fC_{4m}(\C)$ 
with stabilizer  $\bG_\cC$ equal to $\Gamma$.
\end{prop} 
\smallskip

\begin{rem}\label{R-allsub}
A catalog of all possible finite subgroups of $\bG=\PGL_3(\C)$
 has been provided by
Miller, Blichfeld, and Dickson \cite[Part II]{MBD}. (See also Hambleton and Lee 
\cite{HL}.) Without giving the complete list, here are some examples:
The group can have arbitrary order, since any abelian group with two generators
is contained in the stabilizer for three lines in general position, as
described at the beginning of this section.
Any finite subgroup of  the rotation group ${\rm SO}_3$ can occur, since
\hbox{${\rm SO}_3\subset \PGL_2\subset\PGL_3$}.  
This includes the icosahedral group, which is isomorphic to the 
alternating group ${\mathfrak A}_5$. 
Two other simple groups also occur: namely ${\mathfrak A}_6$
of order 360, and ${\rm PSL}_2(\F_7)$ of order 168. One other noteworthy
 example is the automorphism group of the Hesse configuration, which has order
216. This can be realized as a stabilizer $G_\cC$ where $\cC$ is a curve 
consisting of twelve lines, which intersect in the nine flex points of an
elliptic curve.
\end{rem}
\smallskip

\begin{proof}[Proof of Proposition \ref{P-allstab}] 
Let $\Gamma$ be a finite subgroup of $\PGL_3(\C)$ with $m>1$ elements. It
is easy to construct a singular curve in $\fC_{4m}$ which has
$\Gamma$ as stabilizer: According to Theorem~\ref{T-noaut},  a generic curve
$\cC_1\in\fC_4$ has trivial stabilizer. Let $\cC_1^{\,\Gamma}\in \fC_{4m}$
be the union of the translates $\g(\cC_1)$ by the elements $\g\in\Gamma$.
Then it is not hard to see that the stabilizer of $\cC_1^{\,\Gamma}$
is precisely the group $\Gamma$.
\ssk

In order to find a smooth example, we will use Bertini's Theorem,\footnote{We
  thank Robert Lazarsfeld for suggesting this argument.} which asserts that a
locus of the form
\begin{equation}\label{E*}
  \alpha_1\Phi_1+\cdots+\alpha_k\Phi_k=0
\end{equation}
(where the $\Phi_j$ are homogeneous polynomials of the same degree) is
non-singular for a generic choice of the coefficients $\alpha_j$, provided
that the common zero locus
$$\Phi_1=\cdots=\Phi_k=0$$
is empty. (See for example \cite{Harr} or \cite{Nam}.) To apply this Theorem,
choose three curves $\cC_i\in\fC_4$ which are generic in the sense that
the triple $(\cC_1,\,\cC_2,\,\cC_3)$  is a generic point of
$\fC_4\times\fC_4\times\fC_4$. Then each pair $\cC^{\,\Gamma}_i$ and
$\cC^{\,\Gamma}_j$ will intersect transversally in $(4m)^2$ distinct points,
but the $3$-fold intersection
$\cC_1^{\,\Gamma}\cap\cC_2^{\,\Gamma}\cap\cC_3^{\,\Gamma}$
will be empty. Now let $\Phi_j=0$ be the equation of $\cC^{\,\Gamma}_j$. Then
for a generic choice of coefficients $\alpha_j$ the locus~(\ref{E*}) will be
a smooth $\Gamma$-invariant curve. If we assume more explicitly that
$(\alpha_1,\alpha_2,\alpha_3,\cC_1,\cC_2,\cC_3)$ is a generic point of
$\C^3\times\fC_4\times\fC_4\times\fC_4$, then it follows
easily  that this curve will have stabilizer precisely equal to $\Gamma$.
\end{proof}

\bigskip 

\begin{figure}[ht!]
\centerline{\includegraphics[width=1.5in]{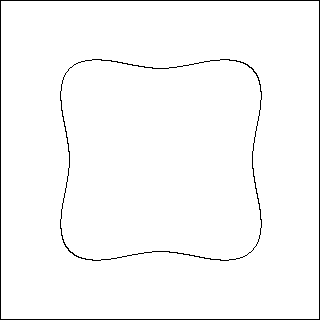}\quad
\includegraphics[width=1.5in]{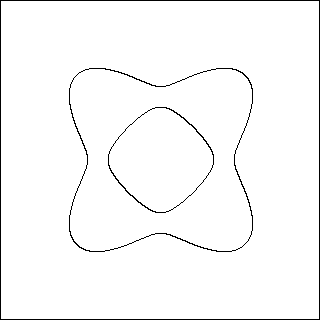}\quad
\includegraphics[width=1.5in]{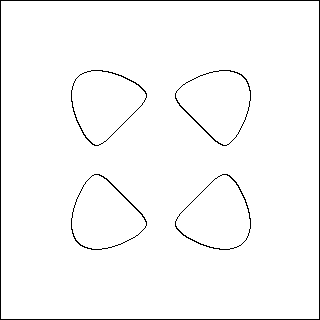}}
\centerline{\includegraphics[width=1.5in]{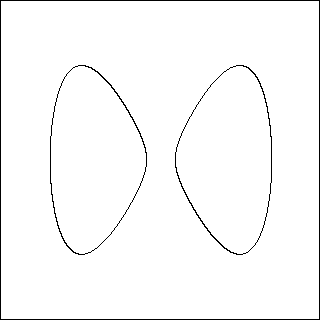}\quad
\includegraphics[width=1.5in]{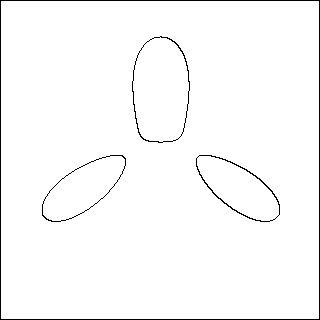}\quad
\includegraphics[width=1.5in]{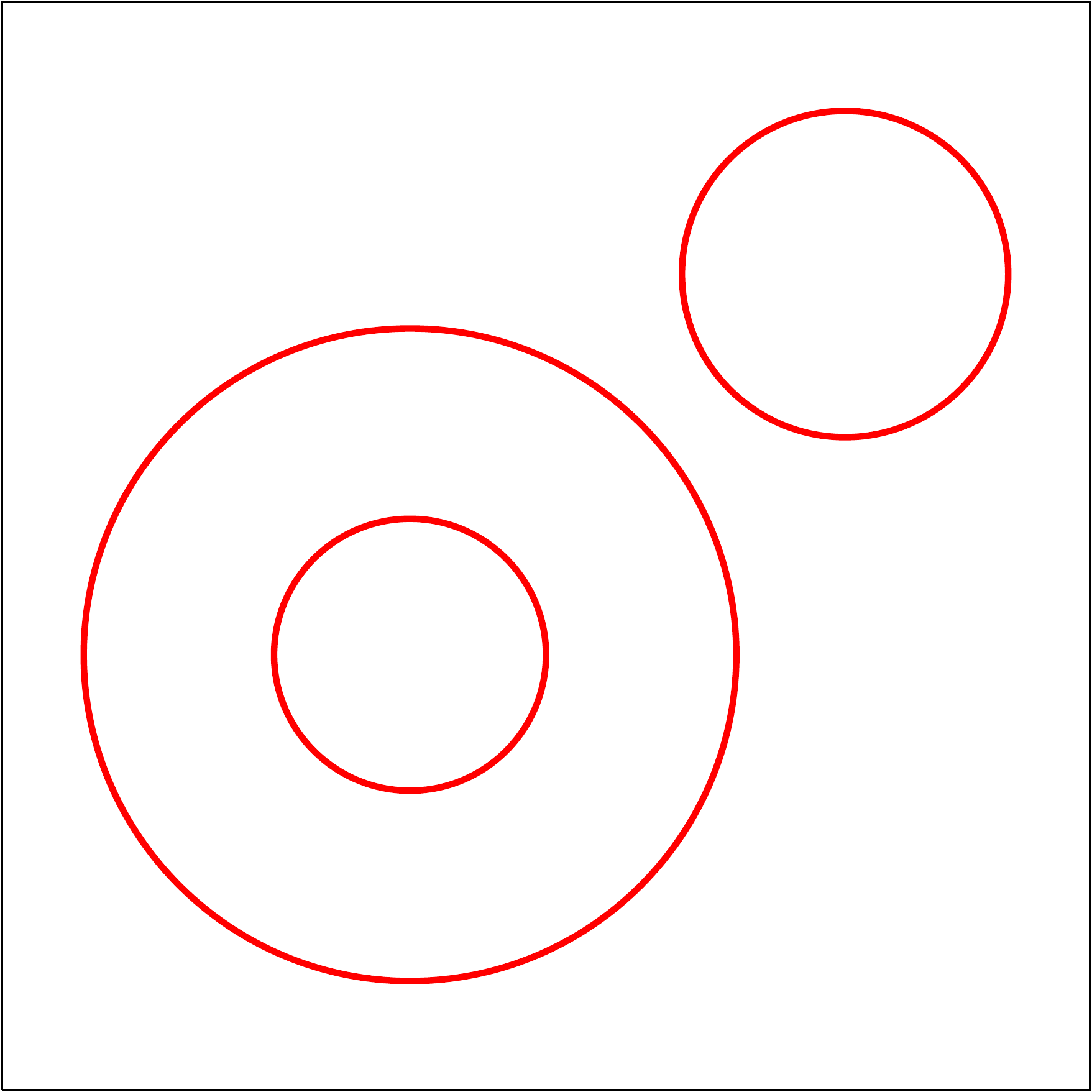}}

\caption{\label{F-deg4}
\sf Six examples. The first five panels above show representative curves for
the five connected components \protect\footnotemark[41]
of $\fC^\rs_4(\R)$ for which the real locus $|\cC|_\R$ is non-empty. 
Note that the various 
components of $|\cC|_\R$ always arrange themselves so that no line intersects
more  than two of them. The sixth panel shows a
configuration of three components which cannot occur for any curve of degree
less than six, since a line through the two smaller circles
crosses all three circles, and hence has six
intersection points.}
\end{figure}
\msk

\setcounter{lem}{0}
\section{Real Curves: The Harnack-Hilbert Problem.}
\label{s-genR}

This section will be a digression, discussing a different kind of problem.
Harnack \cite{Harn} in 1876 proved that:

\footnotetext[41]{The classification depends on Georg Zeuthen's proof
  that  smooth curves of degree four are isotopic through a smooth
one-parameter family of projective curves, and hence belong to the same 
connected component of $\fC^\rs_4$, if and only if they are 
  topologically isotopic. (See \cite[p.~197]{Vi3}.)}

\setcounter{footnote}{41}

\begin{quote} \sf The number of connected components of a smooth
curve of degree $n$
in the real projective plane is at most ${n-1\choose 2}+1$.
\end{quote}

\noindent
As examples, for degree three the curve $|\cC|_\R$ has most two components,
and for degree four at most four. (Compare Figure~\ref{F-deg4}.)

The most famous question about such curves is Hilbert's $16$-th Problem
\cite{Hi}:

\begin{quote} \sf
``The maximum number of closed and separate branches which a plane algebraic 
curve of the $n$th order can have has been determined by Harnack. There arises
the further question as to the relative position of the branches in the 
plane. $\ldots$''
\end{quote}
For modern expositions, as well as much further information, see \cite{Vi1},
\cite{Vi2}, \cite{Vi3}. 

\subsection*{\bf Real-Smooth and Complex-Smooth Curves}\ssk

A curve $\cC$ defined over $\R$ will be called \textbf{\textit{real-smooth}} if
 there are  no singularities in the real zero-locus $|\cC|_\R$, and 
\textbf{\textit{complex-smooth}} 
if there are no singularities in the complex zero locus $|\cC|_\C$.
Thus there are open subsets 
$$ \fC^\cs_n~\subset~\fC^\rs_n~\subset~\fC_n(\R) $$
consisting of complex-smooth and real-smooth curves.
It is somewhat easier to construct examples if we work in the larger
space $\fC^\rs_n$.  For example, any
union of two or more disjoint circles or ellipses in $\bP^2(\R)$ is real-smooth
but not complex-smooth.  (Of course the number of components in such
trivial examples is very much smaller than Harnack's upper bound.)

On the other hand, in the complex-smooth case we can obtain
 extra information by considering the way that the real locus $|\cC|_\R$
is embedded in the Riemann surface $|\cC|_\C$. However,
for Hilbert's problem it doesn't matter whether we work with real-smooth
or complex-smooth curves: 
\smallskip

\begin{prop}\label{p-ass} \it Every real-smooth curve can be approximated
arbitrarily closely by a complex-smooth curve. Furthermore,
in the space ${\mathfrak C}_n^{\sf rs}$ 
of real-smooth curves of degree \hbox{$n\ge 3$,}
 the subvariety consisting of curves $\cC$ such that the complex zero-set 
$|\cC|_\C$ is singular has codimension two. Therefore no
 connected component in  ${\mathfrak C}_n^{\sf rs}$  is disconnected
by this subvariety. In other words, every connected component in
  ${\mathfrak C}_n^{\sf rs}$ determines a unique connected component in the 
smaller set ${\mathfrak C}^{\sf cs}_n$. 
\end{prop}\msk

\begin{proof} {\bf Step 1.} 
The space of all curves of degree $n$ in ${\mathbb P}^2(\C)$ has complex
dimension $d(n)={n+2\choose 2}-1=n(n+3)/2$. Let $V_n$ be the 
subvariety consisting of
curves having singular points at $(0:0:1)$ and $(0:1:0)$. Then the dimension
of $V_n$ is $d(n)-6$. In fact the curve defined by the equation\footnote{Here
$\Phi$ should have no squared factor, so that this equation defines a curve
rather than a 1-cycle.} 
$$ \Phi(x,y,z)~=~\sum_{i+j+k=n}  a_{i,j,k}\, x^iy^jz^k~=~0 $$ 
will pass through these two points only
if $a_{0,0,n}=a_{0,n,0}=0$, and will be singular at these two points only if
$$ a_{1,0,n-1}=a_{0,1,n-1}= a_{1,n-1,0}=a_{0,n-1,1}=0~.$$\ssk

{\bf Step 2.} Given two (not necessarily disjoint)
 small open sets $U_1,\,U_2\subset {\mathbb P}^2(\C)$, 
 let $W_{U_1,U_2}$ be the set of triples $(\cC,\,\p,\,\q)$ 
consisting of a degree $n$ complex curve $\cC$ having a marked 
singular   point $\p\in U_1$ and a marked singular point $\q\in U_2$, with
$\p\ne\q$. It takes 
four parameters to specify $\p$ and $\q$. We can choose a projective 
transformation $T_{\p,\q}$ depending on these four parameters which carries $\p$
 to $(0:0:1)$ and $\q$ to  $(0:1:0)$.
The equation $\Phi\big(T_{\p,\q}(x,y,z)\big)=0$ will then uniquely describe the 
most general curve of degree $n$ with $\p$ and $\q$ as singular points. It 
follows that the dimension of  $W_{U_1,U_2}$ is 
$$4+\big(d(n)-6\big)~=~d(n)-2~.$$ 
Since  a curve can have at most finitely many critical points, it follows that 
the projection map from $W_{U_1,U_2}$ into ${\mathfrak C}_n(\C)$ 
 is finite-to-one; and hence must map to a $d(n)-2$ dimensional set 
$W'_{U_1,U_2}$. Now cover ${\mathbb P}^2(\C)\times{\mathbb P}^2(\C)$
by finitely many $U_1\times U_2$
and let $W_n$ be the union of the $W'_{U_1,U_2}$.
\msk

{\bf Step 3.} Since the Zariski closure  $\overline{W_n}$ is invariant under
 complex conjugation, it must be defined over the real numbers. 
 Therefore its intersection with ${\mathfrak C}_n(\R)$ has real codimension two
 in  ${\mathfrak C}^{cs}_n(\R)$. But every
real-smooth curve with a complex singularity must also have a complex
conjugate singularity, and hence must belong to the codimension two subset
 $\overline{W_n}\cap{\mathfrak C}_n(\R)$.\end{proof} 
\bigskip

\begin{definition}
  A circle smoothly embedded in  $\bP^2=\bP^2(\R)$ will be
  called an \textbf{\textit{oval}} if it
is two-sided, separating the plane into two components, and a
\textbf{\textit{non-oval}} if it is one-sided, not separating the plane.
 (Note than an oval in this sense need  not be convex.) 
\end{definition}
\smallskip

Equivalently, an embedded circle is an oval if and only if the associated
homology class in $H_1(\bP^2;~\Z/2)$ is zero. Every oval has a
neighborhood which is an annulus. Furthermore,
one of its two complementary components
must be a topological disk, while the other must be a M\"obius band.
On the other hand, every non-oval has a M\"obius band neighborhood, and an open
topological disk as complement. It follows 
from this that any two non-ovals must 
intersect each other, since it is impossible to embed a M\"obius band in a disk.
Note that a generic line intersects an oval in an even number of points,
and a non-oval in an odd number of points.

Now consider a curve $\cC\in \fC^{\cs}_n$.  The number of intersections
between $|\cC|_\R$ and a generic line in $\bP^2(\C)$ is always congruent to $n$
mod 2. (In fact the complexified line intersects $|\cC|_\C$ $n$ times; but
an even number of these intersection points belong to complex conjugate pairs.)
Therefore the discussion above implies 
that the real locus $|\cC|_\R$ is a union of ovals if
the dimension is even; but that it contains exactly one non-oval if $n$ is odd.
\msk

In order to distinguish between different configurations of topological circles,
it  is convenient to introduce the dual graph, which is a 
combinatorial description of the topological arrangement.
\smallskip

\begin{definition} First consider a a collection of $N$ disjoint ovals 
$O_1,\,\cdots,\,O_N$ in $\bP^2(\R)$, as in Figure~\ref{F-tree}.
The associated \textbf{\textit{dual graph}} $\Gamma$
is a rooted tree which has $N+1$
vertices, one vertex $v_k$ corresponding to
each connected component $U_k$ of the complementary region.
The root point corresponds to the unique complementary region $U_0$ 
which is non-orientable. Two vertices
are joined by an edge, which will be labeled $e_j$, if and only if the closures
of the corresponding regions intersect in the common boundary curve $O_j$.
(We should think of this dual graph is an abstract tree: It can be embedded in 
the plane for illustrative purposes, but the particular choice of embedding
is arbitrary.)

Now suppose that we are given a configuration consisting of
 $N-1$ ovals together with one non-oval $O_N$. The root point will now
correspond to the complementary region $U_N$ which surrounds $O_N$.
Since we cross from $U_N$ to itself as we cross $O_N$, it is natural to define 
the resulting dual graph to be the rooted tree as constructed above, 
but augmented by an extra
edge $e_N$ which is a loop, with both endpoints at the root point.
Compare Figure~\ref{F-odd}, which shows five ovals plus one  non-oval,
indicated schematically by a line segment, together with the corresponding
dual graph. (Of course, if we ignore the non-oval, then we get 
 a rooted tree in all cases.)
\end{definition}
\smallskip

\begin{figure}[h!]
\centerline{\includegraphics[width=4.5in]{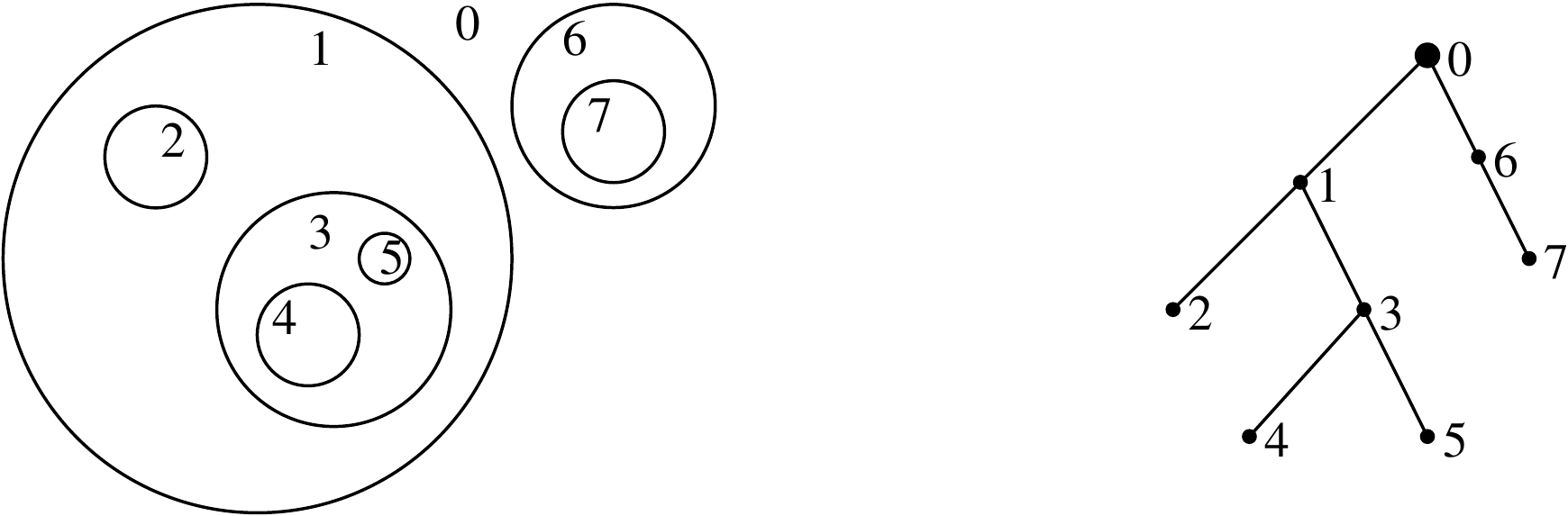}\vspace{1cm}}
\caption{\label{F-tree}\sf A collection of seven ovals in the plane,
and the associated dual graph. Each numbered vertex corresponds to the
associated numbered complementary region, and each edge corresponds to
the oval which separates two such regions.\bigskip}
\bigskip
\end{figure}
\bigskip

\begin{figure} [h!]
\centerline{\includegraphics[width=4in]{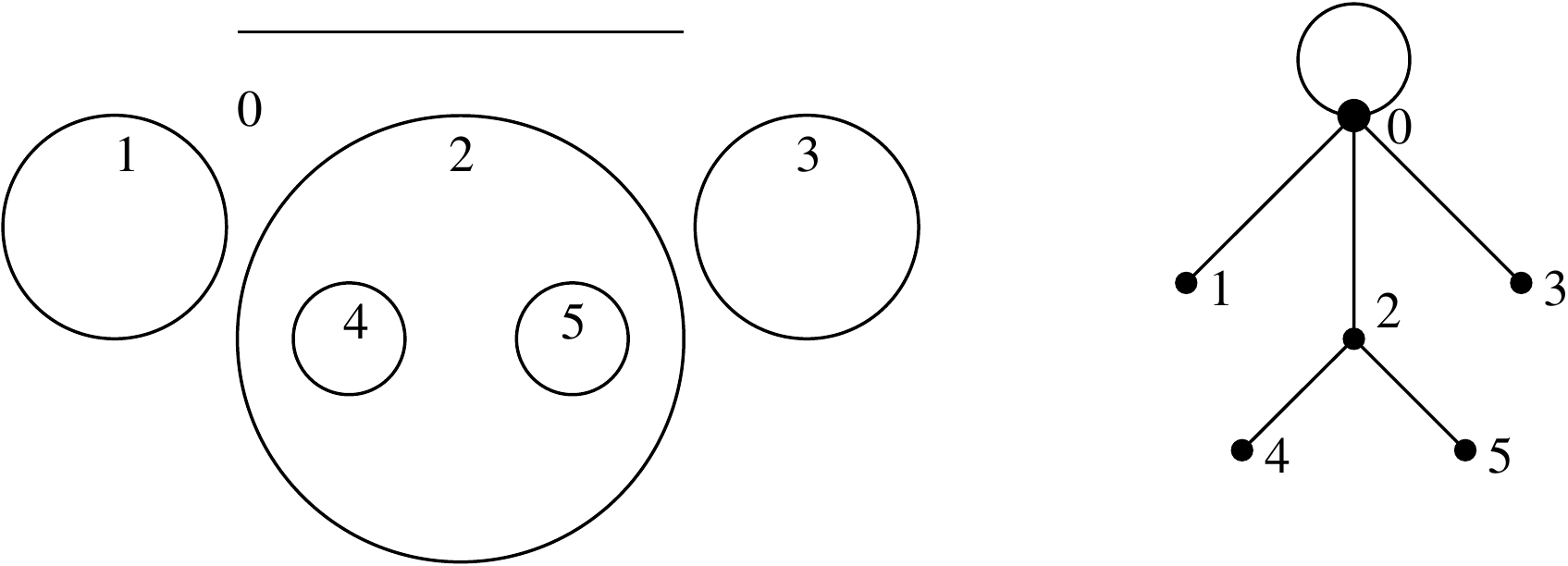}\vspace{1cm}}
\caption{\sf A similar figure with five ovals plus one non-oval.
\label{F-odd}\bigskip}
\bigskip
\end{figure}
\bigskip

It is not hard to check that one collection of embedded
topological circles in $\bP^2(\R)$ can be deformed continuously into another
if and only if they have isomorphic dual graphs, where the isomorphism
is required to preserve the root point.

However, we are interested in smooth algebraic curves in $\bP^2(\R)$.
If two curves of degree $n$ represent the same connected component in the
space $\fC^\rs_n$ or $\fC^\cs_n$ of smooth curves (or equivalently, the same 
component in the  moduli space $\M^\cs_n(\R)$ for smooth curves), then it 
follows  that they have isomorphic rooted graphs. However the converse
 statement is false.
 \ssk
 
One can learn much about a curve  $\cC\in\fC^\cs_n$ by thinking of its real
 locus $|\cC|_\R$ as a collection of  topological circles 
 embedded in the smooth Riemann surface $|\cC|_\C$. Following Klein, 
a curve is said to be of \hbox{\textbf{\textit{Type 1}}} if the Riemann
surface $|\cC|_\C$ is disconnected by $|\cC|_\R$, 
and of \textbf{\textit{Type 2}} if the difference set  $|\cC|_\C\ssm |\cC|_\R$
is connected. Rokhlin \cite{Ro} described an example of two 
connected components in the space $\fC^\cs_5$ such the the real loci $|\cC|_\R$
for curves in one component
can be deformed continuously to the real loci for curves in the other component,
even though one of these components has Type 1, while the other has Type 2.
For curves of degree six, Rokhlin and Nikulin showed 
that the space $\fC^\cs_6$ has 64  distinct connected components, although 
there are only 56 distinct real topological types.
(Compare \cite{KKPSS}.) 
\medskip

Perhaps Hilbert's Problem should be reformulated in more modern terms as
 follows:

\begin{quote}\sf
Is it possible to find an algorithm which, for any specified degree $n$ and
any rooted tree,  will decide whether or not
there is a curve in $\fC_n^\cs(\R)$ with the topological type corresponding
to this tree? More precisely, can it decide how many components in 
$\fC_n^\cs(\R)$ have this topological type; and can it decide when two
given curves belong to the same component? (Of course, to be really useful
such an algorithm would have to run in polynomial time.)
\end{quote}

It would also be interesting to find out what one can say about the topology 
of the various components of $\fC_n^\cs(\R)$. Perhaps, some components
have a complicated 
fundamental group? For even $n$ there is one component 
which is easy to  understand: It is not hard to see that
 the component consisting of curves $\cC$ with no real points, so that
 $|\cC|_\R$ is empty, is a convex subset of projective space.
\msk

\begin{rem}
One can also consider the  \textbf{\textit{moduli space}} 
$\M_n^\cs =\fC^\cs_n/\bG$
for real curves which are complex-smooth, where $\bG=\PGL_3(\R)$.
Since $\fC^\cs_n(\R)$ is by definition a subset of $\fC^\sm_n(\C)$, it follows
easily from Corollary \ref{C-dis} that the action of $\bG$ on 
$\fC^\cs_n(\R)$ is proper, and hence that the quotient space 
$\M^\cs_n(\R)$ is a Hausdorff orbifold. 
Since this group $\bG$ is connected,
it follows easily that there is a one-to-one correspondence between connected
components of $\fC_n^\cs$ and connected components of $\M_n^\cs$.
\end{rem}
\bigskip 

\setcounter{lem}{0}

\appendix
\section{Remarks on the Literature.}\label{A-A}

The moduli space $\M^\sm_n(\C)$
 for smooth curves of degree $n$ has long been studied. In
many cases it is known to be a rational variety. (Compare \cite{S-B}
 and \cite{BBK}.) For the problem of compactifying $\M^\sm_n(\C)$, compare 
\cite{Hac}. Curves in $\bP^2$ (and more generally in $\bP^n$) with an infinite
group of projective automorphisms were  studied by F. Klein and S. Lie 
in 1871  (see \cite{KL}).
 For the classification of such curves, see \cite[\bf AF2]{AF1}. 

The Algebraic
Geometer's Bible for studying moduli spaces is Mumford's\break
``Geometric Invariant Theory'' \cite{Mu}. For other expositions of this theory,
see for example \cite{New}, or \cite{Sim};  
and for the special case of 
$\PGL_{k+1}$ acting on hypersurfaces in $\bP^k$ see  \cite{Ne}, as well as
 \cite[Ch.4,\,\S2]{Mu}.
The theory takes a simpler form in the very special case where the
reductive Lie group
$\bG$ acts on a variety  $\bX\subset\bP^k$ \textbf{\textit{linearly}}, that 
is by an embedding into  $\PGL_{k+1}$ which lifts to an
embedding  into $\GL_{k+1}$. A point of $\xx\in \bX$ is then
 called \textbf{\textit{stable}} if the stabilizer
$\bG_\xx$ of $\xx$ is finite, and if the orbit of a representative point
$\widehat\xx\in\C^{k+1}\ssm\{\bf 0\}$ over $\xx$
is closed and bounded away from zero. If $\bX^{\sf s}$ is the open subset
consisting of stable points, then the quotient $\bX^{\sf s}/\bG$ is well 
behaved.
The extension of this definition to more general group actions depends
on a study of suitably linearized line bundles over $\bX$.
Particularly noteworthy are the Hilbert-Mumford numerical criterion for
stability \cite[Ch.\,2]{Mu}, and the related Kempf-Ness criterion \cite{KN}.

Alternative definitions can be provided in a somewhat simpler way 
by introducing symplectic structures (Compare  \cite{GRS},
or \cite{MS2}.)   Here is a brief outline: Suppose that $\bG$
is a complex reductive group with maximal compact subgroup $\bf K$ (for example
$\bG=\PGL_k(\C)$, with  ${{\bf K}=\rm PU}_k$), and suppose that $\bG$ acts
on a manifold $\bX$, which is provided with a ${\bf K}$-invariant symplectic 
structure.  In good cases, there is an associated \textbf{\textit{moment map}}
$$ {\mathfrak m}:\bX~\to~ \L^*,$$
where $\L^*={\rm Hom}_\R(\L,\,\R)$ is the dual vector space
to the Lie algebra $\L=\L(\bK)$, considered as a real vector space.
 (Compare \cite{GGK}.)
This map  ${\mathfrak m}$ has two important properties: The given 
action of $\bK$ on $\bX$ corresponds to the adjoint action of $\bK$ on $\L^*$.
Furthermore, for each vector $v\in\L$, if we think of the map 
$\xx\mapsto  {\mathfrak m}(\xx)(v)$
from $\bX$ to $\R$ as a Hamiltonian function, then
 the solution curves for the associated Hamiltonian differential
equation on $\bX$ are just the orbits $t\mapsto \exp(tv)(\xx)$ under the 
 one parameter subgroup $t\mapsto \exp(tv)$ of $\bK$ which is generated by $v$.
A point  $\xx\in\bX$ is called \textbf{\textit{stable}}, with respect to this 
choice of  ${\mathfrak m}$, if the stabilizer $\bG_\xx$ is finite, and if the 
intersection of the set  ${\mathfrak m}^{-1}({\bf 0})$ with the $\bG$ orbit of
$\xx$ is non-empty.

In the case of interest, with $\bG=\PGL_k(\C)$ with $k\ge 2$, there is a 
unique choice of the moment map  ${\mathfrak m}$, so that the open set
$\bX^{\sf s}$ 
consisting of stable points is also uniquely defined. Furthermore, the
quotient space  $\bX^s/\bG$ is a well defined orbifold with a Hausdorff
topology.\msk

In \S\ref{s-cc} and \S\ref{s-G} we describe open subsets of $\bX$ with
a Hausdorff orbifold quotient. Perhaps these are contained in
Mumford's set of stable points; but we don't have a proof.

\end{document}